\documentclass[12pt]{amsart}
\usepackage{float}
\usepackage{subfigure}
\usepackage{tikz}
\usepackage{tkz-euclide}
\usetikzlibrary{shapes.geometric}
\tikzset{
	rline/.style ={color = red, line width =1pt}
}
\tikzset{
	bline/.style ={color = blue, line width =1pt}
}
\tikzset{
	gline/.style ={color = green, line width =1pt}
}
\usepackage[colorlinks,
            linkcolor=black,
            anchorcolor=blue,
            citecolor=black]{hyperref}
\usepackage{cleveref}
\usepackage{graphicx}
\usepackage{amsmath}
\usepackage{amssymb}
\usepackage{bm}
\usepackage{setspace}
\usepackage{amsthm}
\usepackage{appendix}
\allowdisplaybreaks[4]
\usepackage{geometry}
\newtheorem{thm}{Theorem}[section]

\newtheorem{lem}[thm]{Lemma}
\newtheorem{prop}[thm]{Proposition}
\theoremstyle{definition}
\newtheorem{defn}[thm]{Definition}
\theoremstyle{remark}
\newtheorem{rem}[thm]{Remark}
\theoremstyle{conclusion}

\theoremstyle{conjecture}

\numberwithin{equation}{section}
\newcommand{\be}{\begin{equation}}
\newcommand{\ee}{\end{equation}}

\begin{document}
\title[Liouville theorems for $p$-Laplacian equations in convex cones]{Liouville theorems for $p$-Laplacian equations in convex cones without finite-energy condition}

\author{Lu Chen, Wei Dai, Changfeng Gui, Yunpeng Luo}

\address{Key laboratory of Algebraic Lie Theory and Analysis of Ministry of Education,
School of Mathematics and Statistics, Beijing Institute of Technology, Beijing 100081, P. R. China}
\email{chenlu5818804@163.com}

\address{School of Mathematical Sciences, Beihang University (BUAA), Beijing 100191, P. R. China, and Key Laboratory of Mathematics, Informatics and Behavioral Semantics, Ministry of Education, Beijing 100191, P. R. China}
\email{weidai@buaa.edu.cn}

\address{Department of Mathematics, University of Macau, Macau SAR, P. R. China}
\email{changfenggui@um.edu.mo}

\address{School of Mathematical Sciences, Beihang University (BUAA), Beijing 100191, P. R. China}
\email{yunpengluo@buaa.edu.cn}

\thanks{Lu Chen is supported by National Key Research and Development of China (No. 2022YFA1006900) and the National Natural Science Foundation of China (No. 12271027). Wei Dai is supported by the NNSF of China (No. 12222102 \& 12571113), the National Science and Technology Major Project (2022ZD0116401) and the Fundamental Research Funds for the Central Universities. Changfeng Gui is supported by NSFC Key Program (Grant No.12531010), University of Macau research grants CPG2024-00016-FST, CPG2025-00032-FST, CPG2026-00027-FST, SRG2023-00011-FST, MYRGGRG2023-00139-FST-UMDF, UMDF Professorial Fellowship of Mathematics, Macao SAR FDCT 0003/2023/RIA1 and Macao SAR FDCT 0024/2023/RIB1. Yunpeng Luo is supported by the NNSF of China (No. 12571113) and the Fundamental Research Funds for the Central Universities.}

\begin{abstract}
We study the anisotropic Finsler $p$-Laplacian equation
\begin{equation*}
      \left\{
          \begin{aligned}
          &-\Delta ^{H}_{p}u=f(u) \quad\,\,\, &{\rm{in}} \,\, \mathcal{C}, \\
          &{\bf{a}}(\nabla u)\cdot \nu =0 \quad\,\,\, &{\rm{on}} \,\, \partial\mathcal{C},
          \end{aligned}
          \right.
    \end{equation*}
where $N\geq3$, $1<p<N$, $\mathcal{C}\subseteq\mathbb{R}^{N}$ is an open convex cone and $\Delta ^{H}_{p}$ is the anisotropic Finsler $p$-Laplacian operator. If $f(u)$ is nonnegative and subcritical, we prove that every bounded nonnegative solution in $\mathcal{C}$ is identically zero. In particular, for $f(u)=u^{q}$ with $0<q<p^*-1$, we establish a pointwise decay estimate in $\mathcal{C}$ via the doubling argument and blowing-up method and prove that all nonnegative solutions must be zero without the boundedness assumption. Our results are the subcritical counterpart of the classification result for the critical case in \cite{CFR}, and extend the Liouville type theorems in $\mathbb{R}^{N}$ for the standard $p$-Laplacian in \cite{SZ} and for the anisotropic $p$-Laplacian in \cite{CFV, CHN} to general convex cones $\mathcal{C}$. In the critical case $f(u)=u^{p^*-1} $ and typical case $H(\xi)=|\xi|$, for $\frac{N+1}{3}<p<N$, we classified the positive solutions of the critical $p$-Laplacian equation in convex cones $\mathcal{C}$ without finite-energy assumption. This extends the classification result of \cite{Ou} in $\mathbb{R}^{N}$ to general convex cones $\mathcal{C}$, and removes the finite-energy assumption in \cite{CFR} in the typical case $H(\xi)=|\xi|$.
\end{abstract}
\maketitle {\small {\bf Keywords:}  Anisotropic Finsler $p$-Laplacian; convex cones; Liouville type theorems;  classification of solutions; critical $p$-Laplacian equation.  \\

}

\section{Introduction}

\subsection{Background and setting of the problem}
In this paper, we are concerned with the following anisotropic Finsler $p$-Laplacian equation with Neumann boundary condition:
    \begin{equation}\label{eq:1.1}
      \left\{
          \begin{aligned}
          &-\Delta ^{H}_{p}u=f(u) \quad\,\,\, &{\rm{in}} \,\, \mathcal{C}, \\
          &{\bf{a}}(\nabla u)\cdot \nu =0 \quad\,\,\, &{\rm{on}} \,\, \partial\mathcal{C},
          \end{aligned}
          \right.
    \end{equation}
where $N\geq3$,  $\mathcal{C}\subseteq\mathbb{R}^{N}$ is an open convex cone (with vertex set $\mathcal{V}\subseteq\partial\mathcal{C}$) including $\mathbb{R}^{N}$, the half space $\mathbb{R}^{N}_{+}$ and $\frac{1}{2^{m}}$-space $\mathbb{R}^{N}_{2^{-m}}:=\{x\in\mathbb{R}^{N}\mid x_{1},\cdots,x_{m}>0\}$ ($m=1,\cdots,N$). The function $H:\mathbb{R}^N\rightarrow\mathbb{R}$ is a Finsler norm, i.e., $H$ is convex, positively homogeneous of degree $1$ and positive on $\mathbb{S}^{N-1}:=\{\xi\in\mathbb{R}^{N}\mid\,|\xi|=1\}$. The anisotropic $p$-Laplacian $\Delta ^{H}_{p}$ is defined by
\begin{equation*}\label{NL}
  \Delta ^{H}_{p}u:=\operatorname{div}\left({\bf{a}}(\nabla u)\right)=\operatorname{div}\left(H^{p-1}(\nabla u)\nabla H(\nabla u)\right),
\end{equation*}
where
\begin{equation*}\label{a}
  {\bf{a}}(\xi):=H^{p-1}(\xi)\nabla H(\xi), \qquad \forall \,\, \xi\in\mathbb{R}^{N}.
\end{equation*}

\medskip

From now on, without loss of generalities, up to rotations and translations, we can write $\mathcal{C}=\mathbb{R}^{k}\times\mathcal{\widetilde{C}}$ with $k\in\{0,\cdots ,N\}$ and $\mathcal{\widetilde{C}}\subset\mathbb{R}^{N-k}$ is an open convex cone with vertex at the origin $0_{\mathbb{R}^{N-k}}$ which contains no lines, and hence the vertex set $\mathcal{V}=\mathbb{R}^{k}\times\{0_{\mathbb{R}^{N-k}}\}$ (see \cite{DPV} for more details). Note that, $\mathcal{V}=\emptyset$ if $k=N$, $\mathcal{V}=\partial\mathcal{C}$ is a hyperplane if $k=N-1$, and $\mathcal{H}^{N-1}(\mathcal{V})=0$ if $k\leq N-2$. In particular, $\mathcal{C}=\mathbb{R}^{N}$ if $k=N$, $\mathcal{C}=\mathbb{R}^{N}_{+}$ if $k=N-1$ and $\mathcal{\widetilde{C}}=\mathbb{R}_{+}$, and $\mathcal{C}=\mathbb{R}^{N}_{2^{-m}}$ if $k=N-m$ and $\mathcal{\widetilde{C}}=\mathbb{R}^{m}_{2^{-m}}:=\{x\in\mathbb{R}^{m}\mid x_{1},\cdots,x_{m}>0\}$. Note that if $\mathcal{C}=\mathbb{R}^{N}$, there is no Neumann boundary condition in \eqref{eq:1.1}.

\smallskip

Comparing with the regular Euclidean space $(\mathbb{R}^{N},|\cdot|)$, the anisotropic Finsler manifold $(\mathbb{R}^{N},H(\cdot))$ may lose the usual symmetry properties and the directional independence. When $H(\cdot)$ equals to the Euclidean norm $|\cdot|$, the anisotropic $p$-Laplacian $\Delta ^{|\cdot|}_{p}=\Delta_{p}$, where the regular $p$-Laplacian $\Delta_{p}$ is given by $\Delta_{p}u:=\operatorname{div}\left(|\nabla u|^{p-2}\nabla u\right)$ which arises as the first variation of the functional $\int|\nabla u|^p\mathrm{d}x$. Thus it is natural for us to consider the anisotropic $p$-Laplacian $\Delta ^{H}_{p}$ appearing in the Euler-Lagrange equations for Wulff-type functionals $\int H(\nabla u)^p\mathrm{d}x$. For literatures on anisotropic problems and their extensive applications, refer to e.g. \cite{ACF,BFK,BC,CRS,CFR,CL1,CL2,CFV,DGL,DP,DPV,FF,FM,Taylor,WX1,WX2,Wulff,XG,ZZ} and the references therein.

\medskip

Next we state the definition of weak solution for \eqref{eq:1.1}.
\begin{defn}
A function $u\in W^{1,p}_{loc}(\overline{\mathcal{C}})$ is said to be a weak solution to \eqref{eq:1.1} if
\begin{equation}\label{def weak solu}
 \int_{\mathcal{C}} H^{p-1}(\nabla u)\left\langle \nabla H(\nabla u),\nabla \varphi\right\rangle\mathrm{d}x= \int_{\mathcal{C}}f(u)\varphi\mathrm{d}x
\end{equation}
for all $\varphi\in W^{1,p}(\Omega)\cap L^{\infty}(\Omega)$ with $\Omega\subset\mathbb{R}^N$ bounded and $\varphi=0$ on $\partial\Omega\cap\mathcal{C}$.
\end{defn}
For $1<p<N$, let $p^*=\frac{Np}{N-p}$ be the critical exponent for the Sobolev embedding. A nonlinearity $f$ is called subcritical if there exists $1<\alpha<p^*$ such that, for all $u\geq0$,
\begin{equation}\label{subcritical}
f(u)\geq0 \quad \text{and} \quad (\alpha-1)f(u)-uf'(u)\geq0.
\end{equation}

When $f(u)=u^{q-1}$ for some $q>1$, $H(\xi)=|\xi|$ and $p=2$, equation \eqref{eq:1.1} reduces to the well-known Lane-Emden equation. In a celebrated paper \cite{GS}, Gidas and Spruck established the Liouville theorem for the Lane-Emden equation. It states that for $N>2$, $2\leq q<2^*=\frac{2N}{N-2}$, any nonnegative solution $u$ of
$$\Delta u+u^{q-1}=0\qquad \text{in}\,\,\mathbb{R}^N$$
must be trivial. However, this result fails in the critical and supercritical cases $q\geq\frac{2N}{N-2}$. Indeed, at the critical exponent $q=\frac{2N}{N-2}$, the Emden solution
$$u(x)=\left(\frac{C\lambda}{\lambda^2+|x|^2}\right)^{\frac{N-2}{2}}$$
with $\lambda>0$ and $C=\sqrt{N(N-2)}$, provides an explicit counterexample. For related work on the Lane-Emden equation, refer to e.g. \cite{BV,CGS1,PS}.

In another seminal paper, Serrin and Zou \cite{SZ} are concerned with the following nonhomogeneous degenerate elliptic equation
\begin{equation}\label{eq1.8}
 \Delta_p u+f(u)=0,\quad\,\,x\in\Omega,
\end{equation}
where $N\geq2$, $1<p<N$ and $\Omega$ is a domain in $\mathbb{R}^N$. They showed that, if $\Omega=\mathbb{R}^N$ and $f$ is subcritical, then every bounded nonnegative weak solution of \eqref{eq1.8} is trivial. Recently, for the special case $f(u)=u^q$ with $1\leq q<p^*-1$, Roncoroni \cite{R} proved that any nonnegative solution (bounded or not) to \eqref{eq1.8} must be zero. Moreover, Roncoroni raised the open question that whether this result remains valid in the anisotropic setting. Subsequently, Cheng, Huang and Niu \cite{CHN} provided an affirmative answer to this problem by extending the integral inequalities from \cite{SZ} to the anisotropic case (see also \cite{LWY}).

While extensive research has been devoted to these equations on the whole space $\mathbb{R}^N$ or the half-space $\mathbb{R}^N_+$, the Liouville-type theorem for the subcritical $p$-Laplacian equation on a general convex cone $\mathcal{C}$ remains largely unexplored, where $\mathcal{C}$ includes, for instance, the whole space $\mathbb{R}^N$, a half-space $\mathbb{R}^N_+$, or the $\frac{1}{2^m}$-space $\mathbb{R}^N_{2^{-m}}:= \{x\in\mathbb{R}^N \mid x_1,\dots,x_m>0\}$ for $m=1,\dots,N$. This gap in the literature naturally motivates the following two questions:

\smallskip

\textit{$(\rm{1})$ If $f$ is nonnegative and subcritical, are all bounded nonnegative weak solutions of \eqref{eq:1.1} on a general convex cone $\mathcal{C}$ trivial?}

\textit{$(\rm{2})$ If $f(u)=u^q$ with $0< q<p^*-1$, are all nonnegative solutions (including unbounded ones) to \eqref{eq:1.1} on a general convex cone $\mathcal{C}$ trivial? }

\smallskip
The first goal of this paper is to extend the Liouville type theorems in $\mathbb{R}^{N}$ for the standard $p$-Laplacian in Proposition 6.1 of \cite{SZ} and for the anisotropic $p$-Laplacian in Proposition 3.1 of \cite{CHN} to general convex cones $\mathcal{C}$, thereby affirmatively answering the above two questions, as stated in Theorems \ref{Liou} and \ref{liou2}. For optimal Liouville type theorems on general bounded or unbounded domains (including all cone-like domains), see \cite{DQ0,DQ}.

\smallskip
Turning to the critical case, when $f(u)=u^{p^*-1}$ and $\Omega=\mathbb{R}^N$, \eqref{eq1.8} becomes the standard critical $p$-Laplacian equation
\begin{equation}\label{eq1.9}
-\Delta_p u=u^{p^*-1}, \qquad x\in\mathbb{R}^N.
\end{equation}
This equation is intimately connected to the study of the extremals for the Sobolev inequality. In the special case $p=2$, it is also related to the well known Yamabe problem. For $p=2$, under the decay assumption
$$u(x)=O\left(\frac{1}{|x|^{N-2}}\right)\qquad\text{as}\quad|x|\rightarrow\infty, $$
Gidas, Nirenberg and Ni \cite{GNN} employed the method of moving planes to provide a complete classification of all positive solutions to the critical Laplacian equation
$$-\Delta u=u^{\frac{N+2}{N-2}}\qquad\text{in}\quad\mathbb{R}^N,\,\,\,N\geq3.$$
Subsequently, Caffarelli, Gidas and Spruck \cite{CGS2}  removed the decay assumption via the Kelvin transform and established the same result. For the general case $1<p<N$, \eqref{eq1.9} admits a two-parameter family of solutions as follows

\begin{equation}\label{eq1.10}
  \mathcal{U}_{\lambda,x_0}(x)=\left[\frac{\lambda^{\frac{1}{p-1}}N^{\frac{1}{p}}(\frac{N-p}{p-1})^{\frac{p-1}{p}}}{\lambda^{\frac{p}{p-1}}+|x-x_0|^{\frac{p}{p-1}}}\right]^{\frac{N-p}{p}},
\end{equation}
where $\lambda>0$ and $x_0\in\mathbb{R}^N$. While for $p\neq2$, the Kelvin transform is no longer applicable for equation \eqref{eq1.9}. This renders the classification of positive solutions to \eqref{eq1.9} considerably more challenging. Under the finite energy assumption
\begin{equation}\label{finite energy}
  \int_{\mathbb{R}^N}|\nabla u|^p\mathrm{d}x<+\infty,
\end{equation}
Damascelli, Merch\'an, Montoro and Sciunzi \cite{DMMS} proved that for $\frac{2N+2}{N+2}<p<2$, all positive weak solutions of \eqref{eq1.9} are the form \eqref{eq1.10}. Subsequently, by exploiting the asymptotic behavior at infinity of solution $u$ and carrying out the method of moving planes, V\'etois \cite{Ve} and Sciunzi \cite{Sc} significantly improved the result  to the range of $1<p<2$ and $2<p<N$, respectively. More recently, by exploiting the sharp asymptotic behavior at infinity of solution $u$ and $\nabla u$, Ciraolo, Figalli and Roncoroni \cite{CFR} extended these results for all $p\in(1,N)$ for the following critical anisotropic $p$-Laplacian equation in a general convex cone
 \begin{equation*}
      \left\{
          \begin{aligned}
          &-\Delta ^{H}_{p}u=u^{p^*-1} \quad\,\,\, &{\rm{in}} \,\, \mathcal{C}, \\
          &{\bf{a}}(\nabla u)\cdot \nu =0 \quad\,\,\, &{\rm{on}} \,\, \partial\mathcal{C}.
          \end{aligned}
          \right.
    \end{equation*}
Being different from \cite{DMMS,Ve,Sc}, they used a new approach based on integral identities rather than moving planes, which is inspired by \cite{SZ}. One should note that, the classification results in \cite{CFR,DMMS,Sc,Ve} rely crucially on the finite energy condition.

\smallskip

A crucial advancement without the finite energy assumption \eqref{finite energy} was achieved by Catino, Monticelli and Roncoroni \cite{CMR}. They obtained the classification result for the positive solutions of \eqref{eq1.9} without the finite energy assumption \eqref{finite energy} in dimensions $N=2,3$ with $\frac{N}{2}<p<2$. This result was substantially extended by Ou \cite{Ou} to the significantly broader range $\frac{N+1}{3}<p<N$ for $N\geq2$. Although the approaches in \cite{Ou}, \cite{CMR}, and the earlier work \cite{CFR} all rely on the integral identities introduced by Serrin and Zou \cite{SZ}, the key ingredients in \cite{Ou} and \cite{CMR} diverge sharply from those in \cite{CFR} precisely due to the absence of the finite-energy assumption. Actually, due to the lackness of finite energy condition, the pointwise estimates of solution $u$ and $\nabla u$ established in \cite{CFR} are no longer available.  Consequently, the proofs in \cite{Ou} and \cite{CMR} depend primarily on integral decay estimates of the solutions.

\smallskip

In contrast to the well-developed theory on $\mathbb{R}^N$, the classification of positive solutions for the critical equation on general convex cones without the finite energy assumption remains an open problem. The second goal of this paper is to address this gap. We establish the classification results for the critical $p$-Laplacian equation on any convex cone $\mathcal{C}$ without any finite energy assumption, thereby extending the result of \cite{Ou} in $\mathbb{R}^N$ to a more general domain, see Theorem \ref{cla}. We also mention that the classification of anisotropic Sobolev equation without finite volume constraints on the whole space has been recently established by Chen-Wu-Yang-Yan in \cite{CWYY} employing the invariant tensors technique under some suitable assumption on anisotropic norm. For more classification, Liouville type results and other related results on $p$-Laplacian equations with $1<p\leq N$, c.f. \cite{CDL,CWYY,CHN,CFR,CL1,CL2,DDGL,DGHP,DGL,DLL,Da,DFSV,DMMS,DHYZ,DWY,LWY,Sc,Ve1,Ve,Zhou} and the references therein. For sphere covering inequality and its applications on classification results, please refer to \cite{GM}. For classification of solutions to semi-linear equations on Heisenberg group or CR manifolds via Jerison-Lee identities and invariant tensor techniques, see \cite{MO,MOW} and the references therein.

\subsection{Main results}
We begin by establishing the second-order regularity of weak solutions up to the boundary $\partial\mathcal{C}$. Subsequently, via an approximation approach, we extend the integral identities of Propositions 6.1 and 7.1 in \cite{SZ} (corresponding to the cases $p\geq2$ and $1<p<2$, respectively) to the anisotropic setting in general convex cones $\mathcal{C}$ (Proposition \ref{inte ineq}). On this basis, we derive a series of integral inequalities in convex cones. These form the foundation for our first main result, a Liouville theorem for bounded nonnegative solutions, which is stated as the following theorem.

\begin{thm}\label{Liou}
Let $N\geq3$ and $\mathcal{C}\subseteq\mathbb{R}^{N}$ be an open convex cone with vertex set $\mathcal{V}$. Assume that $f\in C[0,+\infty)\cap C^{1}(0,+\infty)$ is subcritical with $f(t)>0$ for all $t>0$, and that $H\in C^{2}_{\text{loc}}(\mathbb{R}^{N}\setminus\{0\})$ is a Finsler norm such that $H^{2}$ is uniformly convex, i.e., there exist constants $0<\sigma_{1}\leq\sigma_{2}$ such that
\begin{equation}\label{uniformly convex}
\sigma_1 \mathrm{Id}\leq H(\xi)\nabla^2H(\xi)+\nabla H(\xi)\otimes \nabla H(\xi)\leq\sigma_2\mathrm{Id}, \qquad \forall \,\, \xi\in\mathbb{R}^N\setminus\{0\}.
\end{equation}
 Let $u(x)\in W^{1,p}_{loc}(\overline{\mathcal{C}})$ be a bounded nonnegative weak solution
 of \eqref{eq:1.1}. Then
 $$u(x)\equiv 0\qquad\text{in}\quad\mathcal{C}.$$
\end{thm}

For the special case $f(u)=u^q$, where $1\leq q<p^*-1$, combining Theorem \ref{Liou} with a pointwise decay estimate of solutions in convex cone $\mathcal{C}$ (Lemma \ref{decay}), we remove the boundedness assumption of solutions and establish the following Liouville type theorem for all nonnegative solutions.

\begin{thm}\label{liou2}
  Let $N\geq3$ and  $\mathcal{C}\subseteq\mathbb{R}^{N}$ be an open convex cone with vertex set $\mathcal{V}$. Assume that $f(u)=u^{q}$ with $0<q<p^*-1$, and  $H\in C^2_{loc}(\mathbb{R}^N\setminus \{0\})$ is a Finsler norm, such that $H^2$ is uniformly convex. Let $u(x)\in W^{1,p}_{loc}(\overline{\mathcal{C}})$ be a nonnegative weak solution
 of \eqref{eq:1.1}. Then
$$u(x)\equiv 0\qquad\text{in}\quad\mathcal{C}.$$
\end{thm}

\begin{rem}
Theorems \ref{Liou} and \ref{liou2} extend the Liouville type theorems in $\mathbb{R}^{N}$ for the standard $p$-Laplacian in \cite{SZ} and for the anisotropic $p$-Laplacian in \cite{CFV,CHN} to general convex cones $\mathcal{C}$, and can be regarded as the subcritical counterpart of the classification result for the critical case in \cite{CFR}. The exponent range in our Liouville Theorem \ref{Liou} is more wider than that in \cite{CHN}, since we require $1<\alpha<p^*$ while \cite{CHN} require $p<\alpha<p^*$ in the definition \eqref{subcritical} of ``subcritical".
\end{rem}

In the critical case $f(u)=u^{p^*-1}$, thanks to Proposition \ref{prop4.1}, we generalize Proposition 2.3 of \cite{Ou} to the setting of a general convex cone $\mathcal{C}$. Then along with Lemma \ref{lem5.5} and the corresponding integral estimates of solutions in convex cone $\mathcal{C}$, we establish the following classification theorem.

\begin{thm}\label{cla}
  Let $N\geq3$ and $\mathcal{C}$ be a convex cone in $\mathbb{R}^{N}$ with vertex set $\mathcal{V}$ and write $\mathcal{C}=\mathbb{R}^{k}\times\mathcal{\widetilde{C}}$, where $k\in\{0,\cdots ,N\}$ and $\mathcal{\widetilde{C}}\subset\mathbb{R}^{N-k}$ is an open convex cone with vertex at the origin $0_{\mathbb{R}^{N-k}}$ which does not contain a line.
  If $H(\xi)=|\xi|$ and $f(u)=u^{p^*-1}$, and assume further that $\frac{N+1}{3}<p<N$. Let $u$ be a positive weak solution of \eqref{eq:1.1}. Then $u=U_{\lambda,x_0}$, where
  $$U_{\lambda,x_0}=\left(\frac{\lambda^{\frac{1}{p-1}}N^{\frac{1}{p}}(\frac{N-p}{p-1})^{\frac{p-1}{p}}}{\lambda^{\frac{p}{p-1}}+|x-x_0|^{\frac{p}{p-1}}}\right)^{\frac{N-p}{p}} \qquad with\,\, \lambda>0\,\,and \,\,x_0\in \mathcal{V}.$$
  More precisely, \\
  $(\rm{i})$ if $k=N$, then $\mathcal{C}=\mathbb{R}^N$ and $x_0$ may be a generic point in $\mathbb{R}^N$;\\
  $(\rm{ii})$ if  $k\in\{1,\cdots ,N-1\}$, then $x_0\in\mathbb{R}^k\times\{0_{\mathbb{R}^{N-k}}\}$;\\
  $(\rm{iii})$ if $k=0$, then $x_0=0$.
\end{thm}

\begin{rem}
Theorem \ref{cla} extends the classification result of \cite{Ou} in $\mathbb{R}^{N}$ to general convex cones $\mathcal{C}$, and removes the finite-energy assumption in \cite{CFR} in the typical case $H(\xi)=|\xi|$.
\end{rem}

\smallskip

We would like to mention some key difficulties and ingredients in the proof of our main results, comparing with the whole space case $\mathbb{R}^{N}$ and the non-anisotropic case $H(\xi)=|\xi|$.

\smallskip

First, the continuation methods such as the method of moving planes (c.f. e.g. \cite{CY,CK,CL,CW,DLL,DLS,DQ1,DFSV,DMMS,Lin,Sc,S,S2,T,Ve,WX}) no longer work for the anisotropic $p$-Laplacian equation and convex cone $\mathcal{C}$. We will use the vector field and integration method to prove our main Theorems, which relies on a series of integral estimates in convex cone.

\smallskip

Then, in order to deal with problems in cones, the second order regularity up to the boundary should be taken into account carefully. A primary difficulty arises from the degeneracy of the equation and the non-smoothness of $\partial\mathcal{C}$. To handle this, inspired by \cite{CFR} (see also \cite{CL1}, \cite{DGL} and \cite{DGHP}), we establish the second-order boundary regularity $\mathbf{a}(\nabla u)\in W^{1,2}_{\text{loc}}(\overline{\mathcal{C}})$ for weak solution $u$ via an approximation method in convex cones, see Proposition \ref{regular} in Section 3. Then, via a careful analysis of the boundary integral on $\partial\mathcal{C}$, by using both Proposition \ref{regular} and the convexity of $\mathcal{C}$, we extend Propositions 6.1 and 7.1 of \cite{SZ} (corresponding to $p \geq 2$ and $1 < p < 2$, respectively; see also Proposition 2.2 in \cite{CMR} and Proposition 3.1 in \cite{CHN}) to convex cones, see Proposition \ref{inte ineq} in Section 3. In fact, by replacing the original boundary integral term with an approximating solution sequence and utilizing the regularity estimates (3.13) and (3.14), along with the weak lower semi-continuity of the convex functional $J$, we derived the convergence of the approximating sequence and hence overcame the difficulties of non-convergence and sign uncertainty of the original boundary integral term. The integral inequality established in Proposition \ref{inte ineq} serves as a anisotropic counterpart to Propositions 6.1 and 7.1 in \cite{SZ} for $p\geq2$ and $1<p<2$ respectively, and also extends the integral identity (3.9) in $\mathbb{R}^N$ of \cite{CHN} to general convex cone $\mathcal{C}$. This integral inequality plays a crucial role in the proof of Liouville type theorems for subcritical anisotropic $p$-Laplacian equation in Section 4 and classification of positive solutions for the critical $p$-Laplacian equation in Section 5.

\smallskip

Next, thanks to Proposition \ref{inte ineq} and a series of integral estimates in convex cones, we complete the proof of Theorem \ref{Liou}, the Liouville type theorem for bounded nonnegative solutions in convex cones. Then, for the case $f(u)=u^q$ with $p-1<q<p^*-1$, by Theorem \ref{Liou}, together with a doubling and blow up method (see also \cite{PQS} for more details and applications), we establish a pointwise decay estimate of solutions in convex cones, see Lemma \ref{decay} in Section 4. Thanks to this, we remove the boundedness assumption of solutions in the case $f(u)=u^q$ with $0<q<p^*-1$, see the proof of Theorem \ref{liou2}.

\smallskip

Finally, Proposition \ref{inte ineq} allows us to extend Proposition 2.3 of \cite{Ou} to convex cones, see Proposition \ref{pro2.4} in Section 5. Based on Proposition \ref{inte ineq}, following the proof strategy of Theorem 1.1 in \cite{Ou}, we can deduce the classification result in Theorem \ref{cla} and hence remove the finite-energy assumption in \cite{CFR} in the typical case $H(\xi)=|\xi|$.

\medskip

The paper is organized as follows. Section 2 collects preliminary results, including the basic properties of the Finsler norm $H$, the comparison principles for anisotropic Finsler $p$-Laplacian, the lower decay estimate of solutions in convex cones and the weak Harnack inequality in convex cones with Neumann boundary condition.
Section 3 is devoted to proving the second-order regularity up to the boundary $\partial{\mathcal{C}}$ for weak solutions and the key integral inequality in convex cones. In section 4, we conclude the Liouville type theorems in convex cones. Section 5 is devoted to proving the classification result of positive solutions for the critical $p$-Laplacian equation in convex cone.

\medskip

In the following, $C$ denotes a general positive constant, depending on $N$, $\mathcal{C}$, $H$, and $u$, whose value may vary from line to line. Moreover, $B_r(x)$ denotes the open Euclidean ball in $\mathbb{R}^N$ of radius $r$ centered at $x$, and $B_r^{\widehat{H}_0}(x):=\{y\in\mathbb{R}^N\mid \widehat{H}_0(y-x)<r\}$ denotes the anisotropic open ball in $\mathbb{R}^N$ centered at $x$ with radius $r$.

\section{Preliminaries}

%We have the following lemmas.
%\begin{lem}\label{L infty}
%Let $u$ be a positive solution of \eqref{eq:1.1}, if $f\in C[0,+\infty)\cap C^1(0,+\infty)$, $f(t)>0,\,\,\forall t>0$, and assume that there exists $1<\alpha\leq p^*$ such that
%$$ f(u)\geq0,\,\,(\alpha-1)f(u)-uf^{\prime}(u)\geq0,\,\,\text{for}\,\,u\geq0.$$
%Then $u\in L^{\infty}_{loc}(\overline{\mathcal{C}})$.
%\end{lem}
%\begin{proof}
%  Similar to the proof of Lemma 2.1 in \cite{CFR}.
%\end{proof}

In this subsection, we assume the Finlser norm $H$ satisfies that $H^2$ is uniformly convex.
We will use the notation $H_0$ to denote the dual norm associated to $H$, defined by
\begin{equation}\label{def:2.1}
    H_0(x):=\sup\limits_{H(\xi)=1}\left\langle x,\xi\right\rangle, \qquad \forall \,\, x\in\mathbb{R}^N.
\end{equation}
In addition, we set
\begin{equation*}\label{def:2.2}
    \widehat{H}_0(x):=H_0(-x), \qquad \forall \,\, x\in\mathbb{R}^N.
\end{equation*}
Obviously $H_0$ and $\widehat{H}_0$ are positively homogeneous of degree $1$, i.e., $H_0(\lambda x)=\lambda H_0(x)$ and $\widehat{H}_0(\lambda x)=\lambda\widehat{H}_0(x)$ for any $\lambda>0$. One can easily verify that
\begin{equation}\label{norm}
  c_{H}|\xi|\leq H(\xi)\leq C_{H}|\xi|, \qquad \forall \,\, \xi\in\mathbb{R}^{N},
\end{equation}
where $C_{H}:=\max\limits_{\xi\in \mathbb{S}^{N-1}}H(\xi)\geq c_{H}:=\min\limits_{\xi\in \mathbb{S}^{N-1}}H(\xi)>0$. Moreover, it follows from \eqref{norm} that, for any $x\neq0$,
\begin{equation*}
        H_0(x)=\sup\limits_{\xi\neq 0}\frac{\left\langle x,\xi\right\rangle }{H(\xi)}\geq \frac{\left\langle x,x\right\rangle }{H(x)}\geq \frac{1}{C_{H}}|x|,
\end{equation*}
then we can deduce further that
\begin{equation}\label{norm0}
        \frac{1}{C_{H}}|x|\leq H_0(x), \, \widehat{H}_0(x)\leq |x|\sup\limits_{H(\xi)=1}|\xi|\leq \frac{1}{c_{H}}|x|, \qquad \forall \,\, x\in\mathbb{R}^{N}.
\end{equation}
Note that if $H\in C^2_{loc}(\mathbb{R}^N\setminus \{0\})$ is positively homogeneous of degree $1$, we have
\begin{equation}\label{homogeneity}
 \left\langle \nabla_\xi H(\xi),\xi\right\rangle=H(\xi)\qquad\text{and}\qquad\nabla^2_{\xi}H(\xi)\xi=0,\qquad\forall\,\,\xi\in\mathbb{R}^N.
\end{equation}

\medskip

Under the assumption of \eqref{uniformly convex}, one can easily prove that there exist positive constants $\lambda_1$, $\lambda_2$ such that
\begin{equation}\label{s}
  \sum\limits_{i,j=1}^{N}\left\lvert \frac{\partial {\bf{a}}_i}{\partial \xi_j}(\xi)\right\rvert \leq \lambda_1|\xi|^{p-2},\qquad\forall\,\,\xi\in\mathbb{R}^N\backslash\{0\},
\end{equation}
and
\begin{equation}\label{l}
  \sum\limits_{i,j=1}^{N} \frac{\partial {\bf{a}}_i}{\partial \xi_j}(\xi)\eta_i\eta_j \geq \lambda_2|\xi|^{p-2}|\eta|^2,\qquad\forall\,\,\xi\in\mathbb{R}^N\backslash\{0\},\,\,\eta\in\mathbb{R}^N,
\end{equation}
where the notation ${\bf{a}}_i$ denotes the $i$-th component of vector ${\bf{a}}$. We refer to \cite{CHN} for the proofs of \eqref{s} and \eqref{l}. By \eqref{s}, \eqref{l} and Lemma 2.1 of \cite{Da}, there exist positive constants $c_1$, $c_2$, depending on $H$ and $p$, such that for all  $\xi_1,\,\xi_2\in\mathbb{R}^N$ with $|\xi_1|+|\xi_2|>0$, there hold
\begin{equation}\label{vector inequalities}
 \begin{aligned}
  \left\langle {\bf{a}}(\xi_1)-{\bf{a}}(\xi_2),\xi_1-\xi_2\right\rangle&\geq c_1(|\xi_1|+|\xi_2|)^{p-2}|\xi_1-\xi_2|^2,\\
  \left\lvert {\bf{a}}(\xi_1)-{\bf{a}}(\xi_2)\right\rvert &\leq c_2(|\xi_1|+|\xi_2|)^{p-2}|\xi_1-\xi_2|,\\
  \left\lvert {\bf{a}}(\xi_1)-{\bf{a}}(\xi_2)\right\rvert &\leq c_2|\xi_1-\xi_2|^{p-1}\quad(1<p\leq2).
 \end{aligned}
\end{equation}
\medskip
We need the following basic properties of the gauge $H$, which containing the convexity of $H$, the smoothness of $H_0$ and important connections between $H$ and $H_0$.
\begin{lem}\cite[Lemma 2.2]{CL1}\label{8,le:2.2}
If $H\in C^2(\mathbb{R}^N\setminus\{0\})$ and the Hessian of $H^{N}$ is positive definite in $\mathbb{R}^N\setminus\{0\}$, then $H$ is convex and $H_{0}\in C^2(\mathbb{R}^N\setminus\{0\})$. Moreover, for $x,\xi\in\mathbb{R}^N\setminus\{0\}$,
\begin{equation*}
H\left(\nabla H_0(x)\right)=H_0(\nabla H(\xi))=1,
\end{equation*}
and
\begin{equation*}\label{eq:2.2}
x=H_0(x) \nabla H\left(\nabla H_0(x)\right), \quad \xi=H(\xi) \nabla H_0(\nabla H(\xi)),
\end{equation*}
where $H_0$ is the dual function of $H$ defined by \eqref{def:2.1}.
\end{lem}

We also need the following two useful comparison principles for the anisotropic Finsler $p$-Laplacian: one in domain $\Omega$ without boundary condition, and another in convex cones with Neumann boundary condition.
\begin{lem}\label{weak comparison1}
Assume that $u$ and $v$ are continuous functions in the Sobolev space $W^{1,p}_{loc}(\Omega)$ which satisfy the distribution inequality
 $$\Delta ^{H}_{p}u-\Delta ^{H}_{p}v \leq 0$$
in a domain $\Omega\subset\mathbb{R}^N$. Suppose that $u\geq v$ on $\partial\Omega$ in the sense that the set $\{u-v+\varepsilon\leq0\}$ has compact support in $\Omega$ for every $\varepsilon>0$. Then $u\geq v$ in $\Omega$.
\end{lem}
\begin{proof}
Lemma \ref{weak comparison1} can be proved in similar way as Lemma 2.2 for the typical case $H(\xi)=|\xi|$ in \cite{SZ}, so we omit the details here.
\end{proof}

\begin{lem}\label{weak comparison}
Let $N\geq2$, $\mathcal{C}\subseteq\mathbb{R}^{N}$ be an open convex cone. Let $E\subset\mathbb{R}^{N}$ be a (bounded or unbounded) domain such that $\mathcal{C}\cap E$ is connected and $\mathcal{H}^{N-1}(\Gamma_0)>0$, where $\Gamma_0:=\mathcal{C}\cap \partial E$. Assume that $u$, $v\in W^{1,p}_{loc}(\overline{\mathcal{C}})$ satisfy the distribution inequality with Neumann boundary condition
 \begin{equation*}
     \left\{
         \begin{aligned}
             &\Delta ^{H}_{p}u-\Delta ^{H}_{p}v \leq 0\ &\rm{in}\,\, \mathcal{C}, \\
             &\langle {\bf{a}}(\nabla u),\nu\rangle=\langle {\bf{a}}(\nabla v),\nu\rangle=0 \quad\,\, &\rm{on}\,\, \partial\mathcal{C},
             \end{aligned}
             \right.
    \end{equation*}
in the sense of
$$\int_{\mathcal{C}}\left\langle {\bf{a}}(\nabla v)-{\bf{a}}(\nabla u),\nabla \psi\right\rangle \mathrm{d}x\leq0$$
for all $0\leq\psi\in W^{1,p}(\Omega)$ with $\Omega\subset\mathbb{R}^N$ bounded and $\psi=0$ on $\partial\Omega\cap \mathcal{C}$. Suppose that $u\geq v$ on $\Gamma_0$ in the sense that the set $\{u-v+\varepsilon\leq0\}\cap E$ has compact support in $\overline{C}\cap E$ for every $\varepsilon>0$. Then $u\geq v$ in $\mathcal{C}\cap E$.
\end{lem}
\begin{proof}
This Lemma has been derived in Lemma 2.4 of \cite{CL} for bounded domain $E$. Here we will prove Lemma \ref{weak comparison} for general unbounded or bounded domain $E$, since we need to apply it to unbounded domains in our subsequent proof.

For any given $\varepsilon>0$, define
$$D_{\varepsilon}:=\left\{x\in\overline{C}\cap E|v(x)>u(x)+\varepsilon\right\}.$$
If $D_{\varepsilon}=\varnothing $, then the conclusion follows. If $D_{\varepsilon}\neq\varnothing $, let
$$\eta(x):=\max\{v(x)-u(x)-\varepsilon,0\}.$$
Since $D_{\varepsilon}$ is bounded, $\eta=0$ on $\partial D_{\varepsilon}\cap\mathcal{C}$ and $\eta\geq0$ in $\mathbb{R}^N$, choose $\eta$ as a test function, then it follows from  \eqref{vector inequalities} that
\begin{equation*}
\begin{aligned}
  c_1\int_{D_{\varepsilon}}(|\nabla u|+|\nabla v|)^{p-2}|\nabla v-\nabla u|^2
  &\leq\int_{D_{\varepsilon}}\left\langle {\bf{a}}(\nabla v)-{\bf{a}}(\nabla u),\nabla v-\nabla u\right\rangle \mathrm{d}x\\
  &=\int_{\mathcal{C}}\left\langle {\bf{a}}(\nabla v)-{\bf{a}}(\nabla u),\nabla \eta\right\rangle \mathrm{d}x=0.
  \end{aligned}
\end{equation*}
This gives that $\nabla u=\nabla v$ a.e. in $D_{\varepsilon}$, so that $v=u+C$ in every connected component of $D_{\varepsilon}$ for some constant $C$. Combining with $v=u+\varepsilon$ on $\partial D_{\varepsilon}$, we have $v=u+\varepsilon$ in $D_{\varepsilon}$. This completes the proof of Lemma \ref{weak comparison} by the arbitrariness of $\varepsilon>0$.
\end{proof}
As an application of Lemma \ref{weak comparison}, we have the following lower decay estimate of solution in convex cones.
\begin{lem}\label{lower decay}
Let $N\geq2$, $\mathcal{C}\subseteq\mathbb{R}^{N}$ be an open convex cone. Suppose that $u\in W^{1,p}_{loc}(\overline{\mathcal{C}})$ is a positive weak solution of the inequality with Neumann boundary condition
\begin{equation*}
      \left\{
          \begin{aligned}
          &-\operatorname{div}{\bf{a}}(\nabla u)\geq0 \quad\,\,\, &{\rm{in}} \,\, \mathcal{C}, \\
          &{\bf{a}}(\nabla u)\cdot \nu =0 \quad\,\,\, &{\rm{on}} \,\, \partial\mathcal{C}.
          \end{aligned}
          \right.
    \end{equation*}
Then for any $R>0$, there exists a constant $C=C(p,H,u,\mathcal{C},R)>0$ such that
\begin{equation}\label{lower estimate}
  u(x)\geq C \widehat{H}_0(x) ^{-\frac{N-p}{p-1}}
\end{equation}
for all $x\in\mathcal{C}\cap \overline{B^{\widehat{H}_0}_{2R}(0)}^c$.
\end{lem}

\begin{proof}
Let $R>0$, define
$$K:=R^{\frac{N-p}{p-1}}\min\limits_{x\in\mathcal{C}\cap \partial B^{\widehat{H}_0}_{2R}(0)}u(x)>0,$$
and
$$v(x)=K\widehat{H}_0(x)^{-\frac{N-p}{p-1}}.$$
Then by Lemma \ref{8,le:2.2}, $v$ is a fundamental solution of
\begin{equation*}
      \left\{
          \begin{aligned}
          &-\operatorname{div}{\bf{a}}(\nabla v)=0 \quad\,\,\, &{\rm{in}} \,\, \mathcal{C}, \\
          &{\bf{a}}(\nabla v)\cdot \nu =0 \quad\,\,\, &{\rm{on}} \,\, \partial\mathcal{C}.
          \end{aligned}
          \right.
    \end{equation*}
Since $u$ is positive and $\lim\limits_{x\in\mathcal{C},|x|\rightarrow\infty}v(x)=0$, the set $\{x\in\mathcal{C}|v(x)-u(x)-\varepsilon\geq 0\}$ is bounded, then by applying Lemma \ref{weak comparison} in the domain $E=\overline{B^{\widehat{H}_0}_{2R}(0)}^c$, we obtain the validity of \eqref{lower estimate}. This concludes the proof of Lemma \ref{lower decay}.
\end{proof}

The following integration by parts formula in convex cones will be essential throughout our subsequent proofs.
\begin{lem}\label{int part}
 Let $N\geq2$, $\mathcal{C}\subseteq\mathbb{R}^{N}$ be an open convex cone. Suppose that $a,b\in W^{1,2}_{loc}(\mathcal{C})$, where $a$ is scalar function and $b$ is vector field. Let $\phi\in C^{1}_c(\mathbb{R}^N)$ be any test function. Then
 $$\int_{\mathcal{C}}a(x)b(x)\cdot\nabla\phi(x)\mathrm{d}x=-\int_{\mathcal{C}}(\nabla a(x)\cdot b(x)+a(x)\operatorname{div}b(x))\phi\mathrm{d}x+\int_{\partial\mathcal{C}}a(x)\phi(x)b(x)\cdot\nu\mathrm{d}\mathcal{H}^{N-1},$$
 where $\nu$ is the unit outwards normal vector of $\partial\mathcal{C}$ at $x$.
\end{lem}

Next, we need the following weak Hanack inequality for anisotropic Finsler $p$-Laplacian in convex cones with Neumann boundary condition.
\begin{lem}\label{Hanack}
Let $N\geq2$, $\mathcal{C}\subseteq\mathbb{R}^{N}$ be an open convex cone. Assume that $u\in W^{1,p}_{loc}(\overline{\mathcal{C}})$ is a nonnegative solution to the following inequality
  \begin{equation*}
      \left\{
          \begin{aligned}
          &-\operatorname{div}{\bf{a}}(\nabla u)\geq0 \quad\,\,\, &{\rm{in}} \,\, \mathcal{C}, \\
          &{\bf{a}}(\nabla u)\cdot \nu =0 \quad\,\,\, &{\rm{on}} \,\, \partial\mathcal{C}.
          \end{aligned}
          \right.
    \end{equation*}
Then for any $\gamma\in(0,p_*-1)$ and $R>0$, there exists a constant $C=C(N,p,\gamma,\mathcal{C})>0$ such that
\begin{equation*}
  \min\limits_{x\in B_R\cap\mathcal{C}}u(x)\geq CR^{-\frac{p}{\gamma}}||u||_{L^{\gamma}(B_{2R}\cap\mathcal{C})}.
\end{equation*}
\end{lem}
\begin{proof}
Firstly, due to the boundary condition $\left\langle{\bf{a}}(\nabla u),\nu\right\rangle=0$ on $\partial\mathcal{C}$, by restricting the integration region to the convex cone $\mathcal{C}$, the formulae (1.16) to (1.22) in \cite{T2} are still valid in convex cone $\mathcal{C}$, up to slightly corresponding modifications.
Next, one should note that, in order to deal with the boundary condition $\left\langle{\bf{a}}(\nabla u),\nu\right\rangle=0$ on $\partial\mathcal{C}$, in the proof of Harnack inequality in \cite{T2}, the Sobolev embedding $W^{1,p}_{0}(\Omega)\hookrightarrow L^{\frac{Np}{N-p}}(\Omega)$ with $1\leq p<N$ in \cite[Lemma 1.1]{T2}
should be replaced by the radial Poincar\'{e} type inequality in Lemma A.3 of \cite{DGL}. Indeed, the radial Poincar\'{e} type inequality in  Lemma A.3 of \cite{DGL} implies that, for any $f\in W^{1,p}(\Omega)$ with $f=0$ on the (back) radial contact set $\Gamma^{+}_{P}\subset\partial\Omega$,
\begin{equation*}
  \|f\|_{L^{\frac{Np}{N-p}}(\Omega)}\leq S(\Omega)\left(\|\nabla f\|_{L^{p}(\Omega)}+\|f\|_{L^{p}(\Omega)}\right)\leq S(\Omega)\left(1+d_{r,P}(\Omega)\right)\|\nabla f\|_{L^{p}(\Omega)},
\end{equation*}
where $S(\Omega)$ is the Sobolev constant of the embedding $W^{1,p}(\Omega)\hookrightarrow L^{\frac{Np}{N-p}}(\Omega)$, and $d_{r,P}(\Omega)$ defined by (A.1) in \cite{DGL} denotes the radial width of $\Omega$ with respect to the radial center $P$. In particular, if $\Omega=\mathcal{C}\cap E$ with $E\subset\mathbb{R}^{N}$ be a bounded domain, then $P=0$ and $\Gamma^{+}_{0}\subset\partial E\cap\mathcal{C}$.
Therefore, the formula (1.23) in \cite{T2} is still valid in our setting. Then, by repeating the remaining process of the proof in \cite[Theorem 1.2]{T2}, we conclude the validity of Lemma \ref{Hanack}.
\end{proof}

The following existence result of positive radial solution for $p$-Laplacian equation in ball is useful to construct solutions to the anisotropic Finsler $p$-Laplacian equation in anisotropic ball.
\begin{lem}\label{lem2.8}\cite[Lemma 2.6]{SZ}
Let $N>1$, $m>1$. There exists $R_m>0$ such that the equation
\begin{equation*}\label{rs}
  \Delta_mv+v^{m-1}=0
\end{equation*}
has a positive radial solution $v_m(|x|)$ in the ball $|x|<R_m$, with $v_m=0$ on $|x|=R_m$ and $v_m(0)=1$.
\end{lem}
As an application of Lemmas \ref{lem2.8} and \ref{weak comparison1}, we have the following lower decay estimate of solutions in convex cones at infinity.
\begin{lem}\label{lem2.9}
 Let $N\geq2$, $\mathcal{C}\subseteq\mathbb{R}^{N}$ be an open convex cone. Let $g(s)>0$ for all $s>0$ with $\inf\limits_{s>s_0}g(s)>0$ for any $s_0>0$. Suppose that $u\in W^{1,p}_{loc}(\overline{\mathcal{C}})$ is a nonnegative solution to the following inequality
 \begin{equation}\label{eq2.6}
  \left\{
        \begin{aligned}
        &-\operatorname{div}\left({\bf{a}}(\nabla u)\right)\geq g(u) \quad\,\,\, &{\rm{in}} \,\, \mathcal{C}, \\
        &{\bf{a}}(\nabla u)\cdot \nu =0 \quad\,\,\, &{\rm{on}} \,\, \partial\mathcal{C}.
        \end{aligned}
        \right.
 \end{equation}
 Then $\liminf\limits_{x\in\mathcal{C},|x|\rightarrow+\infty}u(x)=0$.
\end{lem}
\begin{proof}
  Let $w(x):=\frac{1}{N^{\frac{1}{p-1}}}\frac{p-1}{p}\widehat{H}^{\frac{p}{p-1}}_0(x)$, then a direct calculation shows that
  \begin{equation*}
  \left\{
        \begin{aligned}
        &-\operatorname{div}\left({\bf{a}}(-\nabla w)\right)= 1 \quad\,\,\, &{\rm{in}} \,\, \mathcal{C}, \\
        &{\bf{a}}(-\nabla w)\cdot \nu =0 \quad\,\,\, &{\rm{on}} \,\, \partial\mathcal{C}.
        \end{aligned}
        \right.
 \end{equation*}
Next, let $\omega:=\mathcal{C}\cap\mathbb{S}^{N-1}$ and $\omega_0\subset\omega$ be any compact subset of $\omega$. Define
$$\mathcal{C}_0:=\{tx\,|\,x\in\omega_0,t\in(0,+\infty)\}.$$
We prove that $\liminf\limits_{x\in\mathcal{C}_0,|x|\rightarrow+\infty}u(x)=0$.
Suppose for contradiction that $\liminf\limits_{x\in\mathcal{C}_0,|x|\rightarrow+\infty}u(x)=\varepsilon_0>0$, and let $y^j\in\mathcal{C}_0$ with $|y^j|\rightarrow+\infty$ as $j\rightarrow+\infty$, such that
 $$\lim\limits_{j\rightarrow+\infty}u(y^j)=\varepsilon_0.$$
Define $\gamma:=\inf\limits_{s>\frac{\varepsilon_0}{2}}g(s)>0$, then by \eqref{eq2.6} and the conditions on $g$, there holds
$$-\operatorname{div}\left({\bf{a}}(\nabla u)\right)\geq g(u)\geq\gamma$$
for $x\in\mathcal{C}$ with $|x|$ large enough. Let
$$w_{\varepsilon_0}(x):=2\varepsilon_0-\gamma^{\frac{1}{p-1}} w(x-y^j),$$
then by the definition of $\mathcal{C}_0$, there exist $j_0\in\mathbb{N}$ and $R_{\varepsilon_0}>0$ such that $ B^{\widehat{H}_0}_{R_{\varepsilon_0}}(y^j)\subset\mathcal{C}$ holds for all $j\geq j_0$, and $w_{\varepsilon_0}$ satisfies
\begin{equation*}
  \left\{
        \begin{aligned}
        &-\operatorname{div}\left({\bf{a}}(\nabla w_{\varepsilon_0})\right)= \gamma\quad\,\,\, &{\rm{in}} \,\, B^{\widehat{H}_0}_{R_{\varepsilon_0}}(y^j), \\
        &w_{\varepsilon_0}>0\quad\,\,\, &{\rm{in}} \,\, B^{\widehat{H}_0}_{R_{\varepsilon_0}}(y^j),\\
        &w_{\varepsilon_0}=0\quad\,\,\, &{\rm{on}} \,\, \partial B^{\widehat{H}_0}_{R_{\varepsilon_0}}(y^j).
        \end{aligned}
        \right.
 \end{equation*}
 By Lemma \ref{weak comparison1} we conclude that $u\geq w_{\varepsilon_0}$ in $B^{\widehat{H}_0}_{R_{\varepsilon_0}}(y^j)$.
This yields that $u(y^j)\geq w_{\varepsilon_0}(y^j)=2\varepsilon_0$, which contradicts $\lim\limits_{j\rightarrow+\infty}u(y^j)=\varepsilon_0$. Thus we complete the proof of Lemma \ref{lem2.9}.
\end{proof}

The following Liouville theorem for differential inequality in convex cones with Neumann boundary condition plays a crucial role in the proof of Theorem \ref{liou2}.

\begin{lem}\label{lem2.10}
 Let $N\geq2$, $\mathcal{C}\subseteq\mathbb{R}^{N}$ be an open convex cone. Assume that $1<q\leq p$ and $u$ is a solution of
  \begin{equation*}
    \left\{
        \begin{aligned}
        &-\operatorname{div}\left({\bf{a}}(\nabla u)\right)\geq u^{q-1} \quad\,\,\, &{\rm{in}} \,\, \mathcal{C}, \\
        &u\geq0\quad\,\,\, &{\rm{in}} \,\, \mathcal{C},\\
        &{\bf{a}}(\nabla u)\cdot \nu =0 \quad\,\,\, &{\rm{on}} \,\, \partial\mathcal{C}.
        \end{aligned}
        \right.
  \end{equation*}
Then $u\equiv 0$ in $\mathcal{C}$.
\end{lem}

\begin{proof}
We suppose for contradiction that $u\not\equiv 0$. Then it follows from Lemma \ref{Hanack} that $u>0$ in $\mathcal{C}$. Let $\mathcal{U}_p$ be the radial solution given by Lemma \ref{lem2.8} with $m=p$ and $R_m=R_p$. Note that Lemma \ref{lem2.9} guarantees that there exists $y\in\mathcal{C}$ such that $u(y)\leq1$. By the proof of Lemma \ref{lem2.9}, we may assume that $|y|$ large enough such that $B^{\widehat{H}_0}_{R_p}(y)\subset\mathcal{C}$. Then by a direct calculation, $\widetilde{u}(x):=\mathcal{U}_p(\widehat{H}_0(x-y))$ satisfies the following equation:

\begin{equation}\label{eq2.4}
    \left\{
        \begin{aligned}
        &-\operatorname{div}\left({\bf{a}}(\nabla \widetilde{u})\right)=\widetilde{u}^{p-1} \quad\,\,\, &{\rm{in}} \,\, B^{\widehat{H}_0}_{R_p}(y), \\
        &\widetilde{u} =0 \quad\,\,\, &{\rm{on}} \,\, \partial B^{\widehat{H}_0}_{R_p}(y),\\
        &\widetilde{u}(y)=1.
        \end{aligned}
        \right.
  \end{equation}
Next, it follows from \eqref{eq2.4} that there exists some constant $c\in(0,1]$ such that
$$u\geq c\widetilde{u}>0\quad\text{in}\,\,B^{\widehat{H}_0}_{R_p}(y)\quad\text{and}\quad u=c\widetilde{u}\quad\text{at\,\,some\,\,point\,\,in}\,\, B^{\widehat{H}_0}_{R_p}(y).$$
Moreover, it is easy to see that, for any sufficiently small constant $\varepsilon>0$, there exists a non-empty domain $D_{\varepsilon}$ strictly contain in $B^{\widehat{H}_0}_{R_p}(y)$, such that
$$c\widetilde{u}>u-\varepsilon\quad\text{in}\quad D_{\varepsilon}$$
and
$$c\widetilde{u}\leq u-\varepsilon\quad\text{in}\quad \left(B^{\widehat{H}_0}_{R_p}(y)\right) \setminus D_{\varepsilon}.$$
Then it is clearly that
\begin{equation*}
  \begin{aligned}
  \operatorname{div}\left({\bf{a}}(\nabla (c\widetilde{u}))\right)-\operatorname{div}\left({\bf{a}}(\nabla (u-\varepsilon))\right)
  &=\operatorname{div}\left({\bf{a}}(\nabla (c\widetilde{u}))\right)-\operatorname{div}\left({\bf{a}}(\nabla u)\right)\\
  &\geq -(c\widetilde{u})^{p-1}+u^{q-1}\\
  &\geq -(c\widetilde{u})^{p-1}+u^{p-1}\geq0
  \end{aligned}
\end{equation*}
holds in $D_{\varepsilon}$, where we have used the fact that $1<q\leq p$, $c\widetilde{u}\leq u\leq1$. By Lemma \ref{weak comparison1} we deduce $c\widetilde{u}\leq u-\varepsilon$ in $D_{\varepsilon}$, a contradiction. This completes our proof of Lemma \ref{lem2.10}.

\end{proof}

The following Lemma on a property of the trace of matrices will also be needed.

\begin{lem}\cite[Lemma 4.5]{AKM}\label{le:2.9}
    Let the matrix $A$ be symmetric with positive eigenvalues, and let $\lambda_{min}$ and $\lambda_{max}$ be its smallest and largest eigenvalue, respectively;
     let $B$ be a symmetric matrix. Then we have
     \begin{equation*}
        \mathrm{trace}((AB)^2)\approx  \mathrm{trace}(AB(AB)^T),
     \end{equation*}
     where the involved constants depend only on the ratio $\frac{\lambda_{max}}{\lambda_{min}}$ and $N$. In particular, we have
     \begin{equation*}
        \mathrm{trace}(AB(AB)^T)\leq N\left(\frac{\lambda_{max}}{\lambda_{min}}\right) ^2  \mathrm{trace}((AB)^2).
     \end{equation*}
\end{lem}

\section{the second order regularity for weak solution and integral inequality in convex cones}
In this section, we will prove the second order regularity $\mathbf{a}(\nabla u)\in W^{1,2}_{loc}(\overline{\mathcal{C}})$ for weak solution $u$ via an approximation method and an integral inequality in convex cones, which is the key ingredient in the proof of Liouville theorem and the classification of results.

\smallskip

We have the following regularity result for ${\bf{a}}(\nabla u):=H^{p-1}(\nabla u)\nabla H(\nabla u)$.
\begin{prop}\label{regular}
Let $N\geq2$, $\mathcal{C}\subseteq\mathbb{R}^{N}$ be an open convex cone and $H\in C^{2}_{\text{loc}}(\mathbb{R}^{N}\setminus\{0\})$ be a Finsler norm such that $H^{2}$ is uniformly convex. Assume that $f\in C^1[0,+\infty)$ is a positive function. Let $u\in W^{1,p}_{loc}(\overline{\mathcal{C}})\cap L^{\infty}_{loc}(\overline{\mathcal{C}})$ be a positive solution of \eqref{eq:1.1}. Then
$${\bf{a}}(\nabla u)\in W^{1,2}_{loc}(\overline{\mathcal{C}}).$$
\end{prop}

\begin{proof}
   The estimate ${\bf{a}}(\nabla u)\in W^{1,2}_{loc}(\mathcal{C})$ can be obtained from \cite[Theorem 1.1]{ACF}. Our task is to show further that ${\bf{a}}(\nabla u)\in W^{1,2}_{loc}(\overline{\mathcal{C}})$. The main difficulties are the degeneracy of the equation and the non-smoothness of $\partial\mathcal{C}$. We will argue by the approximation method in \cite{AKM,CFR, CM}.

   Thanks to the proof of \cite[Theorem 1.3]{CRS}, there exists a sequence of convex cones $\{\mathcal{C}_k\}$ such that $\mathcal{C}_k\subset\mathcal{C}$, $\partial\mathcal{C}_k\setminus\{0\}$ is smooth and $\mathcal{C}_k$ approximate $\mathcal{C}$. For any given $R>1$ large, we define
 \begin{equation*}
    \mathcal{C}_{k,R}:=\mathcal{C}_k\cap B_R(0),\quad\, \Gamma_{k,0}^R:=\mathcal{C}_k\cap \partial B_R(0),\quad\, \Gamma_{k,1}^R:=\partial\mathcal{C}_k\cap  B_R(0).
 \end{equation*}
    For $k\in\mathbb{N}$ fixed, we let $u_k\in W^{1,p}(\mathcal{C}_{k,2R})$ be the weak solution to
    \begin{equation}\label{eq:3.1}
        \left\{
            \begin{aligned}
            & -\Delta ^{H}_{p}u_{k,R}=f(u)\quad \, &{\rm{in}}\ \mathcal{C}_{k,2R},\\
            &u_{k,R}=u \quad \, &{\rm{on}}\, \Gamma_{k,0}^{2R},\\
            &{\bf{a}}(\nabla u_{k,R})\cdot \nu =0\quad\, &{\rm{on}}\ \Gamma_{k,1}^{2R}.
            \end{aligned}
            \right.
    \end{equation}
    The existence of solutions to \eqref{eq:3.1} can be derived from the variational minimization problem $\min\limits_{v}\left\{J_{k,R}(v):=\int_{\mathcal{C}_{k,2R}}\left(\frac{H^{p}(\nabla(v+u))}{p}-f(u)v\right)\mathrm{d} x \mid v\in W^{1,p}(\mathcal{C}_{k,2R}), v=0 \,\, \text{on}\,\, \Gamma_{k,0}^{2R}\right\}$ and the Ekeland variational principle.
    Since $u$ is locally bounded, by the elliptic estimates in \cite{L,S,T}, we obtain
    that $\{u_{k,R}\}$ are uniformly bounded in $C^{1,\theta}_{loc}(\overline{\mathcal{C}_{k,2R}}\setminus((\partial\mathcal{C}_k\cap\partial B_{2R}(0))\cup\{0\}))\cap C^{0,\theta}_{loc}(\overline{\mathcal{C}_{k,2R}})$ with respect to $k$. Then
    by the Ascoli--Arzel$\acute{\text{a}} $'s Theorem and a diagonal process, we have
    \begin{equation}\label{lim}
      u_{k,R}\rightarrow v\,\,\, \text{in}\,\,\, C^{1}_{loc}(\mathcal{C}\cap B_{2R}(0)),\qquad \text{as}\ k\rightarrow+\infty.
    \end{equation}

\smallskip

In view of \eqref{eq:3.1}, we deduce that $v$ satisfies
 \begin{equation}\label{v}
     \left\{
         \begin{aligned}
         & -\operatorname{div}{\bf{a}}(\nabla v)=f(u)\quad \, &{\rm{in}}\,\, \mathcal{C}\cap B_{2R}(0),\\
         &v=u \quad \, &{\rm{on}}\, \mathcal{C}\cap \partial B_{2R}(0),\\
         &{\bf{a}}(\nabla v)\cdot \nu =0\quad\, &{\rm{on}}\,\, \partial\mathcal{C}\cap B_{2R}(0).
         \end{aligned}
         \right.
 \end{equation}
 Note that the solution of \eqref{v} is unique. Actually, assume that $u_1$, $v_1$ are two solutions of equation \eqref{v}, then we have
 \begin{equation*}
     \left\{
         \begin{aligned}
         & -\operatorname{div}{\bf{a}}(\nabla u_1)+\operatorname{div}{\bf{a}}(\nabla v_1)=0\quad \, &{\rm{in}}\,\, \mathcal{C}\cap B_{2R}(0),\\
         &u_1=v_1=u \quad \, &{\rm{on}}\, \mathcal{C}\cap \partial B_{2R}(0),\\
         &{\bf{a}}(\nabla u_1)\cdot \nu ={\bf{a}}(\nabla v_1)\cdot \nu=0\quad\, &{\rm{on}}\,\, \partial\mathcal{C}\cap B_{2R}(0).
         \end{aligned}
         \right.
 \end{equation*}
Choosing $u_1-v_1$ as a test function, we get
$$\int_{\mathcal{C}}\langle {\bf{a}}(\nabla u_1)-{\bf{a}}(\nabla v_1), \nabla u_1-\nabla v_1\rangle {\rm{d}}x-\int_{\partial\mathcal{C}} ({\bf{a}}(\nabla u_1)-{\bf{a}}(\nabla v_1))\cdot \nu(u_1-v_1){\rm{d}}\mathcal{H}^{N-1}=0.$$
Note that $({\bf{a}}(\nabla u_1)-{\bf{a}}(\nabla v_1))\cdot \nu=0$ on $\partial\mathcal{C}$, and it follows from \eqref{vector inequalities} that
\begin{equation*}
     c_1\int_{\mathcal{C}}(|\nabla u_1|+|\nabla v_1|)^{p-2}|\nabla(u_1-v_1)|^2{\rm {d}}x\leq\int_{\mathcal{C}}\langle {\bf{a}}(\nabla u_1)-{\bf{a}}(\nabla v_1), \nabla u_1-\nabla v_1\rangle {\rm{d}}x=0,
\end{equation*}
and hence $\nabla(u_1-v_1)=0$. Since convex cone $\mathcal{C}\cap B_{2R}(0)$ is connected, then $u_1=v_1+C_1$ for some constant $C_1$. Moreover, it follows from $u_1=v_1=u$ on $\mathcal{C}\cap \partial B_{2R}(0)$ that $C_1=0$, hence the solution of \eqref{v} is unique. Since both $u$ and $v$ are solutions of \eqref{v}, thus we have $v\equiv u$. Then from \eqref{lim}, we infer that, for any $R>1$ large,
\begin{equation}\label{lim2}
    u_{k,R}\rightarrow u\,\,\, \text{in}\,\,\, C^{1}_{loc}(\mathcal{C}\cap B_{2R}(0)),\qquad \text{as}\ k\rightarrow+\infty.
\end{equation}

\smallskip

Let $\{\phi _l\}$ with $l\in\mathbb{N}$ be a family of radially symmetric smooth mollifiers and define
\begin{equation*}
 {\bf{a}}^l(z):=({\bf{a}}\ast \phi_l)(z)\qquad \text{for}\,\, z\in\mathbb{R}^N,
\end{equation*}
where ${\bf{a}}\ast\phi_l$ stands for the convolution. By the properties of convolution and continuity of ${\bf{a}}(\cdot)$, we have
${\bf{a}}^l\rightarrow {\bf{a}}$ uniformly on compact subsets of $\mathbb{R}^N$, as $l\rightarrow+\infty$.  By \cite[Lemma 2.4]{FF} and its proof therein, we know that
${\bf{a}}^l$ satisfies
\begin{equation}\label{eq3.5}
 \left\langle \nabla {\bf{a}}^l(z)\xi,\xi\right\rangle\geq\frac{1}{\lambda}(|z|^2+s_l^2)^{\frac{p-2}{2}}|\xi|^2
 \qquad \text{and}\qquad |\nabla {\bf{a}}^l(z)|\leq \lambda (|z|^2+s_l^2)^{\frac{p-2}{2}}
\end{equation}
for every $z$, $\xi\in\mathbb{R}^N$, with $s_l\neq0$ and $s_l\rightarrow0$ as $l\rightarrow+\infty$, and some constant $\lambda>0$.

Let $u_{l,k,R}\in W^{1,p}(\mathcal{C}_{k,2R})$ be a weak solution to
\begin{equation}\label{ulk}
    \left\{
            \begin{aligned}
            & -\Delta ^{H,l}_{p}u_{l,k,R}=-\operatorname{div}\left({\bf{a}}^l(\nabla u_{l,k,R})\right)=f(u)\quad \, &{\rm{in}}\,\, \mathcal{C}_{k,2R},\\
            &u_{l,k,R}=u \quad \, &{\rm{on}}\, \Gamma_{k,0}^{2R},\\
            &{\bf{a}}^l(\nabla u_{l,k,R})\cdot \nu =0\quad\, &{\rm{on}}\,\, \Gamma_{k,1}^{2R}.
            \end{aligned}
            \right.
\end{equation}
By the locally boundedness of $u$ and the elliptic estimates in \cite{L,S,T}, we deduce that $\{u_{l,k,R}\}$ are bounded in
 $C^{1,\theta}_{loc}(\overline{\mathcal{C}_{k,2R}}\setminus((\partial\mathcal{C}_k\cap\partial B_{2R}(0))\cup\{0\}))\cap C^{0,\theta}_{loc}(\overline{\mathcal{C}_{k,2R}})$ uniformly in $l$, as $l\rightarrow+\infty$. Then by the Ascoli--Arzel$\acute{\text{a}}$'s Theorem and a diagonal process, we obtain that  $u_{l,k,R}$
   converges in $C^{1}_{loc}(\overline{\mathcal{C}_{k,2R}}\setminus((\partial\mathcal{C}_k\cap\partial B_{2R}(0))\cup\{0\}))$ to the unique solution $\overline{u}_{k,R}$ to
 \begin{equation}\label{uk bar}
    \left\{
        \begin{aligned}
        & -\Delta ^{H}_{p}\overline{u}_{k,R}=f(u)\quad \, &{\rm{in}}\,\, \mathcal{C}_{k,2R},\\
        &\overline{u}_{k,R}=u \quad \, &{\rm{on}}\, \Gamma_{k,0}^{2R},\\
        &{\bf{a}}(\nabla \overline{u}_{k,R})\cdot \nu =0\quad\, &{\rm{on}}\,\, \Gamma_{k,1}^{2R}.
        \end{aligned}
        \right.
   \end{equation}
Note that $u_k$ is also a solution of \eqref{uk bar}, it follows from the uniqueness that $\overline{u}_{k,R}=u_{k,R}$, and hence $u_{l,k,R}\rightarrow u_{k,R}$ in $C^{1}_{loc}(\overline{\mathcal{C}_{k,2R}}\setminus((\partial\mathcal{C}_k\cap\partial B_{2R}(0))\cup\{0\}))$, as $l\rightarrow+\infty$.

\smallskip
Let $L>2R$ to be chosen later. For $\delta>0$ small enough, such that $2\delta<\frac{1}{L}$, we define the set
 \begin{equation*}
    \mathcal{C}_{k,R}^{\delta}:=\{x\in\mathcal{C}_{k,R}\mid\,\text{dist}(x,\partial\mathcal{C}_{k,R})>\delta\}.
 \end{equation*}
Let $\varphi \in C^{1}_c(B_{2R}(0)\setminus B_{\frac{1}{L}}(0))$.  Since $\mathcal{C}_{k,2R}\cap \text{supp}(\varphi)$
 is piecewise smooth, for $\delta>0$ small we can obtain
 that $(\mathcal{C}_{k,2R}^{\delta}\setminus \mathcal{C}_{k,2R}^{2\delta})\cap\text{supp}(\varphi)$ is piecewise smooth. In particular,
 every $x\in(\mathcal{C}_{k,2R}^{\delta}\setminus \mathcal{C}_{k,2R}^{2\delta})\cap\text{supp}(\varphi)$ can be written as
 \begin{equation*}
    x=y-|x-y|\nu(y),
 \end{equation*}
 where $y=y(x)\in\partial\mathcal{C}_{k,2R}^{\delta}$ is the unique projection of $x$ on $\partial\mathcal{C}_{k,2R}^{\delta}$ and $\nu(y)$
 is the unit outer normal vector to $\partial\mathcal{C}_{k,2R}^{\delta}$ at $y$. Moreover, by the method in \cite[Formula (14.98)]{GT}, we can
  define a $C^1$ function $g$, such that $x=g(y,d)$ holds for every
   $x\in(\mathcal{C}_{k,2R}^{\delta}\setminus \mathcal{C}_{k,2R}^{2\delta})\cap\text{supp}(\varphi)$, where $y\in\partial\mathcal{C}_{k,2R}^{\delta}$ is the projection of $x$ on $\partial\mathcal{C}_{k,2R}^{\delta}$ and $d=|x-y|$. Since $u_{l,k,R}$
     solves the non-degenerate equation, we have $u_{l,k,R}\in C^1(\overline{\mathcal{C}_{k,2R}}\setminus((\partial\mathcal{C}_k\cap\partial B_{2R}(0))\cup\{0\}))\cap W^{2,2}_{loc}(\overline{\mathcal{C}_{k,2R}})$ and
     ${\bf{a}}^l(\nabla u_{l,k,R})\in W^{1,2}_{loc}(\overline{\mathcal{C}_{k,2R}})$. Moreover, since
       $\mathcal{C}_{k,2R}^{\delta}\setminus \mathcal{C}_{k,2R}^{2\delta}$ is piecewise smooth, then $u_{l,k,R}\in C^2(\mathcal{C}_{k,2R}^{\delta}\setminus \mathcal{C}_{k,2R}^{2\delta})$.
\smallskip

Let $\zeta _{\delta}: \mathcal{C}_{k,2R}\rightarrow[0,1]$ be a piecewise smooth function such that
\begin{equation*}
    \zeta _{\delta}=1\ \text{in}\ \mathcal{C}_{k,2R}^{2\delta},\quad
    \zeta _{\delta}=0\ \text{in}\ \mathcal{C}_{k,2R}\setminus\mathcal{C}_{k,2R}^{\delta},\quad \text{and}\quad
    \nabla\zeta _{\delta}=-\frac{1}{\delta}\nu(y(x))\,\,\, \text{in}\,\, \mathcal{C}_{k,2R}^{\delta}\setminus\mathcal{C}_{k,2R}^{2\delta}.
\end{equation*}
Choosing $\psi \in C^{1}_c(B_{2R}(0)\setminus B_{\frac{1}{L}}(0))$ as the test function in \eqref{ulk}, and integrating by
parts, we get
\begin{equation}\label{eq3.8}
    \int_{\mathcal{C}_{k,2R}}{\bf{a}}^l(\nabla u_{l,k,R})\cdot\nabla \psi \mathrm{d}x-\int_{\partial\mathcal{C}_{k,2R}}\psi {\bf{a}}^l(\nabla u_{l,k,R})\cdot\nu \mathrm{d}\mathcal{H}^{N-1}
    =\int_{\mathcal{C}_{k,2R}}f(u)\psi \mathrm{d}x.
\end{equation}
Since
\begin{equation*}
    \int_{\partial\mathcal{C}_{k,2R}}\psi {\bf{a}}^l(\nabla u_{l,k,R})\cdot\nu \mathrm{d}\mathcal{H}^{N-1}=\int_{\Gamma_{k,0}^{2R}}\psi {\bf{a}}^l(\nabla u_{l,k,R})\cdot\nu \mathrm{d}\mathcal{H}^{N-1}
    +\int_{\Gamma_{k,1}^{2R}}\psi {\bf{a}}^l(\nabla u_{l,k,R})\cdot\nu \mathrm{d}\mathcal{H}^{N-1},
\end{equation*}
in view of $\psi \in C^{\infty}_c(B_{2R}(0)\setminus B_{\frac{1}{L}}(0))$ and the boundary condition in \eqref{ulk}, we obtain
 that the boundary integral term in \eqref{eq3.8} vanishes, thus \eqref{eq3.8} becomes
 \begin{equation}\label{eq3.9}
    \int_{\mathcal{C}_{k,2R}}{\bf{a}}^l(\nabla u_{l,k,R})\cdot\nabla \psi \mathrm{d}x= \int_{\mathcal{C}_{k,2R}}f(u)\psi \mathrm{d}x.
 \end{equation}

Next, we choose $\psi=\partial_m(\varphi\zeta_{\delta})$ in \eqref{eq3.9} with $m\in\{1,\cdots,N\}$, then obtain that
\begin{equation*}
    \sum\limits_{i=1}^{N}\left(\int_{\mathcal{C}_{k,2R}}\partial_m (a^l_i(\nabla u_{l,k,R}))\zeta_{\delta}\partial_i\varphi \mathrm{d}x+\int_{\mathcal{C}_{k,2R}}\partial_m (a^l_i(\nabla u_{l,k,R}))\varphi\partial_i\zeta_{\delta} \mathrm{d}x\right)
    = \int_{\mathcal{C}_{k,2R}}\partial_m f(u)\zeta_{\delta}\varphi \mathrm{d}x,
\end{equation*}
where the notation ${\bf{a}}^l=(a^l_1,\cdots,a^l_N)$ denotes the components of the vector field ${\bf{a}}^l$. Recalling the definition of $\zeta_{\delta}$, we have
\begin{equation*}
    \lim\limits_{\delta\rightarrow0}\int_{\mathcal{C}_{k,2R}}\partial_m (a^l_i(\nabla u_{l,k,R}))\zeta_{\delta}\partial_i\varphi \mathrm{d}x=
    \int_{\mathcal{C}_{k,2R}}\partial_m (a^l_i(\nabla u_{l,k,R}))\partial_i\varphi \mathrm{d}x.
\end{equation*}

Let $f(x)=\partial_m (a^l_i(\nabla u_{l,k,R}))\varphi(x) $, by $\nabla\zeta _{\delta}=-\frac{1}{\delta}\nu(y)$ and the co-area formula, we have
 \begin{equation*}
    \begin{aligned}
        \int_{\mathcal{C}_{k,2R}^{\delta}\setminus \mathcal{C}_{k,2R}^{2\delta}}f\partial_i\zeta_{\delta}\mathrm{d}x
            &=-\frac{1}{\delta} \int_{\mathcal{C}_{k,2R}^{\delta}\setminus \mathcal{C}_{k,2R}^{2\delta}}\nu_i(y(x))f\mathrm{d}x\\
            &=-\frac{1}{\delta}\int_{\delta}^{2\delta}\mathrm{d}t\int_{\partial\mathcal{C}_{k,2R}^{t}}\nu_i(y(x))f(y-t\nu(y))|\det(Dg)|\mathrm{d}\mathcal{H}^{N-1}\\
            &=-\int_{1}^{2}\mathrm{d}s\int_{\partial\mathcal{C}_{k,2R}^{s\delta}}\nu_i(y(x))f(y-s\delta\nu(y))|\det(Dg)|\mathrm{d}\mathcal{H}^{N-1}.
    \end{aligned}
 \end{equation*}
Since $\nabla u_{l,k,R}\in C^1(\mathcal{C}_{k,2R}^{\delta}\setminus\mathcal{C}_{k,2R}^{2\delta})$, one has $f(x)=\partial_m (a^l_i(\nabla u_{l,k,R}))\varphi(x) \in C^{0}$, hence we can pass to the limit and deduce that
\begin{equation*}
    \lim\limits_{\delta\rightarrow0}\int_{\mathcal{C}_{k,2R}}\partial_m (a^l_i(\nabla u_{l,k,R}))\varphi\partial_i\zeta_{\delta} \mathrm{d}x
    =-\int_{\partial\mathcal{C}_{k,2R}}\partial_m (a^l_i(\nabla u_{l,k,R}))\varphi(x)\nu_i \mathrm{d}\mathcal{H}^{N-1}.
\end{equation*}
Thus we obtain
\begin{equation}\label{eq3.10}
    \sum\limits_{i=1}^{N} \left(\int_{\mathcal{C}_{k,2R}}\partial_m (a^l_i(\nabla u_{l,k,R}))\partial_i\varphi \mathrm{d}x
    -\int_{\partial\mathcal{C}_{k,2R}}\partial_m (a^l_i(\nabla u_{l,k,R}))\varphi(x)\nu_i \mathrm{d}\mathcal{H}^{N-1}\right)
     =-\int_{\mathcal{C}_{k,2R}}\partial_m \varphi f(u)\mathrm{d}x.
\end{equation}

Choose $\varphi=a^l_m(\nabla u_{l,k,R})\rho ^2$, $m\in\{1,2,\cdots,N\}$, where $\rho \in C_c^{\infty}(B_{2R}(0)\setminus B_{\frac{1}{L}}(0))$. Next, using the fact that $\mathcal{C}_{k,2R}$ is convex and the argument as in the proof of \cite[Proposition 2.8]{CFR} (see Formulae (2.45)-(2.50) therein), we will show that the summation $\sum\limits_{m=1}^{N}$ over the second term on left-hand side of \eqref{eq3.10} is negative. First, one has
\begin{equation}\label{eq3.11}
    \partial_m (a^l_i(\nabla u_{l,k,R}))a^l_m(\nabla u_{l,k,R})\rho ^2\nu_i=a_m^l(\nabla u_{l,k,R})\partial_m (a^l_i(\nabla u_{l,k,R})\nu_i)\rho^2-a_m^l(\nabla u_{l,k,R})a_i^l(\nabla u_{l,k,R})\rho^2\partial_m\nu_i.
\end{equation}
Note that $\partial_m\nu_i(x)$ is the second fundamental form $\mathrm{II}_x$ of $\Gamma^{R}_{k,1}$ at $x$:
$$\sum\limits_{i,m=1}^{N}\partial_m\nu_i a^l_i(\nabla u_{l,k,R})a^l_m(\nabla u_{l,k,R})=\mathrm{II}_x({\bf{a}}^l(\nabla u_{l,k,R}),{\bf{a}}^l(\nabla u_{l,k,R})).$$
Since $\mathcal{C}_{k,R}$ is convex, then $\mathrm{II}_x$ is nonnegative definite, we deduce that
\begin{equation*}
    \sum\limits_{i,m=1}^{N}\partial_m\nu_i a^l_i(\nabla u_{l,k,R})a^l_m(\nabla u_{l,k,R})\rho^2=\mathrm{II}_x({\bf{a}}^l(\nabla u_{l,k,R}),{\bf{a}}^l(\nabla u_{l,k,R}))\rho^2\geq0.
\end{equation*}
Then, by \eqref{eq3.11}, we have
\begin{equation*}
    \begin{aligned}
        &\quad \sum\limits_{i,m=1}^{N}\int_{\partial\mathcal{C}_{k,2R}}\partial_m (a^l_i(\nabla u_{l,k,R}))a^l_m(\nabla u_{l,k,R})\rho ^2\nu_i\mathrm{d}\mathcal{H}^{N-1} \\
        &\leq \sum\limits_{i,m=1}^{N}\int_{\partial\mathcal{C}_{k,2R}}a_m^l(\nabla u_{l,k,R})\partial_m (a^l_i(\nabla u_{l,k,R})\nu_i)\rho^2\mathrm{d}\mathcal{H}^{N-1}\\
        &=\int_{\partial\mathcal{C}_{k,2R}}{\bf{a}}^l(\nabla u_{l,k,R})\cdot\nabla({\bf{a}}^l(\nabla u_{l,k,R})\cdot\nu)\rho^2\mathrm{d}\mathcal{H}^{N-1}.
    \end{aligned}
\end{equation*}
Note that ${\bf{a}}^l(\nabla u_{l,k,R})\cdot\nu=0$ on $\Gamma_{k,1}^{2R}$, thus ${\bf{a}}^l(\nabla u_{l,k,R})$ is the tangent vector field of $\partial\Gamma_{k,1}^{2R}$, which implies that ${\bf{a}}^l(\nabla u_{l,k,R})\cdot\nabla({\bf{a}}^l(\nabla u_{l,k,R})\cdot\nu)$ is the tangent derivative of ${\bf{a}}^l(\nabla u_{l,k,R})\cdot\nu$, then it follows from ${\bf{a}}^l(\nabla u_{l,k,R})\cdot\nu=0$ on $\Gamma_{k,1}^{2R}$ that ${\bf{a}}^l(\nabla u_{l,k,R})\cdot\nabla({\bf{a}}^l(\nabla u_{l,k,R})\cdot\nu)=0$ on $\Gamma_{k,1}^{2R}$. Since $\partial\mathcal{C}_{k,2R}=\Gamma_{k,0}^{2R}\cup\Gamma_{k,1}^{2R}$ and $\rho=0$ on $\Gamma_{k,0}^{2R}$, one has
$$\int_{\partial\mathcal{C}_{k,2R}}{\bf{a}}^l(\nabla u_{l,k,R})\cdot\nabla({\bf{a}}^l(\nabla u_{l,k,R})\cdot\nu)\rho^2\mathrm{d}\mathcal{H}^{N-1}=0,$$
thus we obtain that
\begin{equation*}
    \sum\limits_{i,m=1}^{N}\int_{\partial\mathcal{C}_{k,2R}}\partial_m (a^l_i(\nabla u_{l,k,R}))a^l_{m}(\nabla u_{l,k,R})\rho ^2\nu_i \mathrm{d}\mathcal{H}^{N-1}\leq0.
\end{equation*}
Hence by \eqref{eq3.10}, we have
\begin{equation*}
    \sum\limits_{i,m=1}^{N}\int_{\mathcal{C}_{k,2R}}\partial_m (a^l_i(\nabla u_{l,k,R}))\partial_i(a^l_m(\nabla u_{l,k,R})\rho ^2)\mathrm{d}x
    \leq \int_{\mathcal{C}_{k,2R}}|\partial_m \varphi| f(u)\mathrm{d}x.
\end{equation*}
Let $A:=\nabla a^l(z)|_{z=\nabla u_{l,k,R}}=\nabla {\bf{a}}^l(\nabla u_{l,k,R})$, $B:=\nabla^2 u_{l,k,R}(x)$. Then $A$, $B$ are symmetric matrices. Let $\lambda_{min}$ be the
smallest eigenvalue of $A$, and $\lambda_{max}$ be the largest eigenvalue of $A$, then by \eqref{eq3.5}, we know that
$$\frac{1}{\lambda}(|\nabla u_{l,k,R}|^2+s_l^2)^{\frac{N-2}{2}}\leq\lambda_{min}\leq\lambda_{max}\leq\lambda(|\nabla u_{l,k,R}|^2+s_l^2)^{\frac{N-2}{2}}.$$
Note that
$$\nabla ({\bf{a}}^l(\nabla u_{l,k,R}))=\nabla {\bf{a}}^l(\nabla u_{l,k,R})\nabla^2u_{l,k,R}=AB,$$
and
$$\mathrm{trace} [(\nabla ({\bf{a}}^l(\nabla u_{l,k,R})))^2]=\sum\limits_{i,m=1}^{N}\partial_i (a^l_m(\nabla u_{l,k,R}))\partial_m (a^l_i(\nabla u_{l,k,R})),$$
thus we can deduce from  $\frac{\lambda_{max}}{\lambda_{min}}\leq\lambda^{2}$ and Lemma \ref{le:2.9} that
\begin{equation*}
    \begin{aligned}
        N\lambda^4 \mathrm{trace} [(\nabla ({\bf{a}}^l(\nabla u_{l,k,R})))^2]
        &=  N\lambda^4\mathrm{trace}(\nabla {\bf{a}}^l(\nabla u_{l,k,R})\nabla^2u_{l,k,R}\nabla {\bf{a}}^l(\nabla u_{l,k,R})\nabla^2u_{l,k,R})\\
        &\geq  \mathrm{trace}[\nabla {\bf{a}}^l(\nabla u_{l,k,R})\nabla^2u_{l,k,R}(\nabla {\bf{a}}^l(\nabla u_{l,k,R})\nabla^2u_{l,k,R})^T]\\
        &=|\nabla {\bf{a}}^l(\nabla u_{l,k,R})\nabla^2u_{l,k,R}|^2\\
        &=|\nabla ({\bf{a}}^l(\nabla u_{l,k,R}))|^2.
    \end{aligned}
\end{equation*}
As a consequence, we have
\begin{equation*}
    \begin{aligned}
        &\quad \sum\limits_{i,m=1}^{N}\partial_m (a^l_i(\nabla u_{l,k,R}))\partial_i(a^l_m(\nabla u_{l,k,R})\rho^2) \\
        &=\sum\limits_{i,m=1}^{N}\partial_m (a^l_i(\nabla u_{l,k,R}))\partial_i (a^l_m(\nabla u_{l,k,R}))\rho^2\\
        &\quad +2\sum\limits_{i,m=1}^{N}\partial_m (a^l_i(\nabla u_{l,k,R}))a^l_m(\nabla u_{l,k,R})\rho_i\rho\\
        &\geq \frac{1}{N\lambda^4}|\nabla {\bf{a}}^l(\nabla u_{l,k,R})|^2\rho^2+2\sum\limits_{i,m=1}^{N}\partial_m (a^l_i(\nabla u_{l,k,R}))a^l_m(\nabla u_{l,k,R})\rho_i\rho.
    \end{aligned}
\end{equation*}
Then, by Young's inequality, we obtain the following Caccioppoli type estimate:
\begin{eqnarray}\label{eq3.12}
    \quad \int_{\mathcal{C}_{k,2R}}\left\lvert \nabla {\bf{a}}^l(\nabla u_{l,k,R})\right\rvert^2\rho^2\mathrm{d}x \leq C\left(\int_{\mathcal{C}_{k,2R}}\left\lvert {\bf{a}}^l(\nabla u_{l,k,R})\right\rvert^2\left\lvert\nabla\rho \right\rvert^2 \mathrm{d}x+
     \int_{\mathcal{C}_{k,2R}}(f(u))^2\rho^2\mathrm{d}x\right)
\end{eqnarray}
for some constant $C=C(N,\lambda)$. Since $ \rho \in C_c^{\infty}(B_{2R}(0)\setminus B_{\frac{1}{L}}(0))$, by approximation, we know that \eqref{eq3.12} holds for any  $\rho \in C_c^{\infty}(B_{2R}(0))$.  For any open ball  $B_{2r}\subset B_R(0)$, by \eqref{eq3.12} and the proof of Proposition 4.3 in \cite{ACF}, there exists constant $C$ depending only on $N$, $p$, $\lambda$ and $H$, scuh that
\begin{equation}\label{al1}
\int_{\mathcal{C}_{k,R}\cap B_r}\left\lvert {\bf{a}}^l(\nabla u_{l,k,R})\right\rvert^2\mathrm{d}x\leq C\left[r^{-N}\left(\int_{\mathcal{C}_{k,R}\cap (B_{2r}\backslash B_r) }\left\lvert {\bf{a}}^l(\nabla u_{l,k,R})\right\rvert\mathrm{d}x\right)^2+r^2\int_{\mathcal{C}_{k,R}\cap B_{2r}}(f(u))^2\mathrm{d}x \right]
\end{equation}
and
\begin{equation}\label{al2}
 \int_{\mathcal{C}_{k,R}\cap B_{\frac{r}{2}}}\left\lvert \nabla {\bf{a}}^l(\nabla u_{l,k,R})\right\rvert^2\mathrm{d}x\leq C\left[r^{-N-2}\left(\int_{\mathcal{C}_{k,R}\cap (B_{2r}\backslash B_r) }\left\lvert {\bf{a}}^l(\nabla u_{l,k,R})\right\rvert\mathrm{d}x\right)^2+\int_{\mathcal{C}_{k,R \cap B_{2r}}}(f(u))^2\mathrm{d}x \right].
\end{equation}
Since $\{u_{l,k,R}\}$ are bounded in
 $C^{1,\theta}_{loc}(\overline{\mathcal{C}_{k,2R}}\setminus((\partial\mathcal{C}_k\cap\partial B_{2R}(0))\cup\{0\}))\cap C^{0,\theta}_{loc}(\overline{\mathcal{C}_{k,2R}})$ uniformly in $l$, as $l\rightarrow+\infty$, we deduce that ${\bf{a}}^l(\nabla u_{l,k,R})\in W^{1,2}_{loc}(\overline{\mathcal{C}_{k,R}})$ and $\{{\bf{a}}^l(\nabla u_{l,k,R})\}$
is uniformly bounded in $ W^{1,2}_{loc}(\overline{\mathcal{C}_{k,R}})$. Then by $u_{l,k,R}\rightarrow u_{k,R}$ in $C^{1}_{loc}(\overline{\mathcal{C}_{k,2R}}\setminus((\partial\mathcal{C}_k\cap\partial B_{2R}(0))\cup\{0\}))$, there exists a subsequence, still denoted by $\{{\bf{a}}^l(\nabla u_{l,k,R})\}$, such that
\begin{equation*}
{\bf{a}}^l(\nabla u_{l,k,R})\rightarrow {\bf{a}}(\nabla u_{k,R})\quad\text{in}\quad L^2_{loc}(\overline{\mathcal{C}_{k,R}})\quad\text{and}\quad{\bf{a}}^l(\nabla u_{l,k,R})\rightharpoonup  {\bf{a}}(\nabla u_{k,R})\quad\text{in}\quad W^{1,2}_{loc}(\overline{\mathcal{C}_{k,R}}).
\end{equation*}
As a result, we deduce that $\{{\bf{a}}(\nabla u_{k,R})\}$
is uniformly bounded in $W^{1,2}_{loc}(\overline{\mathcal{C}_{k,R}})$, and
\begin{equation}\label{eq3.13}
    \int_{\mathcal{C}_{k,R}}\left\lvert \nabla {\bf{a}}(\nabla u_{k,R})\right\rvert^2\rho^2\mathrm{d}x\leq
    C\left(\int_{\mathcal{C}_{k,R}}\left\lvert {\bf{a}}(\nabla u_{k,R})\right\rvert^2\left\lvert\nabla\rho \right\rvert^2 \mathrm{d}x+
    \int_{\mathcal{C}_{k,R}}(f(u))^2\rho^2\mathrm{d}x\right).
\end{equation}
And also, by \eqref{al1} and \eqref{al2}, we have
\begin{equation}\label{a1}
\int_{\mathcal{C}_{k,R}\cap B_r}\left\lvert {\bf{a}}(\nabla u_{k,R})\right\rvert^2\mathrm{d}x\leq C\left[r^{-N}\left(\int_{\mathcal{C}_{k,R}\cap (B_{2r}\backslash B_r) }\left\lvert {\bf{a}}(\nabla u_{k,R})\right\rvert\mathrm{d}x\right)^2+r^2\int_{\mathcal{C}_{k,R}\cap B_{2r}}(f(u))^2\mathrm{d}x \right]
\end{equation}
and
\begin{equation}\label{a2}
 \int_{\mathcal{C}_{k,R}\cap B_{\frac{r}{2}}}\left\lvert \nabla {\bf{a}}(\nabla u_{k,R})\right\rvert^2\mathrm{d}x\leq C\left[r^{-N-2}\left(\int_{\mathcal{C}_{k,R}\cap (B_{2r}\backslash B_r) }\left\lvert {\bf{a}}(\nabla u_{k,R})\right\rvert\mathrm{d}x\right)^2+\int_{\mathcal{C}_{k,R}\cap B_{2r}}(f(u))^2\mathrm{d}x \right].
\end{equation}
Since $\{u_{k,R}\}$ are uniformly bounded in $C^{1,\theta}_{loc}(\overline{\mathcal{C}_{k,2R}}\setminus((\partial\mathcal{C}_k\cap\partial B_{2R}(0))\cup\{0\}))\cap C^{0,\theta}_{loc}(\overline{\mathcal{C}_{k,2R}})$ with respect to $k$ and in view of \eqref{lim2}, by passing the limit $k\rightarrow+\infty$ in \eqref{eq3.13}, we deduce
\begin{equation}\label{a+}
    \int_{\mathcal{C}}\left\lvert \nabla {\bf{a}}(\nabla u)\right\rvert^2\rho^2\mathrm{d}x\leq
    C\left(\int_{\mathcal{C}}\left\lvert {\bf{a}}(\nabla u)\right\rvert^2\left\lvert\nabla\rho \right\rvert^2 \mathrm{d}x+
    \int_{\mathcal{C}}(f(u))^2\rho^2\mathrm{d}x\right) .
\end{equation}
Since $\rho \in C_c^{\infty}(B_R(0))$ and $R>1$ is arbitrarily large, \eqref{a+} implies immediately that ${\bf{a}}(\nabla u)\in W^{1,2}_{loc}(\overline{\mathcal{C}})$.
\end{proof}

Let $u\in W^{1,p}_{loc}(\overline{\mathcal{C}})\cap L^{\infty}_{loc}(\overline{\mathcal{C}})$ be a positive solution of \eqref{eq:1.1} and $a,b\in\mathbb{R}$ are constants to be determined, we set
\begin{equation}\label{vUV}
 \mathbf{v}=u^a{\bf{a}}(\nabla u),\,\,\mathbf{U}=\nabla {\bf{a}}(\nabla u)\,\,\mathbf{V}=\nabla \mathbf{v},
\end{equation}
and
\begin{equation}\label{omega}
  \bm{\omega}(x):=\left(\mathbf{v}\nabla\mathbf{v}-\frac{1}{N}\mathbf{v}\operatorname{div}\mathbf{v}\right)u^b.
\end{equation}
Thanks to Lemma \ref{Hanack} and Proposition \ref{regular}, we know that $\mathbf{v}\in W^{1,2}_{loc}(\overline{\mathcal{C}})$ for any $a\in\mathbb{R}$. Moreover, $\mathbf{U}$, $\mathbf{V}$ and $\bm{\omega}(x)$ are well defined in $\overline{\mathcal{C}}$ with $\mathbf{U},\mathbf{V},\bm{\omega}(x)\in L^2_{loc}(\overline{\mathcal{C}})$. By \eqref{homogeneity} and a direct calculation, we have
$$\mathbf{U}\nabla u=\frac{p-1}{p}\nabla(H^p(\nabla u)).$$

By \eqref{homogeneity}, \eqref{eq:1.1} and \eqref{vUV}, together with a simple calculation, we have the following equalities.
\begin{lem}\label{ident}
Let $N\geq2$ and $\mathcal{C}\subseteq\mathbb{R}^{N}$ be an open convex cone. Assume that $u\in W^{1,p}_{loc}(\overline{\mathcal{C}})\cap L^{\infty}_{loc}(\overline{\mathcal{C}})$ is a positive solution of \eqref{eq:1.1}. Then
$${\bf{a}}(\nabla u)\cdot \nabla u=H^p(\nabla u),\,\,\mathbf{v}\cdot\nabla u=u^aH^p(\nabla u),\,\,\mathbf{V}=au^{a-1}\nabla u\otimes {\bf{a}}(\nabla u)+u^a\mathbf{U}.$$
Moreover,
 \begin{equation*}
      \left\{
          \begin{aligned}
          &\operatorname{div}\mathbf{v}=u^{a-1}(aH^p(\nabla u)-uf(u)) \quad\,\,\, &{\rm{in}} \,\, \mathcal{C}, \\
          &\mathbf{v}\cdot \nu =0 \quad\,\,\, &{\rm{on}} \,\, \partial\mathcal{C}.
          \end{aligned}
          \right.
    \end{equation*}
\end{lem}
Define
$$I(x):=\operatorname{trace}(\mathbf{V}^2)-\frac{1}{N}(\operatorname{div}\mathbf{v})^2,$$
then it follows from $\mathbf{v}\in W^{1,2}_{loc}(\overline{\mathcal{C}})$ that $I(x)\in L^2_{loc}(\overline{\mathcal{C}})$. Moreover, the following lemma implies that $I(x)$ is a nonnegative function, which is useful in our subsequent proofs (the proof follows the same lines as that of Lemma 3.2 in \cite{CHN}).
\begin{lem}\label{I(x)}
Let $N\geq2$, $1<p<N$ and $\mathcal{C}\subseteq\mathbb{R}^{N}$ be an open convex cone and $H\in C^{2}_{\text{loc}}(\mathbb{R}^{N}\setminus\{0\})$ be a Finsler norm such that $H^{2}$ is uniformly convex. Then $I(x)\geq0$ a.e. in $\mathcal{C}$.
\end{lem}

Define
$$\psi(x):=u^{2a+b-1}[Af(u)+\widehat{A}uf^{\prime}(u)]H^p(\nabla u)+Bu^{2a+b-2}H^{2p}(\nabla u)+C\operatorname{div}(u^{2a+b-1}H^p(\nabla u){\bf{a}}(\nabla u)),$$
where
$$A=\left(\frac{1}{N}+\frac{p-1}{p}\right)b,\,\,\widehat{A}=-\frac{N-1}{N},\,\,B=-\frac{N-1}{N}a^2-\frac{p-1}{p}(2a+b-1)b,\,\,C=\frac{N-1}{N}a+\frac{p-1}{p}b.$$

\medskip

The following integral inequality serves as a anisotropic counterpart to Propositions 6.1 and 7.1 in \cite{SZ} for $p\geq2$ and $1<p<2$ respectively, and also extends the integral identity (3.9) in $\mathbb{R}^N$ of \cite{CHN} to general convex cone $\mathcal{C}$. This integral inequality plays a crucial role in the proof of Liouville type theorems for subcritical anisotropic $p$-Laplacian equation in Section 4 and classification of positive solutions for the critical $p$-Laplacian equation in Section 5.
\begin{prop}\label{inte ineq}
Let $N\geq3$, $\mathcal{C}\subseteq\mathbb{R}^{N}$ be an open convex cone and $H\in C^{2}_{\text{loc}}(\mathbb{R}^{N}\setminus\{0\})$ be a Finsler norm such that $H^{2}$ is uniformly convex. Let $u\in  W^{1,p}_{loc}(\overline{\mathcal{C}})\cap L^{\infty}_{loc}(\overline{\mathcal{C}})$ be a positive solution to \eqref{eq:1.1}, and $\bm{\omega}(x)$, $\psi(x)$ are given as above. Then
\begin{equation}\label{key}
 - \int_{\mathcal{C}}\bm{\omega}(x)\cdot\nabla\varphi(x)\mathrm{d}x\geq\int_{\mathcal{C}}(u^bI(x)+\psi(x))\varphi(x)\mathrm{d}x
\end{equation}
holds for any $0\leq\varphi\in C^{\infty}_c(\mathbb{R}^N)$.
\end{prop}

\begin{proof}
Let
$$I_0:=\int_{\mathcal{C}}u^b\mathbf{v}\cdot\nabla \varphi\operatorname{div}\mathbf{v}\mathrm{d}x.$$
It follows from Lemma \ref{ident} and Proposition \ref{regular} that $u^b\operatorname{div}\mathbf{v}, \mathbf{v}\in W^{1,2}_{loc}(\overline{\mathcal{C}})$, by using Lemma \ref{int part} we have
\begin{equation}\label{I0}
  \begin{aligned}
    I_0&=-\int_{\mathcal{C}}\nabla(u^b\operatorname{div}\mathbf{v})\cdot\mathbf{v}\varphi\mathrm{d}x-\int_{\mathcal{C}}u^b(\operatorname{div}\mathbf{v})^2\varphi\mathrm{d}x+\int_{\partial\mathcal{C}}u^b\operatorname{div}\mathbf{v}(\mathbf{v}\cdot\nu)\varphi\mathrm{d}\mathcal{H}^{N-1}\\
       &=-\int_{\mathcal{C}}u^b(\operatorname{div}\mathbf{v})^2\varphi\mathrm{d}x-b\int_{\mathcal{C}}u^{b-1}\operatorname{div}\mathbf{v}(\nabla u\cdot\mathbf{v})\varphi\mathrm{d}x-\int_{\mathcal{C}}u^b\nabla(\operatorname{div}\mathbf{v})\cdot\mathbf{v}\varphi\mathrm{d}x\\
       &=-\int_{\mathcal{C}}u^b(\operatorname{div}\mathbf{v})^2\varphi\mathrm{d}x-b\int_{\mathcal{C}}u^{b-1}u^{a-1}(aH^p(\nabla u)-uf(u))u^aH^p(\nabla u)\varphi\mathrm{d}x\\
       &-\int_{\mathcal{C}}u^b\nabla(u^{a-1}(aH^p(\nabla u)-uf(u)))\cdot\mathbf{v}\varphi\mathrm{d}x\\
       &=-\int_{\mathcal{C}}u^b(\operatorname{div}\mathbf{v})^2\varphi\mathrm{d}x+\int_{\mathcal{C}}u^{2a+b-1}(A_0f(u)+uf^{\prime}(u))H^p(\nabla u)\varphi\mathrm{d}x\\
       &+B_0\int_{\mathcal{C}}u^{2a+b-2}H^{2p}(\nabla u)\varphi\mathrm{d}x+C_0\int_{\mathcal{C}}u^{a+b}{\bf{a}}(\nabla u)\cdot\nabla(u^{a-1}H^p(\nabla u))\varphi\mathrm{d}x,
  \end{aligned}
\end{equation}
where we have used the fact that $\mathbf{v}\cdot\nu=0$ on $\partial\mathcal{C}$, and
$$A_0=a+b,\,\,B_0=-ab,\,\,C_0=-a.$$
We rewrite $-\int_{\mathcal{C}}u^b(\mathbf{v\nabla\mathbf{v}})\cdot\nabla\varphi\mathrm{d}x$ as follows:
\begin{equation*}
  \begin{aligned}
    -\int_{\mathcal{C}}u^b(\mathbf{v\nabla\mathbf{v}})\cdot\nabla\varphi\mathrm{d}x
    &=-a\int_{\mathcal{C}}u^{a+b-1}(\mathbf{v}\cdot\nabla u){\bf{a}}(\nabla u)\cdot\nabla\varphi\mathrm{d}x-\int_{\mathcal{C}}u^{a+b}(\mathbf{v}\nabla {\bf{a}}(\nabla u))\cdot\nabla\varphi\mathrm{d}x\\
    &=-a\int_{\mathcal{C}}u^{2a+b-1}H^p(\nabla u){\bf{a}}(\nabla u)\cdot\nabla\varphi\mathrm{d}x-\int_{\mathcal{C}}u^{a+b}(\mathbf{v}\nabla {\bf{a}}(\nabla u))\cdot\nabla\varphi\mathrm{d}x\\
    &:=I_1+I_2.
  \end{aligned}
\end{equation*}
For the term $I_1$, Lemma \ref{int part} concludes that
\begin{equation}\label{I1}
  \begin{aligned}
    I_1&=a\int_{\mathcal{C}}\nabla(u^{2a+b-1}H^p(\nabla u))\cdot {\bf{a}}(\nabla u)\varphi\mathrm{d}x+a\int_{\mathcal{C}}u^{2a+b-1}H^p(\nabla u)\operatorname{div}{\bf{a}}(\nabla u)\varphi\mathrm{d}x\\
       &-a\int_{\partial\mathcal{C}}u^{2a+b-1}H^p(\nabla u){\bf{a}}(\nabla u)\cdot\nu\varphi\mathrm{d}\mathcal{H}^{N-1}\\
       &=-a\int_{\mathcal{C}}u^{2a+b-1}H^p(\nabla u)f(u)\varphi\mathrm{d}x+B_1\int_{\mathcal{C}}u^{2a+b-2}H^{2p}(\nabla u)\varphi\mathrm{d}x\\
       &+a\int_{\mathcal{C}}u^{2a+b-1}\nabla(H^p(\nabla u))\cdot {\bf{a}}(\nabla u)\varphi\mathrm{d}x,
  \end{aligned}
\end{equation}
where $B_1=a(2a+b-1)$.

Next, we will prove
\begin{equation}\label{II2}
 I_2\geq  (a+b)\int_{\mathcal{C}}u^{a+b-1}\nabla u\cdot (\mathbf{v}\nabla {\bf{a}}(\nabla u))\varphi\mathrm{d}x
  +\int_{\mathcal{C}}u^{a+b}\left(\operatorname{trace}(\mathbf{V}\nabla {\bf{a}}(\nabla u))-\mathbf{v}\cdot \nabla u f^{\prime}(u)\right) \varphi\mathrm{d}x.
\end{equation}
Note that Proposition \ref{regular} implies that $\nabla {\bf{a}}(\nabla u)\in L^2_{loc}(\overline{\mathcal{C}})$, due to the lackness of regularity, we cannot apply Lemma \ref{int part} directly to deal with the term $I_2$. To this end, we employ the approximation method similar to that in the proof of Proposition \ref{regular}.

Let $R>0$ such that $\text{supp}(\varphi)\subset B_R(0)$, $\mathcal{C}_k$, $\mathcal{C}_{k,R}$, $\Gamma_{k,0}^R$, $\Gamma_{k,1}^R$ and ${\bf{a}}^l$ be as defined in the proof of Proposition \ref{regular}, and $u_{l,k,R}\in W^{1,p}_{loc}(\overline{\mathcal{C}_{k,2R}})$ be a weak solution of
\begin{equation*}
 \left\{
         \begin{aligned}
         & -\operatorname{div}\left({\bf{a}}^l(\nabla u_{l,k,R})\right)=f(u)\quad \, &{\rm{in}}\,\, \mathcal{C}_{k,2R},\\
         &u_{l,k,R}(x)=u(x)\quad \, &{\rm{in}}\ \Gamma_{k,0}^{2R},\\
         &{\bf{a}}^l(\nabla u_{l,k,R})\cdot \nu =0\quad\, &{\rm{on}}\,\,\Gamma_{k,1}^{2R}.
         \end{aligned}
         \right.
\end{equation*}
Then by the proof of Proposition \ref{regular}, we have $u_{l,k,R}\rightarrow u_{k,R}$ in $C^{1}_{loc}(\overline{\mathcal{C}_{k,2R}}\setminus((\partial\mathcal{C}_k\cap\partial B_{2R}(0))\cup\{0\}))$, as $l\rightarrow+\infty$. Moreover, since $\Gamma_{k,1}^{2R}$ is smooth, we have $u_{l,k,R}\in C^2(\overline{\mathcal{C}_{k,2R}}\setminus((\partial\mathcal{C}_k\cap\partial B_{2R}(0))\cup\{0\}))$.

For $h\neq0$, let $h^i:=he_i$ and $\mathbf{U}_{l,k,R,h}$ be the matrix with components
$$\{\mathbf{U}_{l,k,R,h}\}_i^j=\frac{{{a}}^l_j(\nabla u_{l,k,R})(x+h^i)-{{a}}^l_j(\nabla u_{l,k,R})(x)}{h},$$
where $x\in \overline{\mathcal{C}_{k,R}}$ and $|h|$ is chosen to be small enough such that those components are meaningful for a given $x\in\overline{\mathcal{C}_{k,R}}$, the notation ${\bf{a}}^l=({a}^l_1,\cdots,{{a}}^l_N)$ denotes the components of the vector field ${\bf{a}}^l$. Then we have

\begin{equation}\label{U}
  \mathbf{U}_{l,k,R,h}\rightarrow \nabla {\bf{a}}^l(\nabla u_{l,k,R})\,\,\text{in}\,\, L^2_{loc}(\overline{\mathcal{C}_{k,R}}),\quad\,\text{as}\,\,h\rightarrow 0.
\end{equation}

Let $\mathbf{v}_{l,k,R}:=u_{l,k,R}^a{\bf{a}}^l(\nabla u_{l,k,R})$ and $\mathbf{V}_{l,k,R}:=\nabla \mathbf{v}_{l,k,R}$. By Lemma \ref{Hanack}, $u$ is uniformly positive in $\mathcal{C}\cap B_{2R}(0)$, thanks to this, $u_{k,R}$ is nonnegative in $\mathcal{C}_{k,R}$ for $k$ large enough, then by applying Lemma \ref{Hanack} to $u_{k,R}$, we obtain that $u_{k,R}$ is uniformly positive in $\mathcal{C}_{k,R}$ for $k$ large enough, and hence $u_{l,k,R}$ is uniformly positive in $\mathcal{C}_{k,R}$ for $k$, $l$ large enough. In what follows, we always assume that $k$, $l$ large enough so that  $u_{k,R}$ and $u_{l,k,R}$ are uniformly positive in $\mathcal{C}_{k,R}$. In particular, $\mathbf{v}_{l,k,R}\in  W^{1,2}_{loc}(\overline{\mathcal{C}_{k,R}})$ for any $a\in\mathbb{R}$. Choose smooth function $\chi:\mathbb{R}\rightarrow [0,1]$, such that $\chi(t)=0$ for $t\leq \frac{1}{2}$, $\chi(t)=1$ for $t\geq 1$ and $|\chi^{\prime}(t)|\leq C$ for some constant $C>0$. For any fixed $\varepsilon>0$ small enough, consider smooth function $\varphi_{\varepsilon}:=\chi_{\varepsilon}\varphi$, where $\chi_{\varepsilon}(x):=\chi(\frac{|x|}{\varepsilon})$, then $|\nabla \chi_{\varepsilon}|\leq\frac{C}{\varepsilon}$, $0\leq \varphi_{\varepsilon}\leq \varphi$, $\varphi_{\varepsilon}\rightarrow \varphi$ as $\varepsilon\rightarrow0$ and $\varphi_{\varepsilon}=0$ in $B_{\frac{\varepsilon}{2}}(0)$. Define
$$I^{\varepsilon}_{l,k,R}:=-\int_{\mathcal{C}_{k,R}}u_{l,k,R}^{a+b}(\mathbf{v}_{l,k,R}\nabla {\bf{a}}^l(\nabla u_{l,k,R}))\cdot\nabla\varphi_{\varepsilon}\mathrm{d}x,$$
and
$$I^{\varepsilon}_{l,k,R,h}:=-\int_{\mathcal{C}_{k,R}}u_{l,k,R}^{a+b}(\mathbf{v}_{l,k,R}\mathbf{U}_{l,k,R,h})\cdot\nabla\varphi_{\varepsilon}\mathrm{d}x.$$
Then it follows from \eqref{U} that $I^{\varepsilon}_{l,k,R}=\lim\limits_{h\rightarrow0}I^{\varepsilon}_{l,k,R,h}$. By Lemma \ref{int part}, we have
\begin{equation*}
  \begin{aligned}
   I^{\varepsilon}_{l,k,R,h}
    &=\int_{\mathcal{C}_{k,R}}\nabla(u_{l,k,R}^{a+b})\cdot (\mathbf{v}_{l,k,R}\mathbf{U}_{l,k,R,h})\varphi_{\varepsilon}\mathrm{d}x
    +\int_{\mathcal{C}_{k,R}}u_{l,k,R}^{a+b}\operatorname{div} (\mathbf{v}_{l,k,R}\mathbf{U}_{l,k,R,h})\varphi_{\varepsilon}\mathrm{d}x\\
    &-\int_{\Gamma_{k,1}^R}u_{l,k,R}^{a+b}(\mathbf{v}_{l,k,R}\mathbf{U}_{l,k,R,h})\cdot\nu\varphi_{\varepsilon}\mathrm{d}\mathcal{H}^{N-1}\\
    &=(a+b)\int_{\mathcal{C}_{k,R}}u_{l,k,R}^{a+b-1}\nabla u_{l,k,R}\cdot (\mathbf{v}_{l,k,R}\mathbf{U}_{l,k,R,h})\varphi_{\varepsilon}\mathrm{d}x
    +\int_{\mathcal{C}_{k,R}}u_{l,k,R}^{a+b}\operatorname{trace}(\mathbf{V}_{l,k,R}\mathbf{U}_{l,k,R,h})\varphi_{\varepsilon}\mathrm{d}x\\
    &+\int_{\mathcal{C}_{k,R}}u_{l,k,R}^{a+b}\mathbf{v}_{l,k,R}\cdot \operatorname{div}(\mathbf{U}_{l,k,R,h}^T)\varphi_{\varepsilon}\mathrm{d}x
    -\int_{\Gamma_{k,1}^R}u_{l,k,R}^{a+b}(\mathbf{v}_{l,k,R}\mathbf{U}_{l,k,R,h})\cdot\nu\varphi_{\varepsilon}\mathrm{d}\mathcal{H}^{N-1},
  \end{aligned}
\end{equation*}
where we have used the fact that $\varphi_{\varepsilon}=0$ on $\Gamma_{k,0}^R$. By \eqref{U}, we have
\begin{equation*}
 \lim\limits_{h\rightarrow0} \int_{\mathcal{C}_{k,R}}u_{l,k,R}^{a+b-1}\nabla u_{l,k,R}\cdot (\mathbf{v}_{l,k,R}\mathbf{U}_{l,k,R,h})\varphi_{\varepsilon}\mathrm{d}x=\int_{\mathcal{C}_{k,R}}u_{l,k,R}^{a+b-1}\nabla u_{l,k,R}\cdot (\mathbf{v}_{l,k,R}\nabla {\bf{a}}^l(\nabla u_{l,k,R}))\varphi_{\varepsilon}\mathrm{d}x
\end{equation*}
and
\begin{equation*}
 \lim\limits_{h\rightarrow0}\int_{\mathcal{C}_{k,R}}u_{l,k,R}^{a+b}\operatorname{trace}(\mathbf{V}_{l,k,R}\mathbf{U}_{l,k,R,h})\varphi_{\varepsilon}\mathrm{d}x=\int_{\mathcal{C}_{k,R}}u_{l,k,R}^{a+b}\operatorname{trace}(\mathbf{V}_{l,k,R}\nabla {\bf{a}}^l(\nabla u_{l,k,R}))\varphi_{\varepsilon}\mathrm{d}x.
\end{equation*}
Moreover, similar to the proof of formula (6.22) in \cite{SZ} (see also \cite[Formula (3.10)]{CHN} for the anisotropic case), we have
$$\operatorname{div}(\mathbf{U}_{l,k,R,h}^T)\rightarrow-f^{\prime}(u)\nabla u,\quad\text{uniformly on compact subsets of }\mathcal{C}_{k,R}, \quad\text{as}\,\,h\rightarrow0.$$
This gives that
\begin{equation*}
\lim\limits_{h\rightarrow0}\int_{\mathcal{C}_{k,R}}u_{l,k,R}^{a+b}\mathbf{v}_{l,k,R}\cdot \operatorname{div}(\mathbf{U}_{l,k,R,h}^T)\varphi_{\varepsilon}\mathrm{d}x=-\int_{\mathcal{C}_{k,R}}u_{l,k,R}^{a+b}\mathbf{v}_{l,k,R}\cdot \nabla u f^{\prime}(u)\varphi_{\varepsilon}\mathrm{d}x.
\end{equation*}

 Since ${\bf{a}}^l$ is smooth and $u_{l,k,R}\in C^2(\overline{\mathcal{C}_{k,2R}}\setminus((\partial\mathcal{C}_k\cap\partial B_{2R}(0))\cup\{0\}))$, we have $\nabla {\bf{a}}^l(\nabla u_{l,k,R})\in C^0(\overline{\mathcal{C}_{k,2R}}\setminus((\partial\mathcal{C}_k\cap\partial B_{2R}(0))\cup\{0\}))$, this yields that $\mathbf{U}_{l,k,R,h}$ is uniformly bounded in $\overline{\mathcal{C}_{k,R}}\cap (B_{\frac{\varepsilon}{2}}(0))^c$ with respect to $h$ for sufficiently small $h$.
Thus by dominated convergence theorem and a similar calculation in the proof of Proposition \ref{regular} with the test function $\rho^2=u_{l,k,R}^{2a+b}\varphi_{\varepsilon}$, we obtain that
\begin{equation*}
  \begin{aligned}
    \lim\limits_{h\rightarrow0}\int_{\Gamma_{k,1}^R} u_{l,k,R}^{a+b}(\mathbf{v}_{l,k,R}\mathbf{U}_{l,k,R,h})\cdot\nu\varphi_{\varepsilon}\mathrm{d}\mathcal{H}^{N-1}
    &=\int_{\Gamma_{k,1}^R}u_{l,k,R}^{2a+b}\nabla({\bf{a}}^l(\nabla u_{l,k,R})\cdot\nu)\cdot {\bf{a}}^l(\nabla u_{l,k,R})\varphi_{\varepsilon}\mathrm{d}\mathcal{H}^{N-1}\\
    &-\int_{\Gamma_{k,1}^R}u_{l,k,R}^{2a+b}\sum_{i,j=1}^{N}{\bf{a}}^l_j(\nabla u_{l,k,R}){\bf{a}}^l_i(\nabla u_{l,k,R})\partial_i\nu_j\varphi_{\varepsilon}\mathrm{d}\mathcal{H}^{N-1}\\
    &\leq0.
  \end{aligned}
\end{equation*}
Therefore, we deduce that
\begin{equation}\label{Iepi}
\begin{aligned}
  I^{\varepsilon}_{l,k,R}
  &\geq(a+b) \int_{\mathcal{C}_{k,R}}u_{l,k,R}^{a+b-1}\nabla u_{l,k,R}\cdot (\mathbf{v}_{l,k,R}\nabla {\bf{a}}^l(\nabla u_{l,k,R}))\varphi_{\varepsilon}\mathrm{d}x\\
  &+\int_{\mathcal{C}_{k,R}}u_{l,k,R}^{a+b}\operatorname{trace}(\mathbf{V}_{l,k,R}\nabla {\bf{a}}^l(\nabla u_{l,k,R}))\varphi_{\varepsilon}\mathrm{d}x
  -\int_{\mathcal{C}_{k,R}}u_{l,k,R}^{a+b}\mathbf{v}_{l,k,R}\cdot \nabla u f^{\prime}(u)\varphi_{\varepsilon}\mathrm{d}x.
\end{aligned}
\end{equation}
Recall that
\begin{equation*}
  \begin{aligned}
 I^{\varepsilon}_{l,k,R}
 &=-\int_{\mathcal{C}_{k,R}}u_{l,k,R}^{a+b}(\mathbf{v}_{l,k,R}\nabla {\bf{a}}^l(\nabla u_{l,k,R}))\cdot\nabla\varphi_{\varepsilon}\mathrm{d}x\\
 &=-\int_{\mathcal{C}_{k,R}}u_{l,k,R}^{2a+b}({\bf{a}}^l(\nabla u_{l,k,R})\nabla {\bf{a}}^l(\nabla u_{l,k,R}))\cdot\left(\nabla\varphi \chi{\varepsilon}+\nabla\chi_{\varepsilon}\varphi\right) \mathrm{d}x.
  \end{aligned}
\end{equation*}
Since $u_{l,k,R}^{2a+b}\varphi$ and ${\bf{a}}^l(\nabla u_{l,k,R})$ are bounded in $\mathcal{C}_{k,R}$, then H\"older's inequality and $|\nabla \chi_{\varepsilon}|\leq\frac{C}{\varepsilon}$ imply
\begin{equation}\label{Iep}
\begin{aligned}
&\quad \left\lvert \int_{\mathcal{C}_{k,R}}u_{l,k,R}^{2a+b}({\bf{a}}^l(\nabla u_{l,k,R})\nabla {\bf{a}}^l(\nabla u_{l,k,R}))\cdot\nabla\chi_{\varepsilon}\varphi\mathrm{d}x\right\rvert \\
&\leq C\int_{\mathcal{C}_{k,R}}|\nabla {\bf{a}}^l(\nabla u_{l,k,R})||\nabla\chi_{\varepsilon}|\mathrm{d}x\\
&\leq C \varepsilon^{\frac{N-2}{2}}\left(\int_{\mathcal{C}_{k,R}\cap (B_{\varepsilon}\backslash B_{\frac{\varepsilon}{2}})} |\nabla {\bf{a}}^l(\nabla u_{l,k,R})|^2\mathrm{d}x\right)^{\frac{1}{2}}.
\end{aligned}
\end{equation}
By the dominated convergence theorem and \eqref{Iep}, we have
$$ \lim\limits_{\varepsilon\rightarrow0}I^{\varepsilon}_{l,k,R}=-\int_{\mathcal{C}_{k,R}}u_{l,k,R}^{a+b}(\mathbf{v}_{l,k,R}\nabla {\bf{a}}^l(\nabla u_{l,k,R}))\cdot\nabla\varphi\mathrm{d}x\triangleq I_{l,k,R}.$$
Moreover, by the dominated convergence theorem, the three terms on the right-hand side of \eqref{Iepi} converge, thus we arrive at
\begin{equation*}
\begin{aligned}
  I_{l,k,R}
  &\geq(a+b) \int_{\mathcal{C}_{k,R}}u_{l,k,R}^{a+b-1}\nabla u_{l,k,R}\cdot (\mathbf{v}_{l,k,R}\nabla {\bf{a}}^l(\nabla u_{l,k,R}))\varphi\mathrm{d}x\\
  &+\int_{\mathcal{C}_{k,R}}u_{l,k,R}^{a+b}\operatorname{trace}(\mathbf{V}_{l,k,R}\nabla {\bf{a}}^l(\nabla u_{l,k,R}))\varphi\mathrm{d}x
  -\int_{\mathcal{C}_{k,R}}u_{l,k,R}^{a+b}\mathbf{v}_{l,k,R}\cdot \nabla u f^{\prime}(u)\varphi\mathrm{d}x.
\end{aligned}
\end{equation*}
By the proof of Proposition \ref{regular}, there exists a subsequence, still denoted by $\{{\bf{a}}^l(\nabla u_{l,k,R})\}$, such that
\begin{equation*}
{\bf{a}}^l(\nabla u_{l,k,R})\rightarrow {\bf{a}}(\nabla u_{k,R})\quad\text{in}\quad L^2_{loc}(\overline{\mathcal{C}_{k,R}})\quad\text{and}\quad{\bf{a}}^l(\nabla u_{l,k,R})\rightharpoonup  {\bf{a}}(\nabla u_{k,R})\quad\text{in}\quad W^{1,2}_{loc}(\overline{\mathcal{C}_{k,R}}).
\end{equation*}
Thanks to this, we have
\begin{equation}\label{3.26}
\lim\limits_{l\rightarrow\infty}\int_{\mathcal{C}_{k,R}}u_{l,k,R}^{a+b-1}\nabla u_{l,k,R}\cdot (\mathbf{v}_{l,k,R}\nabla {\bf{a}}^l(\nabla u_{l,k,R}))\varphi\mathrm{d}x=\int_{\mathcal{C}_{k,R}}u_{k,R}^{a+b-1}\nabla u_{k,R}\cdot (\mathbf{v}_{k,R}\nabla {\bf{a}}(\nabla u_{k,R}))\varphi\mathrm{d}x,
\end{equation}

\begin{equation}\label{3.27}
\lim\limits_{l\rightarrow\infty}\int_{\mathcal{C}_{k,R}}u_{l,k,R}^{a+b}\mathbf{v}_{l,k,R}\cdot \nabla u f^{\prime}(u)\varphi\mathrm{d}x=\int_{\mathcal{C}_{k,R}}u_{k,R}^{a+b}\mathbf{v}_{k,R}\cdot \nabla u f^{\prime}(u)\varphi\mathrm{d}x,
\end{equation}
and
\begin{equation}\label{3.28}
\begin{aligned}
  &\lim\limits_{l\rightarrow\infty}\int_{\mathcal{C}_{k,R}} u_{l,k,R}^{2a+b-1}\operatorname{trace}(\nabla u_{l,k,R}\otimes {\bf{a}}^l(\nabla u_{l,k,R})\nabla {\bf{a}}^l(\nabla u_{l,k,R}))\varphi\mathrm{d}x\\
  &=\int_{\mathcal{C}_{k,R}} u_{k,R}^{2a+b-1}\operatorname{trace}(\nabla u_{k,R}\otimes {\bf{a}}(\nabla u_{k,R})\nabla {\bf{a}}(\nabla u_{k,R}))\varphi\mathrm{d}x,
\end{aligned}
\end{equation}
where $\mathbf{v}_{k,R}=u_{k,R}^a {\bf{a}}(\nabla u_{k,R})$. Since $\{\nabla {\bf{a}}^l(\nabla u_{l,k,R})\}$ is uniformly bounded in $L^2_{loc}(\overline{\mathcal{C}_{k,R}})$, we know that  $\{\operatorname{trace}(\nabla {\bf{a}}^l(\nabla u_{l,k,R})\nabla {\bf{a}}^l(\nabla u_{l,k,R}))\}$ is uniformly bounded in $L^1_{loc}(\overline{\mathcal{C}_{k,R}})$, and it follows from \eqref{al2} that $\{\operatorname{trace}(\nabla {\bf{a}}^l(\nabla u_{l,k,R})\nabla {\bf{a}}^l(\nabla u_{l,k,R}))\}$ is equi-integrable in $L^1_{loc}(\overline{\mathcal{C}_{k,R}})$ in view of $N\geq 3$. Then by Theorem 4.30 of \cite{Brezis}, there exists a subsequence of $\{\operatorname{trace}(\nabla {\bf{a}}^l(\nabla u_{l,k,R})\nabla {\bf{a}}^l(\nabla u_{l,k,R}))\}$ (still denoted by $\{\operatorname{trace}(\nabla {\bf{a}}^l(\nabla u_{l,k,R})\nabla {\bf{a}}^l(\nabla u_{l,k,R}))\}$) and a function $g\in L^1_{loc}(\overline{\mathcal{C}_{k,R}})$, such that $\operatorname{trace}(\nabla {\bf{a}}^l(\nabla u_{l,k,R})\nabla {\bf{a}}^l(\nabla u_{l,k,R}))\rightharpoonup g$ in $L^1_{loc}(\overline{\mathcal{C}_{k,R}})$. For any fixed nonnegative test function $\psi\in C^{\infty}_c(\mathbb{R}^N)$, we define
$$\mathcal{J}(A):=\int_{\mathcal{C}_{k,R}}\operatorname{trace}(A(x))^2\psi\mathrm{d}x,$$
where $A(x)\in L^2_{loc}(\overline{\mathcal{C}_{k,R}})$ is a symmetric matrix function, and hence $\mathcal{J}$ is a convex functional. By $\nabla{\bf{a}}^l(\nabla u_{l,k,R})\rightharpoonup  \nabla{\bf{a}}(\nabla u_{k,R})$ in $L^2_{loc}(\overline{\mathcal{C}_{k,R}})$ and the weakly lower semi-continuity of $\mathcal{J}$, we obtain
\begin{equation}\label{3.29}
\begin{aligned}
  \liminf\limits_{l\rightarrow\infty}\mathcal{J}(\nabla {\bf{a}}^l(\nabla u_{l,k,R}))
  &=\liminf\limits_{l\rightarrow\infty}\int_{\mathcal{C}_{k,R}}\operatorname{trace}(\nabla {\bf{a}}^l(\nabla u_{l,k,R})\nabla {\bf{a}}^l(\nabla u_{l,k,R}))\psi\mathrm{d}x\\
  &\geq \mathcal{J}(\nabla {\bf{a}}(\nabla u_{k,R})) \\
  &=\int_{\mathcal{C}_{k,R}}\operatorname{trace}(\nabla {\bf{a}}(\nabla u_{k,R})\nabla {\bf{a}}(\nabla u_{k,R}))\psi\mathrm{d}x.
\end{aligned}
\end{equation}
On the other hand, $\operatorname{trace}(\nabla {\bf{a}}^l(\nabla u_{l,k,R})\nabla {\bf{a}}^l(\nabla u_{l,k,R}))\rightharpoonup g$ in $L^1_{loc}(\overline{\mathcal{C}_{k,R}})$ gives that
\begin{equation}\label{3.30}
\lim\limits_{l\rightarrow\infty}\int_{\mathcal{C}_{k,R}}\operatorname{trace}(\nabla {\bf{a}}^l(\nabla u_{l,k,R})\nabla {\bf{a}}^l(\nabla u_{l,k,R}))\psi\mathrm{d}x=\int_{\mathcal{C}_{k,R}}g\psi\mathrm{d}x.
\end{equation}
By the arbitrariness of $\psi$, we deduce from \eqref{3.29} and \eqref{3.30} that
\begin{equation}\label{3.31}
  g\geq \operatorname{trace}(\nabla {\bf{a}}(\nabla u_{k,R})\nabla {\bf{a}}(\nabla u_{k,R}))\quad\text{a.e.}\quad\text{in}\quad \mathcal{C}_{k,R}.
\end{equation}
Thanks to \eqref{3.31} and $u_{l,k,R}\rightarrow u_{k,R}$ in $C^{1}_{loc}(\overline{\mathcal{C}_{k,2R}}\setminus((\partial\mathcal{C}_k\cap\partial B_{2R}(0))\cup\{0\}))$, we obtain
\begin{equation}\label{3.32}
\begin{aligned}
  &\quad \lim\limits_{l\rightarrow\infty}\int_{\mathcal{C}_{k,R}}u_{l,k,R}^{2a+b}\operatorname{trace}(\nabla {\bf{a}}^l(\nabla u_{l,k,R})\nabla {\bf{a}}^l(\nabla u_{l,k,R}))\varphi\mathrm{d}x \\ &
  =\int_{\mathcal{C}_{k,R}}u_{k,R}^{2a+b}g\varphi\mathrm{d}x\\
  &\geq \int_{\mathcal{C}_{k,R}}u_{k,R}^{2a+b}\operatorname{trace}(\nabla {\bf{a}}(\nabla u_{k,R})\nabla {\bf{a}}(\nabla u_{k,R}))\varphi\mathrm{d}x.
\end{aligned}
\end{equation}
By \eqref{3.28} and \eqref{3.32}, we have
\begin{equation}\label{3.33}
  \lim\limits_{l\rightarrow\infty}\int_{\mathcal{C}_{k,R}}u_{l,k,R}^{a+b}\operatorname{trace}(\mathbf{V}_{l,k,R}\nabla {\bf{a}}^l(\nabla u_{l,k,R}))\varphi\mathrm{d}x\geq\int_{\mathcal{C}_{k,R}}u_{k,R}^{a+b}\operatorname{trace}(\mathbf{V}_{k,R}\nabla {\bf{a}}(\nabla u_{k,R}))\varphi\mathrm{d}x,
\end{equation}
where $\mathbf{V}_{k,R}=\nabla \mathbf{v}_{k,R}$. Combining \eqref{3.26} with \eqref{3.27} and \eqref{3.33}, we obtain
\begin{equation*}
 \begin{aligned}
  \lim\limits_{l\rightarrow\infty}I_{l,k,R}
 &\geq \int_{\mathcal{C}_{k,R}}u_{k,R}^{a+b-1}\nabla u_{k,R}\cdot (\mathbf{v}_{k,R}\nabla {\bf{a}}(\nabla u_{k,R}))\varphi\mathrm{d}x+\int_{\mathcal{C}_{k,R}}u_{k,R}^{a+b}\operatorname{trace}(\mathbf{V}_{k,R}\nabla {\bf{a}}(\nabla u_{k,R}))\varphi\mathrm{d}x\\
 &-\int_{\mathcal{C}_{k,R}}u_{k,R}^{a+b}\mathbf{v}_{k,R}\cdot \nabla u f^{\prime}(u)\varphi\mathrm{d}x\\
 &:=I_{k,R}.
 \end{aligned}
\end{equation*}
By \eqref{a1}, \eqref{a2} and a similar derivation applied to $I_{k,R}$, we obtain
\begin{equation*}
\begin{aligned}
\lim\limits_{k\rightarrow\infty}I_{k,R}
\geq & (a+b)\int_{\mathcal{C}}u^{a+b-1}\nabla u\cdot \Big(\mathbf{v}\nabla \big({\bf{a}}(\nabla u)\big)\Big)\varphi\mathrm{d}x \\
& +\int_{\mathcal{C}}u^{a+b}\left(\operatorname{trace}\Big(\mathbf{V}\nabla \big( {\bf{a}}(\nabla u)\big)\Big)-\mathbf{v}\cdot \nabla u f^{\prime}(u)\right) \varphi\mathrm{d}x.
\end{aligned}
\end{equation*}
As a consequence, we obtain the validity of \eqref{II2}, since $I_2=\lim\limits_{k\rightarrow\infty}\lim\limits_{l\rightarrow\infty}I_{l,k,R}$. By direct calculations, we have
\begin{equation*}
    (a+b)u^{a+b-1}\nabla u\cdot(u^a{\bf{a}}(\nabla u)\mathbf{U})=\frac{p-1}{p}(a+b)u^{2a+b-1}{\bf{a}}(\nabla u)\cdot\nabla (H^p(\nabla u)),
\end{equation*}
and
\begin{equation*}
  \begin{aligned}
  \operatorname{trace}(\mathbf{V}\nabla {\bf{a}}(\nabla u))&=u^{-a}\operatorname{trace}(\mathbf{V}^2)-au^{-1}{\bf{a}}(\nabla u)\cdot(\mathbf{V}\nabla u)\\
    &=u^{-a}\operatorname{trace}(\mathbf{V}^2)-a^2u^{a-2}H^{2p}(\nabla u)-\frac{p-1}{p}au^{a-1}{\bf{a}}(\nabla u)\cdot\nabla(H^p(\nabla u)),
  \end{aligned}
\end{equation*}
where we have used the fact that  $\mathbf{U}\nabla u=\frac{p-1}{p}\nabla(H^p(\nabla u))$. Thus we deduce that
\begin{equation}\label{I2}
  \begin{aligned}
    I_2&\geq\int_{\mathcal{C}}u^b\operatorname{trace}(\mathbf{V}^2)\varphi\mathrm{d}x-\int_{\mathcal{C}}u^{2a+b}f^{\prime}(u)H^{p}(\nabla u)\varphi\mathrm{d}x\\
    &-a^2\int_{\mathcal{C}}u^{2a+b-2}H^{2p}(\nabla u)\varphi\mathrm{d}x+\frac{p-1}{p}b\int_{\mathcal{C}}u^{2a+b-1}{\bf{a}}(\nabla u)\cdot\nabla (H^{p}(\nabla u))\mathrm{d}x.
  \end{aligned}
\end{equation}
Since
\begin{equation*}
  \begin{aligned}
    \operatorname{div}(u^{2a+b-1}H^p(\nabla u){\bf{a}}(\nabla u))
    &=u^{a+b}{\bf{a}}(\nabla u)\cdot\nabla (u^{a-1}H^p(\nabla u))-u^{2a+b-1}f(u)H^p(\nabla u)\\
    &+(a+b)u^{2a+b-2}H^{2p}(\nabla u),
  \end{aligned}
\end{equation*}
and
\begin{equation*}
  \begin{aligned}
    \operatorname{div}(u^{2a+b-1}H^p(\nabla u){\bf{a}}(\nabla u))
    &=u^{2a+b-1}{\bf{a}}(\nabla u)\cdot\nabla (H^p(\nabla u))-u^{2a+b-1}f(u)H^p(\nabla u)\\
    &+(2a+b-1)u^{2a+b-2}H^{2p}(\nabla u),
  \end{aligned}
\end{equation*}
and noting that
$$ - \int_{\mathcal{C}}\bm{\omega}(x)\cdot\nabla\varphi(x)\mathrm{d}x=\frac{I_0}{N}+I_1+I_2,$$
combining with \eqref{I0}, \eqref{I1} and \eqref{I2}, we obtain the validity of \eqref{inte ineq}. Thus we complete the proof of Proposition \ref{inte ineq}.
\end{proof}

\section{The Liouville's type theorem in convex cones}
In this subsection, we assume that $f\in C^{1}[0,+\infty)$ is subcritical with $f(t)>0$ for all $t>0$. Let
$$a=-\frac{p-1}{p}p^*,\qquad b=\frac{N-1}{N}p^*-d,$$
where $d$ is a positive parameter to be determined and will be chosen small enough. Thanks to Lemma \ref{I(x)} and Proposition \ref{inte ineq}, we have the following integral inequality in convex cones, which is useful in the proof of Liouville type theorem for subcritical anisotropic Finsler $p$-Laplacian equation in convex cones.
\begin{prop}\label{prop4.1}
Let $N\geq3$, $\mathcal{C}\subseteq\mathbb{R}^{N}$ be an open convex cone and $H\in C^{2}_{\text{loc}}(\mathbb{R}^{N}\setminus\{0\})$ is a Finsler norm such that $H^{2}$ is uniformly convex. Suppose that $1<p<N$ and $u$ be a positive weak solution of \eqref{eq:1.1}. Then for any test function $0\leq\varphi\in C^{\infty}_{c}(\mathbb{R}^N)$ and $d\in\mathbb{R}$, there holds
\begin{equation}\label{ineq4.1}
  \begin{aligned}
    &\quad \int_{\mathcal{C}}u^{1-p_*-d}\left\{Af(u)+\widehat{A}uf^{\prime}(u)\right\}H^p(\nabla u)\varphi\mathrm{d}x+B\int_{\mathcal{C}}u^{-p_*-d}H^{2p}(\nabla u)\varphi\mathrm{d}x\\
    &\leq \int_{\mathcal{C}}u^{1-p_*-d}\left\{Duf(u)+\widehat{D}H^{p}(\nabla u)\right\}{\bf{a}}(\nabla u)\cdot\nabla\varphi\mathrm{d}x+ \int_{\mathcal{C}}u^{2-p_*-d}{\bf{a}}(\nabla u)\cdot\nabla^2\varphi {\bf{a}}(\nabla u)\mathrm{d}x,
  \end{aligned}
\end{equation}
where $\nabla^2\varphi$ is the Hessian of $\varphi$ and
\begin{equation*}
  \begin{aligned}
    &A=\frac{N-1}{N}(1-\frac{d}{p_*})(p^*-1),\qquad\widehat{A}=-\frac{N-1}{N},\qquad B=\frac{p-1}{p}(1-d)d,\\
    &D=-\frac{N+1}{N},\qquad\widehat{D}=\left(2-\frac{p_*}{p}\right)-\left(2-\frac{1}{p}\right)d.
  \end{aligned}
\end{equation*}
\end{prop}

\begin{proof}
 By Lemma \ref{ident} and \eqref{omega}, we have
 $$\bm{\omega}=u^{2a+b}{\bf{a}}(\nabla u)\cdot\nabla \big({\bf{a}}(\nabla u)\big)+\frac{(N-1)a}{N}u^{2a+b-1}H^p(\nabla u){\bf{a}}(\nabla u)+\frac{1}{N}u^{2a+b}f(u){\bf{a}}(\nabla u),$$
and
\begin{equation*}
  \begin{aligned}
    &\quad -\int_{\mathcal{C}}u^{2a+b}\Big(({\bf{a}}(\nabla u)\cdot\nabla \big( {\bf{a}}(\nabla u)\big)\Big)\nabla\varphi\mathrm{d}x \\
    &=\int_{\mathcal{C}}({\bf{a}}(\nabla u)\cdot\nabla(u^{2a+b})){\bf{a}}(\nabla u)\cdot\nabla\varphi\mathrm{d}x\\
    &+\int_{\mathcal{C}}u^{2a+b}\left(\operatorname{div}\big({\bf{a}}(\nabla u)\big)({\bf{a}}(\nabla u)\cdot\nabla\varphi)+{\bf{a}}(\nabla u)\cdot\big(\nabla^2\varphi {\bf{a}}(\nabla u)\big)\right) \mathrm{d}x\\
    &+\int_{\partial\mathcal{C}}u^{2a+b}({\bf{a}}(\nabla u)\cdot\nabla\varphi)({\bf{a}}(\nabla u)\cdot\nu)\mathrm{d}\mathcal{H}^{N-1}\\
    &=(2a+b)\int_{\mathcal{C}}u^{2a+b-1}H^p(\nabla u){\bf{a}}(\nabla u)\cdot\nabla\varphi\mathrm{d}x\\
    &-\int_{\mathcal{C}}u^{2a+b}f( u){\bf{a}}(\nabla u)\cdot\nabla\varphi\mathrm{d}x+\int_{\mathcal{C}}u^{2a+b}{\bf{a}}(\nabla u)\cdot\big(\nabla^2\varphi {\bf{a}}(\nabla u)\big)\mathrm{d}x,
  \end{aligned}
\end{equation*}
where we have used the boundary condition ${\bf{a}}(\nabla u)\cdot\nu=0$ on $\partial\mathcal{C}$. Thus we derive that
\begin{equation*}
  \begin{aligned}
    - \int_{\mathcal{C}}\bm{\omega}(x)\cdot\nabla\varphi(x)\mathrm{d}x
    &=\left(\frac{N+1}{N}a+b\right)\int_{\mathcal{C}}u^{2a+b-1}H^p(\nabla u){\bf{a}}(\nabla u)\cdot\nabla\varphi\mathrm{d}x\\
    &-\frac{N+1}{N}\int_{\mathcal{C}}u^{2a+b}f( u){\bf{a}}(\nabla u)\cdot\nabla\varphi\mathrm{d}x+\int_{\mathcal{C}}u^{2a+b}{\bf{a}}(\nabla u)\cdot\big(\nabla^2\varphi {\bf{a}}(\nabla u)\big)\mathrm{d}x.
  \end{aligned}
\end{equation*}
It follows from Lemma \ref{I(x)} and Proposition \ref{key} that
$$- \int_{\mathcal{C}}\bm{\omega}(x)\cdot\nabla\varphi(x)\mathrm{d}x\geq\int_{\mathcal{C}}\psi(x)\varphi(x)\mathrm{d}x.$$
By using the boundary condition ${\bf{a}}(\nabla u)\cdot\nu=0$ on $\partial\mathcal{C}$ again, we conclude
\begin{equation*}
  \begin{aligned}
    \int_{\mathcal{C}}\operatorname{div}(u^{2a+b-1}H^p(\nabla u){\bf{a}}(\nabla u))\varphi\mathrm{d}x
    &=-\int_{\mathcal{C}}u^{2a+b-1}H^p(\nabla u){\bf{a}}(\nabla u)\cdot\nabla\varphi\mathrm{d}x\\
    &+\int_{\partial\mathcal{C}}u^{2a+b-1}H^p(\nabla u)({\bf{a}}(\nabla u)\cdot\nu)\varphi\mathrm{d}\mathcal{H}^{N-1}\\
    &=-\int_{\mathcal{C}}u^{2a+b-1}H^p(\nabla u){\bf{a}}(\nabla u)\cdot\nabla\varphi\mathrm{d}x.
  \end{aligned}
\end{equation*}
Thus we obtain the validity of \eqref{ineq4.1}.
\end{proof}
In what following,we  let $B_R=B_{R}(x_0)$ be the ball with radius $R$ and center $x_0$, such that $B_R\cap\mathcal{C}\neq\varnothing $.
\begin{lem}\label{lem4.2}
Let $N\geq3$, $\mathcal{C}\subseteq\mathbb{R}^{N}$ be an open convex cone and $H\in C^{2}_{\text{loc}}(\mathbb{R}^{N}\setminus\{0\})$ be a Finsler norm such that $H^{2}$ is uniformly convex. Let $u$ be a positive solution of \eqref{eq:1.1}. Suppose that
 $$R>0,\,\,d\in(0,1),\,\,k>2p$$
 and
 \begin{equation}\label{d}
  1<\alpha<p^*-\frac{p^*-1}{p_*}d.
 \end{equation}
 Then there exists a positive constant $C(n,p,d,k,\alpha)$ such that
 \begin{equation}\label{eq4.2}
  \int_{\mathcal{C}}\eta^kf^2(u)u^{2-d-p_*}\mathrm{d}x\leq CR^{-2p}\int_{\mathcal{C}}u^{\sigma-d}\eta^{k-2p}\mathrm{d}x,
 \end{equation}
 where $\sigma=2p-p_*$, and $\eta=\eta(\frac{|x-x_0|}{R})$ is a cut off function on the ball $B_{2R}$.
\end{lem}
\begin{proof}
Let $A,\widehat{A}, B,D,\widehat{D}$ be the coefficients in \eqref{ineq4.1}. Then by assumptions on $d$ and $\alpha$ we know that $B>0$ and
$$\delta:=\frac{N-1}{N}\left(p^*-\alpha-\frac{p^*-1}{p_*}d\right)>0.$$
Since $f$ is subcritical, it follows that
$$Af(u)+\widehat{A}uf^{\prime}(u)=\delta f(u)+\frac{N-1}{N}\left((\alpha-1)f(u)-uf^{\prime}(u)\right)\geq\delta f(u).$$
By letting $\varphi=\eta^k(\frac{|x-x_0|}{R})$, we deduce from \eqref{ineq4.1} that
\begin{equation}\label{eq4.3}
  \begin{aligned}
  \delta\int_{\mathcal{C}}u^{1-p_*-d}f(u)H^p(\nabla u)\eta^k\mathrm{d}x&+B\int_{\mathcal{C}}u^{-p_*-d}H^{2p}(\nabla u)\eta^k\mathrm{d}x \\
  &\leq \int_{\mathcal{C}}u^{1-p_*-d}\{Duf(u)+\widehat{D}H^p(\nabla u)\}{\bf{a}}(\nabla u)\cdot\nabla(\eta^k)\mathrm{d}x\\
  &+\int_{\mathcal{C}}u^{2-p_*-d}{\bf{a}}(\nabla u)\cdot\nabla^2(\eta^k){\bf{a}}(\nabla u)\mathrm{d}x.
  \end{aligned}
\end{equation}
By using Young's inequality with the respective exponents pairs
$$\left(\frac{p}{p-1},p\right),\,\,\left(\frac{2p}{2p-1},2p\right),\,\,\left(\frac{p}{p-1},p\right),$$
we can estimate the terms on the right side of \eqref{eq4.3} as follows:
$$\left\lvert \int_{\mathcal{C}}u^{2-p_*-d}{\bf{a}}(\nabla u)\cdot\nabla^2(\eta^k){\bf{a}}(\nabla u)\mathrm{d}x\right\rvert \leq \frac{B}{2}\int_{\mathcal{C}}u^{-p_*-d}H^{2p}(\nabla u)\eta^k\mathrm{d}x+C\int_{\mathcal{C}}u^{\sigma-d}|\nabla^2(\eta^k)|^p\eta^{(1-p)k}\mathrm{d}x,$$
\begin{equation*}
  \begin{aligned}
 \left\lvert\widehat{D} \int_{\mathcal{C}}u^{1-p_*-d}H^p(\nabla u){\bf{a}}(\nabla u)\cdot\nabla(\eta^k)\mathrm{d}x\right\rvert &\leq   \frac{B}{2}\int_{\mathcal{C}}u^{-p_*-d}H^{2p}(\nabla u)\eta^k\mathrm{d}x\\
 &+ C\int_{\mathcal{C}}u^{\sigma-d}|\nabla(\eta^k)|^{2p}\eta^{(1-2p)k}\mathrm{d}x,
  \end{aligned}
\end{equation*}
and
\begin{equation*}
  \begin{aligned}
 \left\lvert D \int_{\mathcal{C}}u^{2-p_*-d}f(u){\bf{a}}(\nabla u)\cdot\nabla(\eta^k)\mathrm{d}x\right\rvert &\leq   \frac{\delta}{2}\int_{\mathcal{C}}u^{1-p_*-d}f(u)H^{p}(\nabla u)\eta^k\mathrm{d}x\\
 &+ C\int_{\mathcal{C}}u^{1-d+p-p_*}|\nabla(\eta^k)|^{p}\eta^{(1-p)k}f(u)\mathrm{d}x,
  \end{aligned}
\end{equation*}
where $C=C(n,p,d,\delta)>0$. Therefore, by using the estimates $|\eta|\leq1$, $|\nabla\eta|\leq\frac{2}{R}$ and $|\nabla^2\eta|\leq\frac{C}{R^2}$, we can derive from \eqref{eq4.3} that
\begin{equation}\label{eq4.4}
 \int_{\mathcal{C}}u^{1-p_*-d}f(u)H^p(\nabla u)\eta^k\mathrm{d}x\leq CR^{-2p} \int_{\mathcal{C}}u^{\sigma-d}\eta^{k-2p}\mathrm{d}x+CR^{-p}\int_{\mathcal{C}}u^{1-d+p-p_*}\eta^{k-p}f(u)\mathrm{d}x,
\end{equation}
where $C=C(n,p,d,\delta,k)>0$, and the condition $k>2p$ is used to guarantee the integrals well-defined. Next, we choose $\varphi=\eta^kf(u)u^{2-d-p_*}$ as a test function in \eqref{eq:1.1} to get
\begin{equation}\label{eq4.5}
  \int_{\mathcal{C}}u^{2-p_*-d}f^2(u)\eta^k\mathrm{d}x=\int_{\mathcal{C}}(u^{2-p_*-d}f(u))^{\prime}H^p(\nabla u)\eta^k\mathrm{d}x+k\int_{\mathcal{C}}\eta^{k-1}u^{2-p_*-d}f(u){\bf{a}}(\nabla u)\cdot\nabla\eta\mathrm{d}x.
\end{equation}
By using the Young's inequality with exponent pair $(\frac{p}{p-1},p)$, one has
$$\left\lvert \int_{\mathcal{C}}\eta^{k-1}u^{2-p_*-d}f(u){\bf{a}}(\nabla u)\cdot\nabla\eta\mathrm{d}x\right\rvert\leq \int_{\mathcal{C}}\eta^{k}u^{1-p_*-d}f(u)\left\{H^p(\nabla u)+\frac{u^p}{(R\eta)^p}\right\} \mathrm{d}x. $$
Moreover, since $\alpha<p^*$, we get
\begin{equation*}
  \begin{aligned}
    (u^{2-p_*-d}f(u))^{\prime}&=u^{1-p_*-d}(uf^{\prime}(u)+(2-d-p_*)f(u))\\
    &\leq (\alpha+1-d-p_*)u^{1-p_*-d}f(u)\leq \frac{N}{N-p}u^{1-p_*-d}f(u).
  \end{aligned}
\end{equation*}
Thus it follows from \eqref{eq4.5} that
$$\int_{\mathcal{C}}u^{2-p_*-d}f^2(u)\eta^k\mathrm{d}x\leq C\int_{\mathcal{C}}\eta^{k}u^{1-p_*-d}f(u)H^p(\nabla u)\mathrm{d}x+CR^{-p}\int_{\mathcal{C}}\eta^{k-p}f(u)u^{1+p-d-p_*}\mathrm{d}x.$$
Then combining with \eqref{eq4.4} we deduce
\begin{equation}\label{eq4.6}
 \int_{\mathcal{C}}u^{2-p_*-d}f^2(u)\eta^k\mathrm{d}x\leq CR^{-2p}\int_{\mathcal{C}}\eta^{k-2p}u^{\sigma-d}\mathrm{d}x+CR^{-p}\int_{\mathcal{C}}\eta^{k-p}f(u)u^{1+p-d-p_*}\mathrm{d}x.
\end{equation}
The Cauchy inequality gives that
\begin{equation}\label{b1}
  R^{-p}\int_{\mathcal{C}}\eta^{k-p}f(u)u^{1+p-d-p_*}\mathrm{d}x\leq\frac{1}{2} \int_{\mathcal{C}}u^{2-p_*-d}f^2(u)\eta^k\mathrm{d}x+CR^{-2p}\int_{\mathcal{C}}\eta^{k-2p}u^{\sigma-d}\mathrm{d}x.
\end{equation}
By using \eqref{b1} to eliminate the second integral on the right side in \eqref{eq4.6}, we obtain the validity of \eqref{eq4.2}. Thus we complete the proof of Lemma \ref{lem4.2}.
\end{proof}

\begin{lem}\label{lem4.4}
Let $N\geq3$, $\mathcal{C}\subseteq\mathbb{R}^{N}$ be an open convex cone and $H\in C^{2}_{\text{loc}}(\mathbb{R}^{N}\setminus\{0\})$ be a Finsler norm such that $H^{2}$ is uniformly convex. Let $u$ be a positive weak solution of \eqref{eq:1.1}. Assume that $N\in(p,2p^2-p)$. Let
$$\varrho=\min\left\{2,\frac{2N(p-1)}{2p^2-p-N}\right\} .$$
Then for all $\theta\in(0,\varrho)$, there exists $C(N,\theta)>0$ such that
\begin{equation}\label{eq4.8}
||f(u)u^{1-p}||_{L^{\theta}(B_R\cap\mathcal{C})}\leq CR^{\frac{N}{\theta}-p}.
\end{equation}
\end{lem}
\begin{proof}
Choose $d$ small enough such that \eqref{d} is valid. Then by using the H\"{o}lder's inequality with exponent pair $(\frac{2}{\theta},\frac{2}{2-\theta})$, combining with \eqref{eq4.2}, we derive
\begin{equation}\label{eq4.9}
  \begin{aligned}
    \int_{B_R\cap\mathcal{C}}(f(u)u^{1-p})^{\theta}\mathrm{d}x
    &\leq \left(\int_{B_R\cap\mathcal{C}}f^2(u)u^{2-d-p_*}\mathrm{d}x\right)^{\frac{\theta}{2}} \left(\int_{B_R\cap\mathcal{C}}u^{\frac{\theta(d-\sigma)}{2-\theta}}\mathrm{d}x\right)^{\frac{2-\theta}{2}}\\
    &\leq \left(CR^{-2p}\int_{B_{2R}\cap\mathcal{C}}u^{\sigma-d}\mathrm{d}x\right)^{\frac{\theta}{2}} \left(\int_{B_R\cap\mathcal{C}}u^{\frac{\theta(d-\sigma)}{2-\theta}}\mathrm{d}x\right)^{\frac{2-\theta}{2}}.
  \end{aligned}
\end{equation}
Then there are now two cases.

\smallskip

\noindent {\bf Case 1.} $N\in(p,2p-1]$. Then it follows from $N\leq2p-1$ that
$$\sigma\leq0\qquad\text{and}\qquad\varrho=\frac{2N(p-1)}{2p^2-p-N}.$$
Then we can derive from $\theta<\varrho\leq2$ that
$$0\leq-\frac{\sigma\theta}{2-\theta}<\frac{N(p-1)}{N-p}.$$
Next, choose $d>0$ small enough such that
$$0<\gamma:=\frac{\theta(d-\sigma)}{2-\theta}<\frac{N(p-1)}{N-p},$$
then by using the fact that $\sigma-d<0$ and Lemma \ref{Hanack}, we obtain
\begin{equation}\label{eq4.10}
  \begin{aligned}
  \int_{B_{2R}\cap\mathcal{C}}u^{\sigma-d}\mathrm{d}x
  &\leq\int_{B_{2R}\cap\mathcal{C}} \left[CR^{-\frac{N}{\gamma}}\left(\int_{B_{4R}\cap\mathcal{C}}u^{\gamma}\mathrm{d}x\right) ^{\frac{1}{\gamma}}\right] ^{\sigma-d}\mathrm{d}x\\
  &=CR^{\frac{2N}{\theta}}\left(\int_{B_{4R}\cap\mathcal{C}}u^{\frac{\theta(d-\sigma)}{2-\theta}}\mathrm{d}x\right) ^{-\frac{2-\theta}{\theta}}.
  \end{aligned}
\end{equation}
Then \eqref{eq4.9} and \eqref{eq4.10} yield the validity of \eqref{eq4.8}.

\smallskip

\noindent {\bf Case 2.} $N\in(2p-1,2p^2-p)$. In this case we have $\sigma\in(0,p_*-1)$, then we can choose $d>0$ small enough such that
$$\gamma:=\sigma-d\in(0,p_*-1).$$
Note that $\frac{\theta(d-\sigma)}{2-\theta}<0$, thus we get
\begin{equation}\label{eq4.11}
  \begin{aligned}
    \int_{B_R\cap\mathcal{C}}u^{\frac{\theta(d-\sigma)}{2-\theta}}\mathrm{d}x
  &\leq\int_{B_{R}\cap\mathcal{C}} \left[CR^{-\frac{N}{\gamma}}\left(\int_{B_{2R}\cap\mathcal{C}}u^{\gamma}\mathrm{d}x\right) ^{\frac{1}{\gamma}}\right] ^{\frac{\theta(d-\sigma)}{2-\theta}}\mathrm{d}x\\
  &=CR^{\frac{2N}{2-\theta}}\left(\int_{B_{2R}\cap\mathcal{C}}u^{\sigma-d}\mathrm{d}x\right) ^{-\frac{\theta}{2-\theta}}.
  \end{aligned}
\end{equation}
Then \eqref{eq4.9} and \eqref{eq4.11} give the validity of \eqref{eq4.8}.
\end{proof}
To deal with the cases of $N=2,\ p<\frac{4}{3}$ and $N\geq3,\ N\geq2p-\frac{1}{2}$, we need following lemma.
\begin{lem}\label{lem4.5}\cite[Lemma 3.4]{SZ}
If either
$$N=2,\ p<\frac{4}{3}\qquad\text{or}\qquad N\geq3,\ N\geq2p-\frac{1}{2},$$
then \\
$(\rm{i})$ $\sigma>\frac{p}{N}$,\\
$(\rm{ii})$ $(\alpha-p)N<(2\alpha-p_*-\frac{p}{N})p$ for all $\alpha\in[p,p^*]$.
\end{lem}
\noindent {\bf Proof of Theorem \ref{Liou}.}
We suppose for contradiction that there exists a nontrivial nonnegative bounded solution $u$ of \eqref{eq:1.1}, then Lemma \ref{Hanack} yields that $u>0$ in $\mathcal{C}$.
Let $M:=\sup\limits_{x\in\mathcal{C}}u$, then it follows from \eqref{subcritical} that
\begin{equation}\label{eq f}
  f(u)\geq\frac{f(M)}{M^{\alpha-1}}u^{\alpha-1}\equiv\lambda u^{\alpha-1},
\end{equation}
where $\lambda=\frac{f(M)}{M^{\alpha-1}}>0$ since $f(t)>0$ for any $t>0$.

If $1<\alpha< p$, we deduce from \eqref{eq:1.1} and \eqref{eq f} that $u$ satisfies the inequality
\begin{equation*}
  \left\{
      \begin{aligned}
      &-\operatorname{div}\left({\bf{a}}(\nabla u)\right)\geq\lambda u^{\alpha-1} \quad\,\,\, &{\rm{in}} \,\, \mathcal{C}, \\
      &{\bf{a}}(\nabla u)\cdot \nu =0 \quad\,\,\, &{\rm{on}} \,\, \partial\mathcal{C}.
      \end{aligned}
      \right.
\end{equation*}
By letting $u=\lambda^{p-\alpha}w$, we obtain that $w$ is a positive solution of
\begin{equation*}
  \left\{
      \begin{aligned}
      &-\operatorname{div}\left({\bf{a}}(\nabla w)\right)\geq w^{\alpha-1} \quad\,\,\, &{\rm{in}} \,\, \mathcal{C}, \\
      &{\bf{a}}(\nabla w)\cdot \nu =0 \quad\,\,\, &{\rm{on}} \,\, \partial\mathcal{C}.
      \end{aligned}
      \right.
\end{equation*}
Lemma \ref{lem2.10} derives that $w\equiv0$, a contradiction.

If $\alpha=p$, let $\varphi=u^{1-p}\eta^{\gamma}$ as a test function, where $\gamma>p$ is a constant, $\eta(x)$ is a cut-off function with $0\leq\eta(x)\leq1$, $\eta(x)=1$ in $B_R(0)$, $\eta(x)=0$ in $B^c_{2R}(0)$ and $|\nabla\eta|\leq\frac{C}{R}$ for some constant $C>0$. Then we obtain
\begin{equation*}
  \int_{B_{2R}\cap\mathcal{C}}\eta^{\gamma}\mathrm{d}x+(p-1)\int_{B_{2R}\cap\mathcal{C}}u^{-p}\eta^{\gamma}{\mathbf{a}}(\nabla u)\cdot\nabla u\mathrm{d}x\leq\gamma\int_{B_{2R}\cap\mathcal{C}}u^{1-p}\eta^{\gamma-1}{\mathbf{a}}(\nabla u)\cdot\nabla \eta\mathrm{d}x.
\end{equation*}
By using the Young's inequality with exponents pair $(\frac{p}{p-1},p)$ to the right side, we have
$$\int_{B_{2R}\cap\mathcal{C}}u^{1-p}\eta^{\gamma-1}{\mathbf{a}}(\nabla u)\cdot\nabla \eta\mathrm{d}x\leq \varepsilon\int_{B_{2R}\cap\mathcal{C}}u^{-p}|\nabla u|^p\eta^{\gamma}+CR^{-p}\int_{B_{2R}\cap\mathcal{C}}\eta^{\gamma-p}\mathrm{d}x.$$
Since ${\bf{a}}(\nabla u)\cdot\nabla u=H^p(\nabla u)\geq c_{H}^p|\nabla u|^p$, and choose $\varepsilon>0$ small enough, we deduce
$$\int_{B_{2R}\cap\mathcal{C}}\eta^{\gamma}\mathrm{d}x\leq CR^{-p}\int_{B_{2R}\cap\mathcal{C}}\eta^{\gamma-p}\mathrm{d}x,$$
which gives
$$CR^N\leq CR^{N-p},$$
this derives contradiction as $R\rightarrow+\infty$.

In what following, we assume that $p<\alpha<p^*$. If either $N=2,\,p\geq\frac{4}{3}$ or $N\in[3,2p),\,p>\frac{3}{2}$, then it is easy to check that $N<2p^2-p$. Then Lemma \ref{lem4.4} gives that \eqref{eq4.8} holds for
all $\theta\in(0,\varrho)$. By a simple calculation we have $\frac{N}{p}<\varrho$, then there exists a number $\theta\in(\frac{N}{p},\varrho)$. For such choice of $\theta$, by letting $R\rightarrow+\infty$ we arrive at
$$\left\lVert f(u)u^{1-p}\right\rVert _{L^{\theta}(\mathcal{C})}=0.$$
This contradicts the conditions $f(t)>0$ for all $t>0$ and $u>0$ in $\mathcal{C}$. Then we know that $u\equiv0$.

If either $N=2,\,p<\frac{4}{3}$ or $N\geq3,\,N\geq 2p-\frac{1}{2}$. Define
$$\sigma=2p-p_*,\qquad\tau=2\alpha-p_*,$$
then obviously
$$\tau-\sigma=2(\alpha-p)>0.$$
Note that Lemma \ref{lem4.5} (i) implies
$$\sigma-d>0$$
for all $d<\frac{p}{N}$. Let $k>2p$, by Lemma \ref{lem4.2} we get
$$ \int_{\mathcal{C}}\eta^kf^2(u)u^{2-d-p_*}\mathrm{d}x\leq CR^{-2p}\int_{\mathcal{C}}u^{\sigma-d}\eta^{k-2p}\mathrm{d}x.$$
By \eqref{eq f} and the Young's inequality with exponent pair $(\frac{\tau-d}{\sigma-d},\frac{\tau-d}{\tau-\sigma})$, we have
$$(\eta R)^{-2p}u^{\sigma-d}\leq \varepsilon f^2(u)u^{2-d-p_*}+C(\eta R)^{-p\frac{\tau-d}{\alpha-p}}.$$
Then we obtain
$$\int_{B_R\cap\mathcal{C}}f^2(u)u^{2-d-p_*}\mathrm{d}x\leq CR^{N-p\frac{\tau-d}{\alpha-p}}.$$
Note that Lemma \ref{lem4.5} (ii) implies that
$$N-p\frac{\tau-d}{\alpha-p}<0.$$
Then by letting $R\rightarrow+\infty$, we arrive at
$$\int_{\mathcal{C}}f^2(u)u^{2-d-p_*}\mathrm{d}x=0,$$
which contradicts $f(t)>0$ for all $t>0$.

Thus we complete the proof of Theorem \ref{Liou}.  \qquad\qquad\qquad\quad$\qed$

\smallskip

In order to prove Theorem \ref{liou2}, we need the following decay estimate in convex cones.
\begin{lem}\label{decay}
Let $N\geq3$, $\mathcal{C}\subseteq\mathbb{R}^{N}$ be an open convex cone with vertex set $\mathcal{V}$ and $H\in C^{2}_{\text{loc}}(\mathbb{R}^{N}\setminus\{0\})$ be a Finsler norm such that $H^{2}$ is uniformly convex. Let $f(u)=u^{q}$, $p-1<q<p^*-1$ and $\Omega$ be an arbitrary bounded domain of $\mathbb{R}^N$. Then there exists some constant $C=C(p,N,q,\mathcal{C})>0$ (independent of $\Omega$ and $u$) such that for
any nonnegative solution $u$ of \eqref{eq:1.1} in $\Omega\cap\mathcal{C}$, there holds
\begin{equation}\label{eq decay}
  u(x)\leq C\mathrm{dist}(x,\partial\Omega\cap\mathcal{C})^{-\frac{p}{q+1-p}},\qquad x\in\Omega\cap\mathcal{C}.
\end{equation}
\end{lem}
\begin{proof}
We will prove \eqref{eq decay} by contradiction. Assume that there exist sequence $u_k$ and $x_k\in\Omega_k\cap\mathcal{C}$ such that $u_k$ solves \eqref{eq:1.1} with $f(u)=u^{q}$ on $\Omega_k\cap\mathcal{C}$ and
\begin{equation}\label{eq4.14}
  u_k(x_k)\mathrm{dist}(x_k,\partial\Omega_k\cap\mathcal{C})^{\frac{p}{q+1-p}}\rightarrow+\infty,\qquad\text{as}\,\,k\rightarrow+\infty.
\end{equation}
Denote $d_k=\mathrm{dist}(x_k,\partial\Omega_k\cap\mathcal{C})$, and define
$$h_k(x):=\left(\frac{d_k}{2}-|x-x_k|\right) ^{\frac{p}{q+1-p}}u_k(x)\quad\text{in}\,\, B_{\frac{d_k}{2}}(x_k)\cap\mathcal{C}.$$
Then choose $\widehat{x}_k\in\overline{\mathcal{C}}$ such that $|\widehat{x}_k-x_k|<\frac{d_k}{2}$ such that
$$h_k(\widehat{x}_k)=\max\limits_{x\in\overline{\mathcal{C}},|x-x_k|\leq\frac{d_k}{2}}h_k(x).$$
Let
$$2\rho_k:=\frac{d_k}{2}-|\widehat{x}_k-x_k|,$$
then we have
$$2\rho_k\in\bigg(0,\frac{d_k}{2}\bigg]  \quad\text{and}\quad\frac{d_k}{2}-|x-x_k|\geq\rho_k,\qquad\text{for}\,\, x\in B_{\rho_k}(\widehat{x}_k)\cap\mathcal{C}.$$
Moreover, there holds
$$(2\rho_k)^{\frac{p}{q+1-p}}u_k(\widehat{x}_k)=h_k(\widehat{x}_k)\geq h_k(x)\geq (\rho_k)^{\frac{p}{q+1-p}}u_k(x),\qquad\text{for}\,\, x\in B_{\rho_k}(\widehat{x}_k)\cap\mathcal{C}.$$
This yields
$$2^{\frac{p}{q+1-p}}u_k(\widehat{x}_k)\geq u_k(x)\quad\text{in}\,\,B_{\rho_k}(\widehat{x}_k)\cap\mathcal{C}.$$
Also we derive from \eqref{eq4.14} that
$$(2\rho_k)^{\frac{p}{q+1-p}}u_k(\widehat{x}_k)=h_k(\widehat{x}_k)\geq h_k(x_k)=\left(\frac{d_k}{2}\right)^{\frac{p}{q+1-p}}u_k(x_k)\rightarrow+\infty,\qquad\text{as}\,\,k\rightarrow+\infty. $$
Define
$$v_k(z):=\frac{1}{u_k(\widehat{x}_k)}u_k\left(\widehat{x}_k+\frac{z}{u_k^{\frac{q+1-p}{p}}(\widehat{x}_k)}\right),\quad z\in B_{R_k}(0)\cap\mathcal{C}_k, $$
where
$$R_k=\rho_ku_k^{\frac{q+1-p}{p}}(\widehat{x}_k)\rightarrow+\infty,\quad\text{as}\,\,k\rightarrow+\infty,$$
and
$$\mathcal{C}_k=u_k^{\frac{q+1-p}{p}}(\widehat{x}_k)(\mathcal{C}-\widehat{x}_k):=\{y\in\mathbb{R}^N\,|\,u_k^{-\frac{q+1-p}{p}}(\widehat{x}_k)y+\widehat{x}_k\in\mathcal{C}\}.$$
Moreover, it follows from \eqref{eq:1.1} that $v_k$ is a positive solution to
\begin{equation}\label{eq4.15}
    \left\{
        \begin{aligned}
        &-\operatorname{div}\left(a(\nabla v_k)\right)=v_k^{q} \quad\,\,\, &{\rm{in}} \,\, B_{R_k}(0)\cap\mathcal{C}_k, \\
        &v_k(0)=1,\\
        &a(\nabla v_k)\cdot \nu =0 \quad\,\,\, &{\rm{on}} \,\, \partial\mathcal{C}_k\cap B_{R_k}(0),\\
        &v_k(z)\leq 2^{\frac{p}{q+1-p}} \quad\,\,\, &{\rm{in}} \,\, B_{R_k}(0)\cap\mathcal{C}_k.
        \end{aligned}
        \right.
\end{equation}
Next, we note that, in fact, there holds $\mathcal{C}_k=\mathcal{C}-u_k^{\frac{q+1-p}{p}}(\widehat{x}_k)\widehat{x}_k$ since $\lambda\mathcal{C}=\mathcal{C}$ for any $\lambda>0$. Thus each $\mathcal{C}_k$ is a translation of $\mathcal{C}$
with its vertex set shifted to $\mathcal{V}-u_k^{\frac{q+1-p}{p}}(\widehat{x}_k)\widehat{x}_k$, hence congruent to $\mathcal{C}$ as a convex cone. In what following, there are two possible cases.

\smallskip

\noindent {\bf Case 1.} $u_k^{\frac{q+1-p}{p}}(\widehat{x}_k)|\widehat{x}_k|$ is bounded. In this case, after up to a subsequence, we may assume that $u_k^{\frac{q+1-p}{p}}(\widehat{x}_k)\widehat{x}_k\rightarrow x_{\infty}$ for some $x_{\infty}\in\mathbb{R}^N$
and $\mathcal{C}_k\rightarrow\mathcal{C}_{\infty}$ as $k\rightarrow+\infty$, where $\mathcal{C}_{\infty}$ is a convex cone. Since $R_k=\rho_ku_k^{\frac{q+1-p}{p}}(\widehat{x}_k)\rightarrow+\infty$, we have $B_{R_k}(0)\cap\mathcal{C}_k\rightarrow\mathcal{C}_{\infty}$ as $k\rightarrow+\infty$.
Then we consider a ball $B_R(x_{\infty})$ with any fixed radius $R>0$. For every compact set $K\subset\subset B_R(x_{\infty})\cap\mathcal{C}_{\infty}$, there exists $\overline{k}\in\mathbb{N}$ such that $K\subset\mathcal{C}_k\cap B_R(x_{\infty})$ for every $k\geq\overline{k}$.
 Then for every $k\geq\overline{k}$, $v_k$ is a bounded solution of \eqref{eq4.15} in $K$. By the regularity estimates for quasilinear elliptic equations of \cite{Dib}, there exist a constant $C>0$ and a real number $\theta\in(0,1)$ such that
 $$\left\lVert v_k\right\rVert _{C^{1,\theta}(K^{\prime})}\leq C $$
 for any $k\geq\overline{k}$ and $K^{\prime}\subset\subset K$. Then it follows from the regularity estimates for quasilinear elliptic equations of \cite{Per}, $v_k\in C^{0,\theta}(B_R(x_{\infty})\cap\overline{\mathcal{C}_{\infty}})$ uniformly for any $R>0$. By the arbitrariness of $R>0$, the Ascoli--Arzel$\acute{\text{a}} $'s Theorem and a diagonal process, we derive that, up to
 a subsequence, $\{v_k\}$ converges to some function $v_{\infty}$ in $C^0(B_R(x_{\infty})\cap\overline{\mathcal{C}_{\infty}})\cap C^1_{loc}(B_R(x_{\infty})\cap\mathcal{C}_{\infty})$ for any $R>0$. By letting $R\rightarrow+\infty$ in \eqref{eq4.15}, we deduce that $v_{\infty}\in W^{1,p}_{loc}(\overline{\mathcal{C}_{\infty}})$ is a nonnegative weak solution of
\begin{equation*}
    \left\{
        \begin{aligned}
        &-\operatorname{div}\left(a(\nabla v_{\infty})\right)=v_{\infty}^{q} \quad\,\,\, &{\rm{in}} \,\, \mathcal{C}_{\infty}, \\
        &v_{\infty}(0)=1,\\
        &a(\nabla v_{\infty})\cdot \nu =0 \quad\,\,\, &{\rm{on}} \,\, \partial\mathcal{C}_{\infty},\\
        &v_{\infty}(z)\leq 2^{\frac{p}{q+1-p}} \quad\,\,\, &{\rm{in}} \,\, \mathcal{C}_{\infty},
        \end{aligned}
        \right.
\end{equation*}
this contradicts the conclusion of Theorem \ref{Liou}.

\smallskip

\noindent {\bf Case 2.} $u_k^{\frac{q+1-p}{p}}(\widehat{x}_k)|\widehat{x}_k|$ is unbounded. By taking a subsequence, we may assume that
$$u_k^{\frac{q+1-p}{p}}(\widehat{x}_k)|\widehat{x}_k|\rightarrow+\infty,\quad \text{as}\quad k\rightarrow+\infty.$$
 In this case, we have  $B_{R_k}(0)\cap\mathcal{C}_k\rightarrow\mathbb{R}^N$ as $k\rightarrow+\infty$, since $R_k=\rho_ku_k^{\frac{q+1-p}{p}}(\widehat{x}_k)\rightarrow+\infty$. Then for any fixed ball $B_R$ in $\mathbb{R}^N$, there exists $\overline{k}\in\mathbb{N}$ such that
 $\mathcal{C}_k\cap B_R=B_R$ for each $k\geq\overline{k}$. And also, $v_k$ is a positive weak solution of \eqref{eq4.15} in $B_R$. Similar to the case 1, it follows from the regularity estimates for quasilinear elliptic equations of \cite{Dib}, there exists a constant $C>0$ and a real number $\theta\in(0,1)$ such that
 $$\left\lVert v_k\right\rVert _{C^{1,\theta}(B_{\frac{R}{2}})}\leq C $$
 for any $k\geq\overline{k}$. Since $R>0$ is arbitrary, then by using the Ascoli--Arzel$\acute{\text{a}} $'s Theorem and a diagonal process, we derive that, up to
 a subsequence, $\{v_k\}$ converges to some function $v_{\infty}$ in $ C^1_{loc}(\mathbb{R}^N)$. By letting $R\rightarrow+\infty$ in \eqref{eq4.15}, we deduce that $v_{\infty}\in W^{1,p}_{loc}(\mathbb{R}^N)$ is a weak solution of
\begin{equation*}
    \left\{
        \begin{aligned}
        &-\operatorname{div}\left(a(\nabla v_{\infty})\right)=v_{\infty}^{q} \quad\,\,\, &{\rm{in}} \,\, \mathbb{R}^N, \\
        &v_{\infty}(0)=1,\\
        &v_{\infty}(z)\leq 2^{\frac{p}{q+1-p}} \quad\,\,\, &{\rm{in}} \,\, \mathbb{R}^N,
        \end{aligned}
        \right.
\end{equation*}
this contradicts the conclusion of Theorem \ref{Liou}.

Thus we complete the proof of Lemma \ref{decay}.
\end{proof}

\noindent {\bf Proof of Theorem \ref{liou2}.} Let $u$ be a nonnegative weak solution of \eqref{eq:1.1} with $f(u)=u^{q}$, and $u\not\equiv0$ in $\mathcal{C}$. Then Lemma \ref{Hanack} implies $u>0$ in $\mathcal{C}$.

Firstly, we consider $p-1<q<p^*-1$. For any $R>0$, let $\Omega=B_R$. Then for any fixed $x\in\mathcal{C}\cap B_R$, by Lemma \ref{decay} we have
$$u(x)\leq C\operatorname{dist}(x,\partial B_R\cap\mathcal{C})^{-\frac{p}{q+1-p}},$$
this gives that $u\equiv0$ in $\mathcal{C}$ by letting $R\rightarrow+\infty$.

In what following, we assume that $0<q\leq p-1$. We consider two cases.

\smallskip

\noindent {\bf Case 1.} If $0<q<p-1$. In this case, we choose parameters $\beta<1-p$ and $\gamma>1$ large enough which values to be determined later. Let $\varphi=u^{\beta}\eta^{\gamma}$ as a test function in \eqref{def weak solu}, where $\eta(x)$ is a cut-off function with $0\leq\eta(x)\leq1$, $\eta(x)=1$ in $B_R(0)$, $\eta(x)=0$ in $B^c_{2R}(0)$ and $|\nabla\eta|\leq\frac{C}{R}$ for some constant $C>0$.
Then we have
\begin{equation*}
  \begin{aligned}
    \int_{B_{2R}\cap\mathcal{C}}u^{q+\beta}\eta^{\gamma}\mathrm{d}x
    &=\int_{B_{2R}\cap\mathcal{C}}{\mathbf{a}}(\nabla u)\cdot\nabla(u^{\beta}\eta^{\gamma})\mathrm{d}x+\int_{\partial(B_{2R}\cap\mathcal{C})}u^{\beta}\eta^{\gamma}{\mathbf{a}}(\nabla u)\cdot\nu\mathrm{d}\mathcal{H}^{N-1}\\
    &=\beta\int_{B_{2R}\cap\mathcal{C}}u^{\beta-1}\eta^{\gamma}{\mathbf{a}}(\nabla u)\cdot\nabla u\mathrm{d}x+\gamma\int_{B_{2R}\cap\mathcal{C}}u^{\beta}\eta^{\gamma-1}{\mathbf{a}}(\nabla u)\cdot\nabla \eta\mathrm{d}x,
  \end{aligned}
\end{equation*}
where we have used the boundary condition ${\bf{a}}(\nabla u)\cdot\nu=0$ on $\partial\mathcal{C}$ and $\eta=0$ on $\partial B_{2R}(0)$. This gives that
\begin{equation}\label{eq4.18}
  \int_{B_{2R}\cap\mathcal{C}}u^{q+\beta}\eta^{\gamma}\mathrm{d}x-\beta\int_{B_{2R}\cap\mathcal{C}}u^{\beta-1}\eta^{\gamma}{\mathbf{a}}(\nabla u)\cdot\nabla u\mathrm{d}x=\gamma\int_{B_{2R}\cap\mathcal{C}}u^{\beta}\eta^{\gamma-1}{\mathbf{a}}(\nabla u)\cdot\nabla \eta\mathrm{d}x.
\end{equation}
By the Young's inequality with the exponents pair $(\frac{p}{p-1},p)$, we obtain
\begin{equation*}\label{eq4.19}
  \begin{aligned}
    \int_{B_{2R}\cap\mathcal{C}}u^{\beta}\eta^{\gamma-1}{\mathbf{a}}(\nabla u)\cdot\nabla \eta\mathrm{d}x
    &\leq \varepsilon\int_{B_{2R}\cap\mathcal{C}}u^{\beta-1}\eta^{\gamma}|{\mathbf{a}}(\nabla u)|^{\frac{p}{p-1}}\mathrm{d}x+C_{\varepsilon}\int_{B_{2R}\cap\mathcal{C}}u^{\beta+p-1}\eta^{\gamma-p}|\nabla \eta|^p\mathrm{d}x\\
    &\leq C\varepsilon\int_{B_{2R}\cap\mathcal{C}}u^{\beta-1}\eta^{\gamma}|\nabla u|^{p}\mathrm{d}x+C_{\varepsilon}R^{-p}\int_{B_{2R}\cap\mathcal{C}}u^{\beta+p-1}\eta^{\gamma-p}\mathrm{d}x,
  \end{aligned}
\end{equation*}
where we have used the fact that $|{\bf{a}}(\nabla u)|\leq C|\nabla u|^{p-1}$ and $|\nabla \eta|\leq \frac{C}{R}$ in the last setp. Note that ${\bf{a}}(\nabla u)\cdot\nabla u=H^p(\nabla u)\geq c_{H}^p|\nabla u|^p$, and choose $\varepsilon>0$ small enough such that
$C\varepsilon=\frac{-\beta}{2}$, we have
$$\int_{B_{2R}\cap\mathcal{C}}u^{q+\beta}\eta^{\gamma}\mathrm{d}x+\frac{-\beta}{2}\int_{B_{2R}\cap\mathcal{C}}u^{\beta-1}\eta^{\gamma}|\nabla u|^{p}\mathrm{d}x\leq C_{\varepsilon}R^{-p}\int_{B_{2R}\cap\mathcal{C}}u^{\beta+p-1}\eta^{\gamma-p}\mathrm{d}x,$$
which yields that
$$\int_{B_{2R}\cap\mathcal{C}}u^{q+\beta}\eta^{\gamma}\mathrm{d}x\leq C_{\varepsilon}R^{-p}\int_{B_{2R}\cap\mathcal{C}}u^{\beta+p-1}\eta^{\gamma-p}\mathrm{d}x.$$
Since $\beta<1-p$ and $q<p-1$, we have $\frac{\beta+q}{\beta+p-1}>1$. By using the Young's inequality with exponents pair $(\frac{\beta+q}{\beta+p-1}, \frac{\beta+q}{q+1-p})$, we deduce
\begin{equation*}
 \begin{aligned}
  \int_{B_{2R}\cap\mathcal{C}}u^{q+\beta}\eta^{\gamma}\mathrm{d}x
  &\leq C_{\varepsilon}R^{-p}\int_{B_{2R}\cap\mathcal{C}}u^{\beta+p-1}\eta^{\gamma-p}\mathrm{d}x\\
  &\leq \varepsilon \int_{B_{2R}\cap\mathcal{C}}u^{q+\beta}\eta^{\gamma}\mathrm{d}x+C_{\varepsilon}R^{-\frac{(\beta+q)p}{(q+1-p)}}\int_{B_{2R}\cap\mathcal{C}}\eta^{\gamma-\frac{(\beta+q)p}{(q+1-p)}}\mathrm{d}x,
 \end{aligned}
\end{equation*}
and thus
\begin{equation}\label{eq4.20}
    \int_{B_{R}\cap\mathcal{C}}u^{q+\beta}\mathrm{d}x\leq CR^{-\frac{(\beta+q)p}{(q+1-p)}}\int_{B_{2R}\cap\mathcal{C}}\eta^{\gamma-\frac{(\beta+q)p}{(q+1-p)}}\mathrm{d}x.
\end{equation}
Next, we choose $\gamma>\frac{(\beta+q)p}{(q+1-p)}$, therefore \eqref{eq4.20} yields
$$\int_{B_{R}\cap\mathcal{C}}u^{q+\beta}\mathrm{d}x\leq CR^{N-\frac{(\beta+q)p}{(q+1-p)}},$$
then choose $\beta<\min\{1-p,\frac{N(q+1-p)}{p}-q\}$ such that $N-\frac{(\beta+q-1)p}{(q+1-p)}<0$, and let $R\rightarrow+\infty$, we obtain
$$\int_{\mathcal{C}}u^{q+\beta}\mathrm{d}x\leq0,$$
which contradicts $u>0$ in $\mathcal{C}$.

\smallskip

\noindent {\bf Case 2.} If $q=p-1$. In this case, we choose $\beta=1-p$, and then by \eqref{eq4.18} we have
\begin{equation}\label{eq4.21}
  \int_{B_{2R}\cap\mathcal{C}}\eta^{\gamma}\mathrm{d}x+(p-1)\int_{B_{2R}\cap\mathcal{C}}u^{-p}\eta^{\gamma}{\mathbf{a}}(\nabla u)\cdot\nabla u\mathrm{d}x=\gamma\int_{B_{2R}\cap\mathcal{C}}u^{1-p}\eta^{\gamma-1}{\mathbf{a}}(\nabla u)\cdot\nabla \eta\mathrm{d}x.
\end{equation}
And also, using the Young's inequality with exponents pair $(\frac{p}{p-1},p)$ to the right side of \eqref{eq4.21}, we have
$$\int_{B_{2R}\cap\mathcal{C}}u^{1-p}\eta^{\gamma-1}{\mathbf{a}}(\nabla u)\cdot\nabla \eta\mathrm{d}x\leq \varepsilon\int_{B_{2R}\cap\mathcal{C}}u^{-p}|\nabla u|^p\eta^{\gamma}+CR^{-p}\int_{B_{2R}\cap\mathcal{C}}\eta^{\gamma-p}\mathrm{d}x.$$
Then using the fact that ${\bf{a}}(\nabla u)\cdot\nabla u=H^p(\nabla u)\geq C|\nabla u|^p$ again, and choose $\varepsilon>0$ small enough, we deduce
$$\int_{B_{2R}\cap\mathcal{C}}\eta^{\gamma}\mathrm{d}x\leq CR^{-p}\int_{B_{2R}\cap\mathcal{C}}\eta^{\gamma-p}\mathrm{d}x.$$
Therefore, we have
$$CR^N\leq CR^{N-p},$$
which leads to a contradiction as $R\rightarrow+\infty$.

Thus we complete the proof of Theorem \ref{liou2}. \qquad\qquad\qquad\quad$\qed$

\section{classification of positive solutions to the critical $p$-Laplacian equation in convex cones}

In this section, we assume that $H(\xi)=|\xi|$ and $f(u)=u^{p^*-1}$. As another application of Proposition \ref{inte ineq}, we can extend the classification result in whole space $\mathbb{R}^N$ in \cite{Ou} to general convex cones. Following the approach in the proof of Theorem 1.1 in \cite{Ou}, let $u>0$ be any positive weak solution of \eqref{eq:1.1}, consider the auxiliary function $v=u^{-\frac{p}{N-p}}$. Then by a directly calculation, $v$ is the positive weak solution to
\begin{equation}\label{eq:5.1}
      \left\{
          \begin{aligned}
          &\Delta _{p}v=\, \frac{N(p-1)}{p}\frac{|\nabla v|^p}{v} +\left(\frac{p}{N-p}\right) ^{p-1}\frac{1}{v} \quad\,\,\, &{\rm{in}} \,\, \mathcal{C}, \\ \\
          &{\bf{a}}(\nabla v)\cdot \nu =0 \quad\,\,\, &{\rm{on}} \,\, \partial\mathcal{C}
          \end{aligned}
          \right.
    \end{equation}
in the sense that
\begin{equation}\label{weak v}
  -\int_{\mathcal{C}}\left\langle {\bf{a}}(\nabla v),\nabla\varphi\right\rangle \mathrm{d}x=\int_{\mathcal{C}}g\varphi\mathrm{d}x
\end{equation}
for all $\varphi\in W^{1,p}(\Omega)\cap L^{\infty}(\Omega)$ with $\Omega\subset\mathbb{R}^N$ bounded and $\varphi=0$ on $\partial\Omega\cap\mathcal{C}$.
Denote the right-hand side of the first identity in \eqref{eq:5.1} by
$$g=\alpha v^{-1}|\nabla v|^p +\beta v^{-1},$$
where
\begin{equation}\label{ab}
  \alpha=\frac{N(p-1)}{p},\qquad\beta=\left(\frac{p}{N-p}\right)^{p-1}.
\end{equation}
Let $\mathcal{C}_{cr}:=\{x\in\mathcal{C}| \nabla u(x)=0\}$ denote the critical points of $u$ in $\mathcal{C}$, and $\mathcal{C}^c_{cr}:=\mathcal{C}\backslash\mathcal{C}_{cr}$. Firstly, by Lemma 2.1 of \cite{CFR} we know that $u\in L^{\infty}_{loc}(\overline{\mathcal{C}})$, and by the fundamental regularity results in \cite{AF,Da,Ev,Le,T} (see also \cite{ACF} for anisotropic setting) and Proposition \ref{regular}, we have
\begin{equation*}
  u\in W_{loc}^{2,2}(\mathcal{C}\backslash\mathcal{C}_{cr})\cap C^{1,\theta}_{loc}(\mathcal{C}),\qquad |\nabla u|^{p-2}\nabla u\in W_{loc}^{1,2}(\overline{\mathcal{C}})
\end{equation*}
for some $\theta\in(0,1)$, and
\begin{equation*}
  |\nabla u|^{p-2}\nabla^2 u\in L_{loc}^{2}(\mathcal{C}\backslash\mathcal{C}_{cr}).
\end{equation*}
Also, $v$ inherts same regularity properties from $u$ that
\begin{equation*}\label{eq:5.2}
  v\in W^{2,2}_{loc}(\mathcal{C})\cap C^{1,\theta}_{loc}(\mathcal{C}),\qquad|\nabla v|^{p-2}\nabla v\in W^{1,2}_{loc}(\overline{\mathcal{C}}),
\end{equation*}
for some $\theta\in(0,1)$ and
\begin{equation*}\label{eq:5.3}
 |\nabla v|^{p-2}\nabla^2v\in L^{2}_{loc}(\mathcal{C}\backslash\mathcal{C}_{cr}),\qquad v\in C^{\infty}(\mathcal{C}\backslash\mathcal{C}_{cr}).
\end{equation*}
Moreover, it follows from Lemma \ref{lower decay} and \eqref{norm0} that there exists a positive constant $C$ depending only on $N,p,H$ and  $\max\limits_{|x|=1,x\in\mathcal{C}}v(x)$, such that
\begin{equation}\label{eq5.4}
  v(x)\leq C|x|^{\frac{p}{p-1}}
\end{equation}
for all $x\in\mathcal{C}$ with $|x|>1$.

Recalling that ${\bf{a}}(\nabla v)=|\nabla v|^{p-2}\nabla v$ ($H(\xi)=|\xi|$),  we denote the $i$-th exponent of ${\bf{a}}(\nabla v)$ by ${\bf{a}}_i(\nabla v)$, that is,
$${\bf{a}}_i(\nabla v)=|\nabla v|^{p-2} v_i,$$
and define
$$E_{ij}=\partial_j({\bf{a}}_i(\nabla v))-\frac{1}{N}\partial_k({\bf{a}}_k(\nabla v))\delta_{ij},\qquad E_j=v^{-1}v_iE_{ij}.$$
We use the Einstein convention of summation over repeated indices hereafter. Obviously, the matrix $E=\{E_{ij}\}$ is trace free, i.e., $\textbf{Tr}E=E_{ii}\equiv 0$, and by \eqref{eq:5.1}, we have
$$E_{ij}=\partial_j({\bf{a}}_i(\nabla v))-\frac{1}{N}g\delta_{ij},\qquad E_j=v^{-1}v_i\partial_j({\bf{a}}_i(\nabla v))-\frac{1}{N}gv^{-1}v_j.$$

Since $v\in C^{\infty}(\mathcal{C}^c_{cr})$, so all differential identities in Lemma 2.2 in \cite{Ou} are still valid. That is, we have the following lemma.
\iffalse\begin{lem}\cite[Lemma 2.1]{Ou}\label{lem5.1}
With the notations as in above, then in $\mathcal{C}^c_{cr}$ we have

\,\,(i)\,\,\,$g_j=Nv^{-1}v_iE_{ij}=NE_j$;

\,(ii)\,\,$E_{ij,i}=\frac{N-1}{N}g_j=(N-1)E_j$;

(iii)\,$({\bf{a}}_j(\nabla v)E_{ij})_{,i}=E_{ij}E_{ji}+(N-1){\bf{a}}_j(\nabla v)E_j=\mathbf{Tr}\{E^2\}+(N-1){\bf{a}}_j(\nabla v)E_j$.
\end{lem}
\fi
\begin{lem}\cite[Lemma 2.2]{Ou}\label{lem5.2}
With the notations as in above, then in $\mathcal{C}^c_{cr}$ we have
\begin{equation}\label{eq5.6}
\big( v^{q}g^{m}{\bf{a}}_j(\nabla v)\big)_{,j}=(\alpha+q)v^{q-1}g^m|\nabla v|^p+\beta v^{q-1}g^{m}+Nmv^{q}g^{m-1}{\bf{a}}_i(\nabla v)E_i
\end{equation}
and
\begin{equation}\label{eq5.7}
\big( v^{q}g^{m}{\bf{a}}_j(\nabla v)E_{ij}\big)_{,i}= v^{q}g^m \mathbf{Tr}\{E^2\} +Nmv^{q}g^{m-1}{\bf{a}}_j(\nabla v)E_{ij}E_i +(N-1+q)v^{q}g^{m}{\bf{a}}_j(\nabla v)E_j,
\end{equation}
where $q$, $m$ are constants underdetermined.
\end{lem}

The key idea is to apply \eqref{eq5.7} with $q=1-N$ and integrate the identity in $\mathcal{C}$. Actually, based on Proposition \ref{inte ineq}, we can easily extend Proposition 2.3 of \cite{Ou}, which is the key integral inequality to obtain classification result of solutions, to the following integral version of \eqref{eq5.7}.

\begin{prop}\label{pro2.4}
Let $N\geq3$, $\mathcal{C}\subseteq\mathbb{R}^{N}$ be an open convex cone. Let $u\in W^{1,p}_{loc}(\overline{\mathcal{C}})$ be a positive weak solution of (\ref{eq:1.1}) with $f(u)=u^{p^*-1}$, and using the notations as before, then
\begin{equation}\label{eq5.8}
\int_{\mathcal{C}}\varphi v^{1-N}g^m \mathbf{Tr}\{E^2\}\mathrm{d}x+Nm\int_{\mathcal{C}}\varphi v^{1-N}g^{m-1}{\bf{a}}_j(\nabla v)E_{ij}E_i\mathrm{d}x\leq -\int_{\mathcal{C}} v^{1-N}g^{m}{\bf{a}}_j(\nabla v)E_{ij}\varphi_i \mathrm{d}x
\end{equation}
holds for every $0\leq \varphi\in W^{1,p}(\Omega)$ with $\Omega\subset\mathbb{R}^N$ bounded and $\varphi=0$ on $\partial\Omega\cap\mathcal{C}$.
\end{prop}

\begin{proof}
Let $a=-\frac{N(p-1)}{N-p}$, $b=\frac{p(N-1)}{N-p}$, then one can easily check that $\psi(x)=0$ in Proposition \ref{inte ineq}. Then we have
\begin{equation}\label{m0}
 - \int_{\mathcal{C}}\bm{\omega}(x)\cdot\nabla\varphi(x)\mathrm{d}x\geq\int_{\mathcal{C}}u^{\frac{p(N-1)}{N-p}}I(x)\varphi(x)\mathrm{d}x
\end{equation}
holds for any $0\leq\varphi\in W^{1,p}(\Omega)$ with $\Omega\subset\mathbb{R}^{N}$ bounded and $\varphi=0$ on $\partial\Omega\cap\mathcal{C}$. Recall that $v=u^{-\frac{p}{N-p}}$, then by \eqref{m0} and a simple calculation, we obtain  for $m=0$, there holds
\begin{equation}\label{m01}
 \int_{\mathcal{C}}\varphi v^{1-N}\mathbf{Tr}\{E^2\} \mathrm{d}x\leq -\int_{\mathcal{C}} v^{1-N}{\bf{a}}_j(\nabla v)E_{ij}\varphi_i\mathrm{d}x .
\end{equation}
For $m\neq 0$ we argue by approximation. So, we first extend $v$ as $0$ outside $\mathcal{C}$. For $\varepsilon>0$, we define $v^{\varepsilon}=v\ast \rho^{\varepsilon}$, where $\rho^{\varepsilon}$ is a standard mollifier. Also, we set
 $$g^{\varepsilon}=\alpha(v^{\varepsilon})^{-1}|\nabla v^{\varepsilon}|^p +\beta(v^{\varepsilon})^{-1}.$$
Now replacing $\varphi$ with $(g^{\varepsilon})^m\varphi$ in \eqref{m01}, we have
\begin{equation}\label{eq5.11}
 \int_{\mathcal{C}}(g^{\varepsilon})^m\varphi v^{1-N}\mathbf{Tr}\{E^2\}\mathrm{d}x
 \leq -\int_{\mathcal{C}} v^{1-N}{\bf{a}}_j(\nabla v)E_{ij}\big[m(g^{\varepsilon})^{m-1}(g^{\varepsilon})_i\varphi+(g^{\varepsilon})^m\varphi_i\big]\mathrm{d}x.
\end{equation}

Noticing that $E_{ij}\in L^{2}_{loc}(\overline{\mathcal{C}})$ and $g\in W^{1,2}_{loc}(\overline{\mathcal{C}})$,
we have $g^{\varepsilon}\rightarrow g$ and $(g^{\varepsilon})_i\rightarrow g_i=Nv^{-1}v_jE_{ji}=NE_i$ in $L^{2}_{loc}(\overline{\mathcal{C}})$ as  $\varepsilon\rightarrow 0$.
 Then we obtain \eqref{eq5.8} from \eqref{eq5.11} easily by letting $\varepsilon\rightarrow 0$.
\end{proof}

By choosing $v^{1-q}\varphi$ as a test function in \eqref{weak v}, we also have following is the integral identity, which is the integral counterpart of \eqref{eq5.6} with $m=0$ and will be used to deal with the ``error" term on the right-hand side of \eqref{eq5.8}.

\begin{lem}\label{lem5}
Let $N\geq2$, $\mathcal{C}\subseteq\mathbb{R}^{N}$ be an open convex cone, and $\alpha$, $\beta$ be as in \eqref{ab}. Then
\begin{equation}\label{eq:5.12}
(\alpha+1-q)\int_{\mathcal{C}}v^{-q}|\nabla v|^p\varphi\mathrm{d}x +\beta\int_{\mathcal{C}} v^{-q}\varphi\mathrm{d}x = -\int_{\mathcal{C}} v^{1-q} {\bf{a}}_j(\nabla v)\varphi_j \mathrm{d}x
\end{equation}
for all $\varphi\in W^{1,p}(\Omega)\cap L^{\infty}(\Omega)$ with $\Omega\subset\mathbb{R}^N$ bounded and $\varphi=0$ on $\partial\Omega\cap\mathcal{C}$.
\end{lem}

By \eqref{eq5.4}, \eqref{eq:5.12} and restricting all integrals in the proof of Lemma 2.1 of \cite{Ve1} in convex cone $\mathcal{C}$, one can easily to check that the integral estimate (2.3) of \cite{Ve1} is still valid by replacing the integration region $B_R$ by $B_R\cap\mathcal{C}$.
\begin{lem}\label{lem5.5}
Let $N\geq3$, $\mathcal{C}\subseteq\mathbb{R}^{N}$ be an open convex cone, and $\alpha$, $\beta$ be as in \eqref{ab}. Let $1<p<N$, $r\in[0,p]$, $q<\alpha+1$ and $R>1$. Then there exists a constant $C=C(N,p,q,r)>0$ such that, for $r \leq q<\alpha+1$, there holds
\begin{equation*}\label{eq5.13}
\int_{B_R\cap\mathcal{C}}v^{-q }|\nabla v|^r\mathrm{d}x\leq CR^{N-q },
\end{equation*}
and for $q<r$, there holds
\begin{equation*}\label{eq5.14}
\int_{B_R\cap\mathcal{C}}v^{-q }|\nabla v|^r\mathrm{d}x \leq CR^{N-\frac{pq-r}{p-1} }.
\end{equation*}
\end{lem}

\medskip
 \noindent {\bf Proof of Theorem \ref{cla}.} Let $\eta$ be smooth cut-off functions satisfying:
\begin{equation}\label{eq5.15}
\begin{cases}
                          \eta\equiv 1  &\text{in} \,\,B_R,\\
                        0\leq\eta\leq1  &\text{in} \,\,B_{2R},\\
                          \eta\equiv 0  &\text{in} \,\,\mathbb{R}^N\backslash B_{2R},\\
     |\nabla \eta|\leq \frac{C}{R}  &\text{in} \,\,\mathbb{R}^N,
\end{cases}
\end{equation}
and take a constant $\theta>0$ big enough.
Replacing  $m$ with  $-m$ and $\varphi$ with $\eta^\theta$ in \eqref{eq5.8}, we have

\begin{equation}\label{eq5.16}
\begin{split}
\int_{\mathcal{C}}\eta^\theta v^{1-N}g^{-m} \mathbf{Tr}\{E^2\}\mathrm{d}x &-Nm\int_{\mathcal{C}}\varphi v^{1-N}g^{-m-1}{\bf{a}}_j(\nabla v)E_{ij}E_i\mathrm{d}x\\
&\leq -\theta\int_{\mathcal{C}} \eta^{\theta-1} v^{1-N}g^{-m}{\bf{a}}_j(\nabla v)E_{ij}\eta_i \mathrm{d}x.
\end{split}
\end{equation}
Choose $m=\frac{p-1}{p}-\varepsilon_0>0$, with $\varepsilon_0>0$ small enough and depending only on $N,p$.
Recalling that
 $$g=\alpha v^{-1}|\nabla v|^p +\beta v^{-1},$$
 then we have
\begin{equation*}\label{eq5.17}
\begin{split}
\int_{\mathcal{C}}\eta^\theta v^{1-N}g^{-m} \mathbf{Tr}\{E^2\}\mathrm{d}x
        = \beta\int_{\mathcal{C}}\eta^\theta v^{-N}g^{-m-1}\mathbf{Tr}\{E^2\}\mathrm{d}x+\alpha\int_{\mathcal{C}}\eta^\theta v^{-N}g^{-m-1}|\nabla v|^p \mathbf{Tr}\{E^2\}\mathrm{d}x.
\end{split}
\end{equation*}
Therefore, the left-hand side of \eqref{eq5.16} can be rewritten as
\begin{equation}\label{eq5.18}
\begin{split}
\, &\int_{\mathcal{C}}\eta^\theta v^{1-N}g^{-m} \mathbf{Tr}\{E^2\}\mathrm{d}x-Nm\int_{\mathcal{C}}\varphi v^{1-N}g^{-m-1}{\bf{a}}_j(\nabla v)E_{ij}E_i\mathrm{d}x\\
 = & \beta\int_{\mathcal{C}}\eta^\theta v^{-N}g^{-m-1}\mathbf{Tr}\{E^2\}\mathrm{d}x+(\alpha-Nm)\int_{\mathcal{C}}\eta^\theta v^{-N}g^{-m-1}|\nabla v|^p \mathbf{Tr}\{E^2\}\mathrm{d}x\\
 \,&\quad +Nm\int_{\mathcal{C}}\eta^\theta v^{-N}g^{-m-1}|\nabla v|^p \left(\mathbf{Tr}\{E^2\}-\frac{v_{j}}{|\nabla v|}E_{ij}\cdot\frac{v_{k}}{|\nabla v|}E_{ki}\right)\mathrm{d}x .
\end{split}
\end{equation}
It follows from Lemma 2.7 in \cite{Ou} that the last integral in \eqref{eq5.18} is nonnegative. Thus by taking $m=\frac{p-1}{p}-\varepsilon_0$ with $0<\varepsilon_0<\frac{p-1}{p}$, we get
\begin{equation}\label{eq5.19}
\begin{split}
   \,&\int_{\mathcal{C}}\eta^\theta v^{1-N}g^{-m} \mathbf{Tr}\{E^2\}\mathrm{d}x-Nm\int_{\mathcal{C}}\varphi v^{1-N}g^{-m-1}{\bf{a}}_j(\nabla v)E_{ij}E_i\mathrm{d}x\\
\geq & \beta\int_{\mathcal{C}}\eta^\theta v^{-N}g^{-m-1}\mathbf{Tr}\{E^2\}\mathrm{d}x+N\varepsilon_0\int_{\mathcal{C}}\eta^\theta v^{-N}g^{-m-1}|\nabla v|^p \mathbf{Tr}\{E^2\}\mathrm{d}x\\
  =  & \frac{p}{p-1}\varepsilon_0\Big[\frac{p-1}{p\varepsilon_0}\beta\int_{\mathcal{C}}\eta^\theta v^{-N}g^{-m-1}\mathbf{Tr}\{E^2\}\mathrm{d}x+\alpha\int_{\mathcal{C}}\eta^\theta v^{-N}g^{-m-1}|\nabla v|^p \mathbf{Tr}\{E^2\}\Big]\mathrm{d}x\\
\geq & \frac{p}{p-1}\varepsilon_0\int_{\mathcal{C}}\eta^\theta v^{1-N}g^{-m} \mathbf{Tr}\{E^2\}\mathrm{d}x.
\end{split}
\end{equation}
Combining \eqref{eq5.19} with \eqref{eq5.16}, we derive

\begin{equation}\label{eq5.20}
 \frac{p}{p-1}\varepsilon_0\int_{\mathcal{C}}\eta^\theta v^{1-N}g^{-m} \mathbf{Tr}\{E^2\}\mathrm{d}x\leq-\theta\int_{\mathcal{C}} \eta^{\theta-1} v^{1-N}g^{-m}{\bf{a}}_j(\nabla v)E_{ij}\eta_i\mathrm{d}x .
\end{equation}
From Corollary 2.6 in \cite{Ou}, we get
\begin{equation}\label{eq5.21}
-\theta\eta^{-1}\eta_i{\bf{a}}_j(\nabla v)E_{ij} \leq \frac{C\theta^2}{\varepsilon}\eta^{-2}\eta_i\eta_j{\bf{a}}_i(\nabla v){\bf{a}}_j(\nabla v)+ \varepsilon \mathbf{Tr}\{E^2\},
\end{equation}
where $\varepsilon>0$ small enough. By \eqref{eq5.20} and \eqref{eq5.21}, we conclude that
\begin{equation}\label{eq5.22}
\left(\frac{p}{p-1}\varepsilon_0-\varepsilon\right)\int_{\mathcal{C}}\eta^\theta v^{1-N}g^{-m} \mathbf{Tr}\{E^2\}\mathrm{d}x
\leq \frac{C\theta^2}{\varepsilon}\int_{\mathcal{C}} \eta^{\theta-2} v^{1-N}g^{-m}\eta_i\eta_j{\bf{a}}_i(\nabla v){\bf{a}}_j(\nabla v)\mathrm{d}x .
\end{equation}
Since $|\eta_i\eta_j|\leq \frac{C}{R^2}$ and $|{\bf{a}}_i(\nabla v){\bf{a}}_j(\nabla v)|\leq |\nabla v|^{2p-2}$,
by letting $\varepsilon= \frac{p}{2(p-1)}\varepsilon_0$  we deduce from \eqref{eq5.22} that
\begin{equation}\label{eq5.23}
\int_{\mathcal{C}}\eta^\theta v^{1-N}g^{-m} \mathbf{Tr}\{E^2\}\mathrm{d}x\leq \frac{C}{R^2}\int_{\mathcal{C}} \eta^{\theta-2} v^{1-N}g^{-m}|\nabla v|^{2p-2}\mathrm{d}x.
\end{equation}
Moreover, since
$$|\nabla v|\leq \left(\frac{pvg}{N(p-1)}\right) ^{\frac{1}{p}},$$
choosing $\theta>2$ large enough, and combining \eqref{eq5.15} with \eqref{eq5.23}, we arrive at
\begin{equation*}
  \begin{aligned}
   \int_{\mathcal{C}}\eta^\theta v^{1-N}g^{-m} \mathbf{Tr}\{E^2\}\mathrm{d}x&\leq \frac{C}{R^2}\int_{\mathcal{C}} \eta^{\theta-2} v^{1-N}g^{-m}(vg)^{\frac{2p-2}{p}}\mathrm{d}x \\
   &\leq \frac{C}{R^2}\int_{B_{2R}\cap\mathcal{C}}  v^{-\frac{Np-3p+2}{p}}g^{\frac{p-1}{p}+\varepsilon}\mathrm{d}x.
  \end{aligned}
\end{equation*}
Choosing $\varepsilon_0>0$ small enough, such that $0<\frac{p-1}{p}+\varepsilon<1$, then by the definition of $g$, we have
\begin{equation*}
\int_{B_{2R}\cap\mathcal{C}}  v^{-\frac{Np-3p+2}{p}}g^{\frac{p-1}{p}+\varepsilon}\mathrm{d}x\leq C\int_{B_{2R}\cap\mathcal{C}}  \left(v^{-(\frac{Np-3p+2}{p}+\frac{p-1}{p}+\varepsilon)}|\nabla v|^{p-1+p\varepsilon}+ v^{-(\frac{Np-3p+2}{p}+\frac{p-1}{p}+\varepsilon)}\right) \mathrm{d}x.
\end{equation*}
Next, since $\frac{N+1}{3}<p<N$, by choosing $\varepsilon_0>0$ small enough, one has
$$\frac{Np-3p+2}{p}+\frac{p-1}{p}+\varepsilon<\frac{Np-N+p}{p}=\alpha+1,$$
and
$$N-\min\left\{\frac{Np-3p+2}{p}+\frac{p-1}{p}+\varepsilon,\frac{p(\frac{Np-3p+2}{p}+\frac{p-1}{p}+\varepsilon)-(p-1+p\varepsilon)}{p-1}\right\} <2.$$
Then it follows from Lemma \ref{lem5.5} that
$$\int_{B_{2R}\cap\mathcal{C}}  v^{-\frac{Np-3p+2}{p}}g^{\frac{p-1}{p}+\varepsilon}\mathrm{d}x=o(R^2)\qquad\text{as}\quad R\rightarrow+\infty, $$
this gives that
\begin{equation}\label{eq5.34}
\int_{\mathcal{C}}  v^{1-N}g^{-m} \mathbf{Tr}\{E^2\}\leq 0.
\end{equation}
Moreover, by the proof of Corollary 2.6 in \cite{Ou}, we know that $\mathbf{Tr}\{E^2\}\geq 0$, and the equality holds if and only if $E=0$. So \eqref{eq5.34} implies that $E=0$ $a.e$ in $\mathcal{C}$, especially $E\equiv 0$ in $\mathcal{C}^{c}_{cr}$. Thus
$v=C_1+C_2|x-x_0|^{\frac{p}{p-1}}$
for some $C_1,C_2>0$ and $x_0\in \overline{\mathcal{C}}$, then $u(x)=U_{\lambda,x_0}$ for some constant $\lambda>0$. Finally, by the Neumann boundary condition ${\bf{a}}(\nabla u)\cdot\nu=0$ on $\partial\mathcal{C}$, we deduce $x_0\in\mathcal{V}$, and hence \\
$(\rm{i})$ if $k=N$, then $\mathcal{C}=\mathbb{R}^N$ and $x_0$ may be a generic point in $\mathbb{R}^N$;\\
  $(\rm{ii})$ if  $k\in\{1,\cdots ,N-1\}$, then $x_0\in\mathbb{R}^k\times\{0_{\mathbb{R}^{N-k}}\}$;\\
  $(\rm{iii})$ if $k=0$, then $x_0=0$.

This concludes our proof of Theorem \ref{cla}. \qed

\end{document}